\documentclass[reqno]{amsart}
\usepackage{amsmath,amssymb,amsthm,graphicx,a4wide,enumerate}
\usepackage[small,bf]{caption}
\theoremstyle{plain}
\newtheorem{thm}{Theorem}
\newtheorem{lemma}[thm]{Lemma}
\theoremstyle{definition}
\newtheorem{defn}{Definition}
\theoremstyle{remark}
\newtheorem{rem}{Remark}
\newtheorem{exmp}{Example}

\newcommand{\prn}[1]{\left(#1\right)}
\newcommand{\abs}[1]{\left|#1\right|}
\newcommand{\norm}[1]{\left\|#1\right\|}
\newcommand{\pd}[2]{\frac{\partial#1}{\partial#2}}

\newcommand{\ud}[1]{\, \mathrm{d}#1}

\begin{document}
\parskip.9ex

\title[Gradient-augmented level set method with optimally local scheme]
{A gradient-augmented level set method with an optimally local, coherent advection scheme}
\author{Jean-Christophe Nave}
\address[Jean-Christophe Nave]
{Department of Mathematics \\ Massachusetts Institute of Technology \\
77 Massachusetts Avenue \\ Cambridge, MA 02139}
\email{jcnave@mit.edu}
\urladdr{http://www.mit.edu/\~{}jcnave}
\author{Rodolfo Ruben Rosales}
\address[Rodolfo Ruben Rosales]
{Department of Mathematics \\ Massachusetts Institute of Technology \\
77 Massachusetts Avenue \\ Cambridge, MA 02139}
\email{rrr@math.mit.edu}
\author{Benjamin Seibold}
\address[Benjamin Seibold]
{Department of Mathematics \\ Temple University \\
1801 North Broad Street \\ Philadelphia, PA 19122}
\email{seibold@temple.edu}
\urladdr{http://www.math.temple.edu/\~{}seibold/}
\begin{abstract}
The level set approach represents surfaces implicitly, and advects them by evolving a
level set function, which is numerically defined on an Eulerian grid. Here we present
an approach that augments the level set function values by gradient information, and
evolves both quantities in a fully coupled fashion. This maintains the coherence between
function values and derivatives, while exploiting the extra information carried by the
derivatives. The method is of comparable quality to WENO schemes, but with optimally
local stencils (performing updates in time by using information from only a single
adjacent grid cell). In addition, structures smaller than the grid size can be located
and tracked, and the extra derivative information can be employed to obtain simple and
accurate approximations to the curvature. We analyze the accuracy and the stability
of the new scheme, and perform benchmark tests.
\end{abstract}

\maketitle

\section{Introduction}
Level set methods represent a surface as the zero contour of a level set
function \cite{OsherSethian1988}, which can be numerically defined on a regular Eulerian
grid. The advection of the surface translates then into an appropriate
advection of the level set function. Derivative quantities, such as normal vectors and
curvature, can be computed from the level set function without
explicitly representing the surface. Commonly used level set approaches encounter
problems or inconveniences in the following aspects: the representation of small
structures, the approximation of derivative quantities (such as curvature), and the large
size of the stencils required by high order finite difference schemes.

We investigate the extent to which these problems can be remedied, if the level set
function is augmented by gradient information. In this case, the surface can be
represented by an appropriate interpolation that incorporates the additional information.
In this paper, we consider a bi-/tri-cubic Hermite interpolation to define the surface
in each cell. This yields a certain level of ``subgrid'' resolution, i.e.~structures
smaller than the grid resolution can be represented. In addition, the Hermite
bi-/tri-cubic interpolant provides a simple and accurate approximation to the normals
and the curvature anywhere in the computational domain.
In the context of level set methods, the use of a bicubic interpolation to construct
second order approximations to interfaces was proposed by Chopp \cite{Chopp2001}, but
without any tracking of gradient information.

The idea of using gradient information to improve the accuracy of numerical methods
for hyperbolic conservation laws was introduced by
van Leer \cite{VanLeer1973,VanLeer1974,VanLeer1977_1,VanLeer1977_2,VanLeer1979}.
In particular, his MUSCL (``Monotonic Upstream-centered Scheme for Conservation Laws'')
scheme and the PPM (``Piecewise Parabolic Method'')
by Colella and Woodward \cite{ColellaWoodward1984} use gradient information.
While in those methods, gradient values are reconstructed from function values,
the CIP method of Takewaki, Nishiguchi, and Yabe
\cite{TakewakiNishiguchiYabe1985,TakewakiYabe1987}
stores gradients as an independent quantity to solve hyperbolic conservation laws.
In the context of level set methods, Raessi, Mostaghimi, and Bussmann present an
approach that advects normal vectors independent of, but in an analogous fashion to,
the level set function values \cite{RaessiMostaghimiBussmann2007}.

In this paper, we present an approach that advects function values and gradients
as independent quantities, but in a fully coupled fashion. The approach is based on a
generalization of the CIR method \cite{CourantIsaacsonRees1952}, and uses the Hermite
cubic interpolant mentioned above in a natural way. Characteristic curves are
tracked backwards, and function values and gradients are obtained from the
interpolation. The resulting advection scheme is globally third order accurate, with
stencils that can be chosen no wider than a single cell.

In Sect.~\ref{sec:overview}, we provide a brief overview of the classical level set
method, introduce the idea of gradient-augmented approaches, and present a Hermite
cubic interpolation that will be the basis for the new method.
Three fundamental problems with classical level set methods are outlined in
Sect.~\ref{sec:benefits_gradient}, with focus on how the incorporation of gradients
is beneficial. The precise numerical scheme is given in
Sect.~\ref{sec:scheme_CIR}. Its accuracy and stability are analyzed theoretically
in Sect.~\ref{sec:numerical_analysis}, for a simple case.
In Sect.~\ref{sec:numerical_results}, we numerically investigate the accuracy of
the new approach, and test its performance on various benchmark tests.
Finally, in Sect.~\ref{sec:conclusions_outlook}, we outline various questions to be
investigated in future work.

\section{Level Set Approaches Without and With Gradients}
\label{sec:overview}

\subsection{Classical Level Set Method}
\label{subsec:classical_LS}
Many applications, such as the simulation of two-phase flows in $\mathbb{R}^p$ require
the representation and advection of a manifold of codimension 1, which in the following
we call a \emph{surface}. Level set methods \cite{OsherSethian1988} represent the surface
as the zero contour of a level set function $\phi:\mathbb{R}^p\rightarrow\mathbb{R}$.
In the domain enclosed by the surface, one has $\phi<0$, while outside $\phi>0$.
Geometric quantities, such as normal vectors and curvature, can be obtained from the
level set function:
$\vec{n}=\frac{\nabla\phi}{|\nabla\phi|}$, and $\kappa=\nabla\cdot\vec{n}$. In order to
move the surface with a velocity field $\vec{v} = (u,v,w)^T$, the level set function is
advected according to the partial differential equation
\begin{equation}
\phi_t+\vec{v}\cdot\nabla\phi = 0\;.
\label{eq:advection_phi}
\end{equation}
The level set function can be defined on a regular grid. High order
ENO~\cite{ShuOsher1988} or WENO~\cite{LiuOsherChan1994} schemes are commonly used to
approximate the advection equation \eqref{eq:advection_phi}.

Gradients and curvatures can be approximated by finite differences. For an accurate and
stable approximation, it is beneficial if $\phi$ is a signed distance function
\begin{equation}
|\nabla\phi(x)| = 1\;.
\label{eq:distance_function}
\end{equation}
Even if $\phi$ is a distance function initially, it typically ceases to be so due to
deformations induced by the velocity field. One remedy to this problem is to recover
\eqref{eq:distance_function} by solving the reinitialization equation
\begin{equation}
\phi_\tau=\mathrm{sign}(\phi_0)(1-|\nabla\phi|)\;,
\label{eq:reinitialization}
\end{equation}
where $\phi_0$ is the level set function at time $t$,
in pseudo-time $\tau$ \cite{SussmanSmerekaOsher1994}, or by solving the stationary
Eikonal equation \eqref{eq:distance_function} using a
fast marching method \cite{Sethian1996}. Another approach is
to solve \eqref{eq:advection_phi} with a modified velocity field, so that
\eqref{eq:distance_function} is preserved. This modified field is constructed by
extending the original velocity field away from the surface \cite{AdalsteinssonSethian1999}.

The actual surface is obtained from the level set function $\phi$ using contouring
algorithms \cite{LorensenCline1987}. These approaches are typically based on a
bi-/tri-linear interpolation inside a grid cell. The linear interpolant along cell edges
locates intersections of the surface with the grid edges. These intersection points are
subsequently connected to form surface patches in each cell. Ambiguous connection
cases are decided based on the full bi-/tri-linear interpolant inside the cell.
In many applications, it is sufficient to know the location of the surface on the grid
edges, for instance in the ghost fluid method \cite{FedkiwLiu2002}.

\subsection{Gradient-Augmented Level Set Method}
\label{subsec:gradient-augmented_LS}
We consider a generalized level set approach. The level set function $\phi$ is augmented
by gradient information $\vec{\psi}=\nabla\phi$, which is defined on the same grid as
$\phi$. The surface is defined using both independent quantities $\phi$ and $\vec{\psi}$.
This approach has various advantages in the context of level set methods.
Raessi et al.~show that under some circumstances, curvature can be
computed more accurately if gradients are accessible \cite{RaessiMostaghimiBussmann2007}.
In addition, as we show in the following, with gradients, subgrid structures can be
represented, and a high order advection scheme with optimally local stencils can be
formulated.

Evolving the implicitly defined surface with a velocity field $\vec{v}$ translates to
equation \eqref{eq:advection_phi} for $\phi$, and equation
\begin{equation}
\vec{\psi}_t+\nabla\prn{\vec{v}\cdot\vec{\psi}} = 0
\label{eq:advection_psi}
\end{equation}
for $\vec{\psi}$. Equation \eqref{eq:advection_psi} is obtained by applying the gradient
operator to \eqref{eq:advection_phi}. A straightforward approach is to
approximate each equation \eqref{eq:advection_phi} and \eqref{eq:advection_psi} using a
high order finite difference scheme. Raessi et al.~apply this
approach \cite{RaessiMostaghimiBussmann2007}, introducing only a weak coupling through an
extension velocity field \cite{AdalsteinssonSethian1999}. While such decoupled (or weakly
coupled) approaches can improve the classical level set approach, the full potential of a
coupled approach is not used. In contrast, here we present an approach that evolves the
level set function and its gradients in a coherent and fully coupled fashion. The
precise methodology is presented in Sect.~\ref{sec:scheme_CIR}.

\subsection{Cell-Based Hermite Interpolant}
\label{subsec:Hermite_interpolant}
In the gradient-augmented level set method, every grid point carries a function value
and a gradient vector. Thus, in $p$ space dimensions, every grid cell has $p+1$ pieces
of information on each of the $2^p$ cell corner points. As we show, this allows the
definition of a cell-based Hermite interpolant, i.e.~a function $\phi(\vec{x})$ that
is $C^\infty$ inside each cell, and that matches the function values and gradients at
all cell corner points. In this paper, we consider a $p$-cubic Hermite interpolant,
i.e.~a cubic in 1D, a bi-cubic in 2D, a tri-cubic in 3D, etc. It is a natural
generalization of the bi-/tri-linear interpolation used in classical level set approaches.
The interpolant is simple to construct, using a tensor product approach.

In the following, we use the classical multi-index notation. For
vectors $\vec{x}\in\mathbb{R}^p$ and $\vec{a}\in\prn{\mathbb{N}_0}^p$, one defines
$\abs{\vec{a}} = \sum_{i=1}^p a_i$, and
$\vec{x}^{\,\vec{a}} = \prod_{i=1}^p x_i^{a_i}$, and
$\partial^{\,\vec{a}} = \partial_1^{a_1}\dots\partial_p^{a_p}$,
where $\partial_i^{a_i} = \frac{\partial^{a_i}}{\partial^{a_i} x_i}$.
For convenience, we formulate some results for cubes, for which $h$ denotes the edge
length $h = \Delta x = \Delta y = \Delta z$. However, the results apply to rectangular
cells of arbitrary edge lengths as well. In this case, the estimates are valid with
respect to the scaling parameter $h = \max\{\Delta x,\Delta y,\Delta z\}$.

\begin{defn}
A \emph{$p$-cubic polynomial} is a polynomial $\mathcal{H} = \mathcal{H}(\vec{x})$
in $\mathbb{R}^p$, of degree $\leq 3$ in each of the variables. Hence, it has an
expression of the form
\begin{equation*}
\mathcal{H}(\vec{x}) = \sum_{\vec{\alpha}\in\{0,1,2,3\}^p}
c_{\vec{\alpha}} \; \vec{x}^{\,\vec{\alpha}}\;,
\end{equation*}
involving $4^p$ parameters $c_{\vec{\alpha}}$.
\end{defn}

\begin{defn}
A \emph{$p$-rectangle} (or simply ``cell'') is a set
$[a_1,b_1]\times\dots\times [a_p,b_p] \subset \mathbb{R}^p$, where $a_i<b_i$.
If $a_i=0,\,b_i=h\,\forall\, i$, we speak of a \emph{$p$-cube} (of size $h$), denoted
by $\mathcal{C}_h$. The $p$-cube $\mathcal{C}_1$ is called \emph{unit $p$-cube}.
\end{defn}

Let the $2^p$ vertices in a cell be indexed by a vector $\vec{v}\in\{0,1\}^p$,
i.e.~the vertex of index $\vec{v}$ is at position
$\vec{x}_{\vec{v}} = \prn{a_1+(b_1-a_1)v_1,\dots,a_p+(b_p-a_p)v_p}$.
In particular, for $\mathcal{C}_h$, we have $\vec{x}_{\vec{v}} = h\vec{v}$.

\begin{defn}
\label{def:p-cubic_data}
Let $\mathcal{C}$ be a $p$-rectangle, and let $\phi$ be a sufficiently smooth function
defined on an open set in $\mathbb{R}^p$ that includes $\mathcal{C}$. The
\emph{data for $\phi$ on $\mathcal{C}$} is the set of $4^p$ scalars given by
\begin{equation*}
\phi_{\vec{\alpha}}^{\vec{v}} = \partial^{\,\vec{\alpha}}\phi(\vec{x}_{\vec{v}})\;,
\end{equation*}
where both $\vec{v},\vec{\alpha}\in\{0,1\}^p$.
\end{defn}

\begin{lemma}
\label{lem:p-cubic_uniqueness}
Two $p$-cubic polynomials with the same data on some $p$-rectangle, must be equal.
\end{lemma}
\begin{proof}
WLOG assume that the $p$-rectangle is the unit $p$-cube $\mathcal{C}_1$.
Let $\mathcal{H}$ be the difference between the two polynomials---with zero data
on $\mathcal{C}_1$. We show now that $\mathcal{H} \equiv 0$.
\begin{enumerate}
\item
If $p = 1$, $\mathcal{H}$ is a cubic polynomial in one variable, with double zeros
at $x = 0,1$. Hence $\mathcal{H} \equiv 0$.
\item
If $p = 2$, the $p = 1$ result yields
$\mathcal{H}(x_1,0) \equiv \mathcal{H}_{x_2}(x_1,0) \equiv
\mathcal{H}(x_1,1) \equiv \mathcal{H}_{x_2}(x_1,1)\equiv 0 \;\forall\, 0 \leq x_1 \leq 1$.
Hence, the same argument as in $p = 1$ yields:
$\mathcal{H}(x_1,x_2) \equiv 0 \;\forall\, 0 \leq x_2 \leq 1$.
\item
If $p = 3$, the $p = 2$ result yields
$\mathcal{H}(x_1,x_2,0) \equiv \mathcal{H}_{x_3}(x_1,x_2,0) \equiv
\mathcal{H}(x_1,x_2,1) \equiv \mathcal{H}_{x_3}(x_1,x_2,1)\equiv 0
\;\forall\, 0 \leq x_1,x_2 \leq 1$.
As before, it follows that from the $p = 1$ result that
$\mathcal{H}(x_1,x_2,x_3) \equiv 0 \;\forall\, 0 \leq x_3 \leq 1$.
\end{enumerate}
From the above, it should now be obvious how to complete the proof
using induction on $p$.
\end{proof}

\begin{thm}
For any arbitrary data set on some $p$-rectangle, there exists exactly one
$p$-cubic polynomial which corresponds to the data.
\end{thm}
\begin{proof}
Lemma~\ref{lem:p-cubic_uniqueness} shows that there exists at most one such polynomial.
Here we construct one.
WLOG assume that the $p$-rectangle is the unit $p$-cube $\mathcal{C}_1$.
For each vertex of $\mathcal{C}_1$, indexed by $\vec{v}$, and each derivative
$\vec{\alpha}\in\{0,1\}^p$, one can construct Lagrange basis polynomials
$W_{\vec{\alpha}}^{\vec{v}}$, i.e.~$p$-cubic polynomials that satisfy
\begin{equation*}
\partial^{\,\vec{\beta}}\phi(\vec{x}_{\vec{w}}) W_{\vec{\alpha}}^{\vec{v}}
= \delta_{\vec{\alpha}\vec{\beta}}\delta_{\vec{v}\vec{w}}
\quad\forall\, \vec{\beta},\vec{w}\in\{0,1\}^p\;.
\end{equation*}
Here $\delta_{\vec{v}\vec{w}} = \prod_{i=1}^p \delta_{v_i,w_i}$,
where $\delta$ is Kronecker's delta. Hence, each of the $2^p\cdot 2^p = 4^p$ basis
polynomials equals 1 on exactly one vertex and for exactly one type of derivative,
and equals 0 for any other vertex or derivative.

The basis polynomials can be constructed as tensor products of the form
\begin{equation*}
W_{\vec{\alpha}}^{\vec{v}} (\vec{x}) = \prod_{i=1}^p w_{\alpha_i}^{v_i}(x_i)\;,
\end{equation*}
where each of the $w_i$ is a $1$-cubic polynomial
\begin{equation}
w_{\alpha}^{v}(x) =
\begin{cases}
 f(x)   &\text{if~} v=0, \alpha=0 \\
 f(1-x) &\text{if~} v=1, \alpha=0 \\
 g(x)   &\text{if~} v=0, \alpha=1 \\
-g(1-x) &\text{if~} v=1, \alpha=1
\end{cases}
\label{eq:1-cubic_basis_functions}
\end{equation}
where $f(x) = 1-3x^2+2x^3$ and $g(x) = x(1-x)^2$.

A $p$-cubic that corresponds to the data, is then defined by a linear combination
of the basis functions
\begin{equation}
\mathcal{H}(\vec{x}) = \sum_{\vec{v},\vec{\alpha}\in\{0,1\}^p}
\phi_{\vec{\alpha}}^{\vec{v}}\, W_{\vec{\alpha}}^{\vec{v}}(\vec{x})\;.
\label{eq:p-cubic_constructed}
\end{equation}
\end{proof}

We now investigate how well the $p$-cubic \eqref{eq:p-cubic_constructed} approximates
a smooth function, when interpolating it on the vertices of a cell. We are interested
in its accuracy, as the cell size approaches zero. While the following analysis also
holds for $p$-rectangles, for simplicity (and WLOG), we provide the expressions
for $p$-cubes only.

\begin{defn}
\label{def:p-cubic_interpolant}
Consider a (sufficiently smooth) scalar function $\phi(\vec{x})$, defined in an open
neighborhood $\Omega$ of the origin in $\mathbb{R}^p$. Let $0 < h \ll 1$ be small
enough so that the $p$-cube $\mathcal{C}_h$ is included in $\Omega$.
The \emph{$p$-cubic Hermite interpolant to $\phi$ in $\mathcal{C}_h$}
is the $p$-cubic polynomial $\mathcal{H} = \mathcal{H}(\vec{x})$, such that $\phi$
and $\mathcal{H}$ have the same data on $\mathcal{C}_h$.
\end{defn}

Using the notation of Def.~\ref{def:p-cubic_data}, and equation
\eqref{eq:p-cubic_constructed}, it is easy to see that the $p$-cubic Hermite
interpolant in Def.~\ref{def:p-cubic_interpolant} can be written in the form
\begin{equation}
\mathcal{H}(\vec{x}) = \sum_{\vec{v},\vec{\alpha}\in\{0,1\}^p}
\phi_{\vec{\alpha}}^{\vec{v}}\, h^{\abs{\vec{\alpha}}}\,
W_{\vec{\alpha}}^{\vec{v}}\prn{\frac{\vec{x}}{h}}\;.
\label{eq:p-cubic_interpolant}
\end{equation}
This expression is straightforward to differentiate analytically:
derivatives of the $1$-cubic basis functions \eqref{eq:1-cubic_basis_functions},
and powers of $1/h$, appear.

\begin{exmp}[1D Hermite cubic interpolant]
On a 1D cell $[x_i,x_{i+1}]$ of size $h$, with grid values denoted by $\phi^i$,
and derivatives by $\phi_x^i$, the Hermite cubic interpolant is defined by
\begin{equation}
\mathcal{H}(x) = \phi^i f(\xi)+\phi^{i+1} f(1-\xi)
+h\prn{\phi_x^i g(\xi)-\phi_x^{i+1} g(1-\xi)}\;,
\label{eq:Hermite_cubic_1D}
\end{equation}
and its derivative is
\begin{equation}
\mathcal{H}'(x) = \tfrac{1}{h}\prn{\phi^i f'(\xi)-\phi^{i+1} f'(1-\xi)}
+\phi_x^i g'(\xi)+\phi_x^{i+1} g'(1-\xi)\;,
\label{eq:Hermite_cubic_1D_prime}
\end{equation}
where we use the relative coordinate $\xi = \frac{x-x_i}{h}$.
\end{exmp}

\begin{rem}
\label{rem:cubic_1d_minimize_l2}
The 1D Hermite cubic \eqref{eq:Hermite_cubic_1D} is the unique minimizer of the functional
$I(u) = \int_{x_i}^{x_{i+1}} u_{xx}^2 \ud{x}$ under the constraints that the function
values and derivatives are matched at $x_i$ and $x_{i+1}$. This is a standard problem in
calculus of variations. The cubic \eqref{eq:Hermite_cubic_1D} is a minimizer, since it
solves the corresponding Euler-Lagrange equation $u_{xxxx} = 0$. It is the unique
minimizer, since $I$ is convex, and the domain for the minimization problem is also convex.
Thus, the 1D Hermite cubic minimizes the $L^2$ norm of the second derivative.
\end{rem}

\begin{lemma}
\label{lem:p-cubic_accuracy_data}
Let the data determining the $p$-cubic Hermite interpolant $\mathcal{H}$
in Def.~\ref{def:p-cubic_interpolant} be known only up to some error.
Then equation \eqref{eq:p-cubic_interpolant} yields the interpolation error
\begin{equation*}
\delta\mathcal{H}(\vec{x}) = \sum_{\vec{v},\vec{\alpha}\in\{0,1\}^p}
W_{\vec{\alpha}}^{\vec{v}}\prn{\frac{\vec{x}}{h}}\,
h^{\abs{\vec{\alpha}}}\, \delta\phi_{\vec{\alpha}}^{\vec{v}}\;,
\end{equation*}
where we use the notation $\delta u$ to indicate the error in some quantity $u$.
In particular, if the data $\phi_{\vec{\alpha}}^{\vec{v}}$ are known with
$O\prn{h^{4-\abs{\vec{\alpha}}}}$ accuracy, then
$\delta\prn{\partial^{\,\vec{\alpha}}\mathcal{H}} = O\prn{h^{4-\abs{\vec{\alpha}}}}$.
\end{lemma}

\begin{thm}
\label{thm:p-cubic_accuracy}
Let $\phi$ and $\mathcal{H}$ be as in Def.~\ref{def:p-cubic_interpolant}.
Then, for any point in $\mathcal{C}_h$, we have
\begin{equation}
\partial^{\,\vec{\alpha}}\mathcal{H}
= \partial^{\,\vec{\alpha}}\phi+O\prn{h^{4-\abs{\vec{\alpha}}}}\;,
\label{eq:p-cubic_error}
\end{equation}
where the coefficients in the error terms $O(h^\mu)$ can be bounded by some constant
multiple of the norm $\|D^4 \phi\|_\infty$ in $\mathcal{C}_h$.
\end{thm}

\begin{proof}
Let $\mathcal{G} = \mathcal{G}(\vec{x})$ be the degree three polynomial obtained by
using a Taylor expansion for $\phi$, centered at some point in $\mathcal{C}_h$.
Then, by construction, $\mathcal{G}$ satisfies \eqref{eq:p-cubic_error} above.
In particular, the data for $\mathcal{G}$ on $\mathcal{C}_h$ is related to the data
for $\mathcal{H}$ (same as the data for $\phi$) in the manner specified
in Lemma~\ref{lem:p-cubic_accuracy_data}. Hence \eqref{eq:p-cubic_error} follows.
\end{proof}

Lemma~\ref{lem:p-cubic_accuracy_data} and Thm.~\ref{thm:p-cubic_accuracy} imply that
the $p$-cubic interpolant is fourth order accurate, if any required derivative
$\partial^{\,\vec{\alpha}}\phi$ is known with $O\prn{h^{4-\abs{\vec{\alpha}}}}$
accuracy on the cell vertices. Since in a gradient-augmented method, we only assume
that the function values ($\abs{\vec{\alpha}} = 0$) and the
gradients ($\abs{\vec{\alpha}} = 1$) are given, we need to construct any required cross
derivatives ($\abs{\vec{\alpha}} \ge 2$) with accuracy $O\prn{h^{4-\abs{\vec{\alpha}}}}$,
from the given data.
Interestingly, this means that one can set to zero all the derivatives of order higher
than 3, without affecting the interpolant's accuracy. Unfortunately, this convenience
does not come into play for dimensions $p\le 3$.

The cross derivatives required by \eqref{eq:p-cubic_interpolant} can be constructed
to the appropriate order from the function values and the derivatives at the grid
points. Two possible approaches are:
\begin{enumerate}[(A)]
\item\textbf{Central differencing.~}
Second order cross derivatives can be approximated by central differences of
neighboring points, such as (here written for $p=3$ on a cell of size
$\Delta x \times \Delta y \times \Delta z$, all proportional to $h$)
\begin{equation*}
\phi_{xy}^{i,j,k} = \frac{\phi_y^{i+1,j,k}
-\phi_y^{i-1,j,k}}{2\Delta x}+O(h^2)\;.
\end{equation*}
By construction, these approximations are second order accurate.
The third order cross derivative can be approximated by
\begin{equation*}
\phi_{xyz}^{i,j,k} = \frac{\phi_z^{i+1,j+1,k}
-\phi_z^{i-1,j+1,k}-\phi_z^{i+1,j-1,k}
+\phi_z^{i-1,j-1,k}}{4\Delta x\Delta y}+O(h^2)\;.
\end{equation*}
This approach defines one unique value for each required cross derivative
at each grid point. Hence, the thus defined $p$-cubic interpolant is $C^1$ across cell
edges. A technical disadvantage is that the optimal locality is to a certain extent lost:
The interpolation in a cell uses information from adjacent cells corner points.
\item\textbf{Cell-based approach.~}
In 2D, a second order accurate approximation to the second order cross derivatives
at the vertices of a cell can be obtained from the gradient values
at the vertices in the same cell, using finite differences
and interpolation/extrapolation:
\begin{enumerate}[(1)]
\item The central differences
$\frac{\phi_x^{i,j+1}-\phi_x^{i,j}}{\Delta y}$,
$\frac{\phi_x^{i+1,j+1}-\phi_x^{i+1,j}}{\Delta y}$,
$\frac{\phi_y^{i+1,j}-\phi_y^{i,j}}{\Delta x}$,
$\frac{\phi_y^{i+1,j+1}-\phi_y^{i,j+1}}{\Delta x}$
approximate $\phi_{xy}$ at the cell edge centers.
\item
Weighted averages of these values yield approximations to $\phi_{xy}$ at the cell
vertices. The weights follow from bilinear interpolation, and are
$\tfrac{3}{4}$ for the two nearby edge centers, and $-\tfrac{1}{4}$ for the two
opposing edge centers.
\end{enumerate}
To obtain the cross derivatives in 3D, the formulas above are applied in a facet of
the 3D cell.

A first order accurate approximation to the third cross derivative is given by
\begin{equation*}
\phi_{xyz}^{i,j,k} = \frac{\phi_z^{i+1,j+1,k}
-\phi_z^{i,j+1,k}-\phi_z^{i+1,j,k}+\phi_z^{i,j,k}}{\Delta x\Delta y}+O(h)\;.
\end{equation*}
This approach is purely cell-based, and the interpolation can be implemented as
a single black-box routine. A technical disadvantage of this approach is that one and
the same grid point is assigned different cross derivative values, depending on which
cell it is a vertex of. Since the approximations are second order accurate, the various
cross derivatives at a grid point differ by $O(h^2)$. Consequently, the resulting $p$-cubic
interpolant is continuous, but the gradient jumps across cell edges, with discontinuities
of size $O(h^3)$.
\end{enumerate}
In practice, both approaches perform well.
Note that above approximations are just one possible way to approximate the
cross derivatives. The final method proves rather robust with respect to the
actual approximation chosen. In fact, even if an $O(1)$ error is done in the
cross derivatives (e.g.~by setting them equal to zero), the method remains convergent,
though with a lower order of accuracy.

\section{Benefits of Incorporating Gradient Information}
\label{sec:benefits_gradient}
While the classical level set method, outlined in Sect.~\ref{subsec:classical_LS}, is
a powerful tool in representing and advecting surfaces, it suffers from some problems
and inconveniences. In this paper, we address three fundamental aspects,
and show how a gradient-augmented approach, as outlined in
Sect.~\ref{subsec:gradient-augmented_LS}, can ameliorate them:
\begin{itemize}
\item
Small structures are lost once below the grid resolution.
Gradients yield a certain level of subgrid resolution
(Sect.~\ref{subsec:small_structures}).
\item
An accurate approximation of the curvature involves difficulties.
With gradients, surface normals and curvature can be easily obtained from the
$p$-cubic Hermite interpolation (Sect.~\ref{subsec:derivative_quantities}).
\item
Accurate schemes for the advection equation \eqref{eq:advection_phi} involve large
stencils. With gradients, a third order accurate scheme can be formulated, with
optimally local stencils (Sect.~\ref{subsec:small_stencils}).
\end{itemize}

\begin{figure}
\centering
\begin{minipage}[t]{.31\textwidth}
\centering
\includegraphics[width=0.99\textwidth]{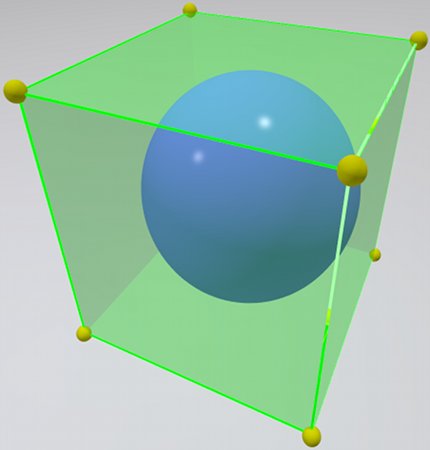}
\end{minipage}
\hfill
\begin{minipage}[t]{.31\textwidth}
\centering
\includegraphics[width=0.99\textwidth]{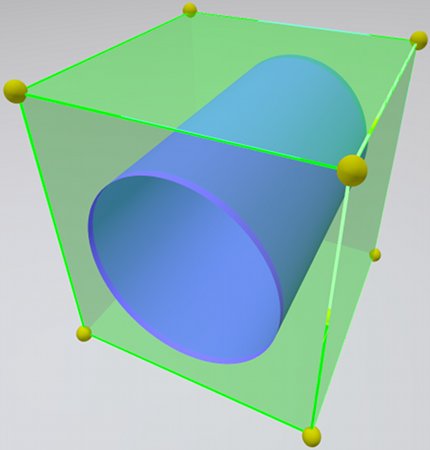}
\end{minipage}
\hfill
\begin{minipage}[t]{.31\textwidth}
\centering
\includegraphics[width=0.99\textwidth]{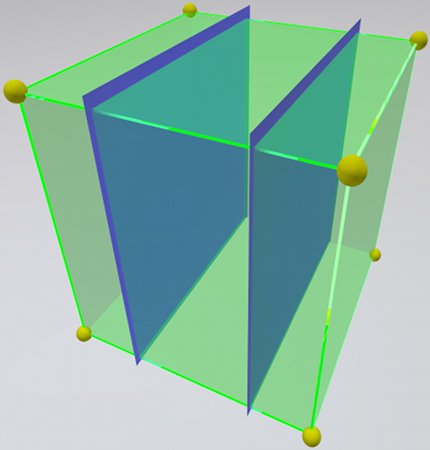}
\end{minipage}
\caption{Subgrid structures in 3D, defined by a tri-cubic:
a drop, a jet, and a film}
\label{fig:subgrid_3d}
\end{figure}

\subsection{Representation of Structures of Subgrid Size}
\label{subsec:small_structures}
Structures that are at least a few cells wide are represented well by the classical
level set method. However, smaller structures are only represented if they
contain grid points. This leads to the inconsistency that one and the same small
structure can be present (if a grid point happens to fall into it) or be missing
(if it happens to fall in between grid points). Furthermore, even if initially present,
small structures may vanish over the course of a computation due to approximation errors
in the numerical scheme, such as numerical diffusion or dispersion.
Small drops, jets or films (see Fig.~\ref{fig:subgrid_3d}) may be lost during a
computation, resulting in a ``loss of volume''. Other difficulties are the numerical
coalescence of nearby structures, or a numerical pinch-off in a thinning process.
In practice, the loss of small structures can be prevented using adaptive mesh
refinement (AMR) techniques \cite{Strain1999,MinGibou2007}, which add a significant
level of complexity, especially when high order approximations need to be preserved
across multiple levels of refinement. The problem of volume loss can be addressed by
enforcing conservation of volume, as proposed by Sussman and Fatemi \cite{SussmanFatemi1999}.
Such methods guarantee conservation of volume (up to a small approximation error), but
may yield incorrect topologies. An alternative approach is to augment the level set
function by Lagrangian particles, as proposed by
Enright, Fedkiw, Ferziger, and Mitchell \cite{EnrightFedkiwFerzigerMitchell2002}.
This latter method resolves the difficulties described above in a satisfactory manner,
but at the expense of simplicity.

Fig.~\ref{fig:subgrid_1d_cubic} shows a signed distance level set function $\phi$
(solid line) of a 1D ``bubble'' between two grid points. Classical level set methods
use a linear interpolation (dashed line). Hence, no structure is identified, since the
level set function has equal sign on the two grid points. Of course, the same structure
would be identified if a grid point fell within it, but it would be lost when
advected into the situation shown in Fig.~\ref{fig:subgrid_1d_cubic}.
Also, if one knows that $\phi$ is a signed distance function, the surface can be
identified, as shown by Chopp \cite{Chopp2001}. Unfortunately, this approach incurs
difficulties in higher space dimensions, where ambiguities arise. In addition, the
assumption of having a precise signed distance function is too limiting in many cases.

A gradient-augmented level set approach allows the representation of a full structure,
i.e.~two surfaces, between two neighboring grid points.
Fig.~\ref{fig:subgrid_1d_cubic} shows a cubic interpolant (dash-dotted line) along a cell
edge, constructed from the correct function values and gradients. In fact, a subgrid
structure is detected. Similarly, with bi-/tri-cubics, subgrid structures in 2D and 3D
can be approximately represented. Each subgrid structure shown in Fig.~\ref{fig:subgrid_3d}
is defined by a tri-cubic, as follows. For a cube of size $h$, we consider the function
values on the vertices $\phi=0.1$. The drop (sphere) is then obtained by providing
gradients
$\vec{\psi} = \tfrac{1}{3h}\prn{\vec{x}-\prn{\tfrac{1}{2},\tfrac{1}{2},\tfrac{1}{2}}}$.
The jet (cylinder) is obtained by
$\vec{\psi} = \tfrac{1}{2h}\prn{0,y-\tfrac{1}{2},z-\tfrac{1}{2}}$.
And the film (two planes) is obtained by
$\vec{\psi} = \tfrac{1}{h}\prn{0,y-\tfrac{1}{2},0}$.

At the same time, the example in Fig.~\ref{fig:subgrid_1d_cubic} indicates a major drawback
of $p$-cubic approaches: Level set functions are typically non-smooth (e.g.~signed distance
functions). Hermite $p$-cubics are smooth, hence the approximation is not very accurate
near kinks. As a result, even with $p$-cubic approaches, smaller structures may vanish
eventually, though significantly later than with classical level set methods.
A potential remedy can be to use higher order, or nonlinear interpolations, which we
shall investigate in future work.

\begin{figure}
\centering
\begin{minipage}[t]{.6\textwidth}
\centering
\includegraphics[width=0.5\textwidth]{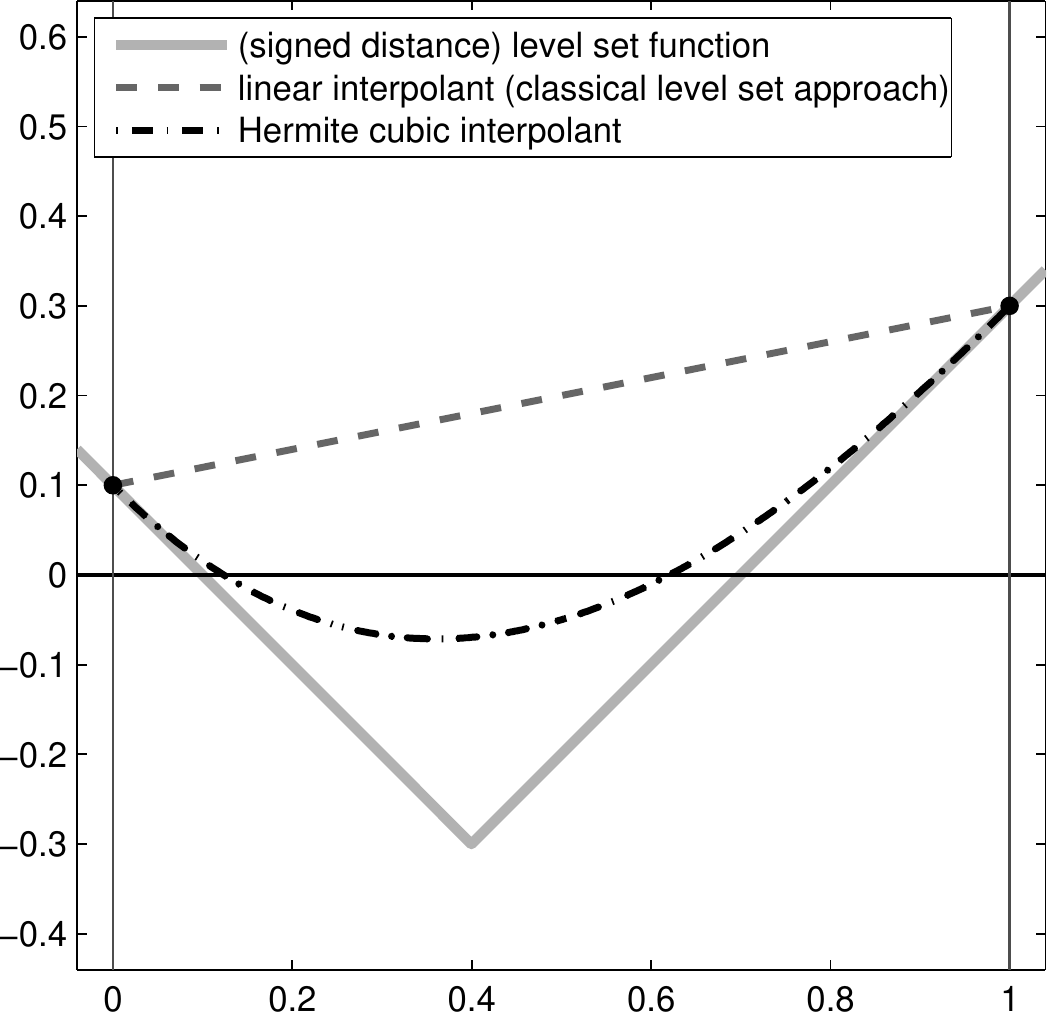}
\caption{A subgrid structure in 1D, not identified by linear interpolation,
but recovered by a Hermite cubic}
\label{fig:subgrid_1d_cubic}
\end{minipage}
\end{figure}

Theoretically, the gradient-augmented level set approach allows the representation of
(at least some) isolated structures of arbitrarily small size. Hence, the use of gradient
information is more than a mere increase of resolution. Of course, in practice there
are limitations to the size of the small structures that can be represented reliably.
In addition, only isolated subgrid structures can be represented. The structure shown
in Fig.~\ref{fig:subgrid_1d_cubic} is indistinguishable from two small structures in the
cell, whose outer boundaries are at the same positions as the boundaries of the single
structure.

\subsection{Approximation of Derivative Quantities}
\label{subsec:derivative_quantities}
In many applications (e.g.~in two-phase flow simulations with surface tension), normal
vectors and curvatures are required. In terms of the first and second derivatives of
the level set function $\phi$, the expressions are (here in 2D)
\begin{align}
\vec{n} &= \frac{\prn{\phi_x,\phi_y}^T}{(\phi_x^2+\phi_y^2)^\frac{1}{2}}\;,
\label{eq:surface_normal} \\
\kappa  &= \frac{\phi_{xx}\phi_y^2-2\phi_x\phi_y\phi_{xy}+\phi_{yy}\phi_x^2}
{(\phi_x^2+\phi_y^2)^\frac{3}{2}}\;.
\label{eq:curvature}
\end{align}
In classical level set methods, the required derivatives
$\phi_x$, $\phi_y$, $\phi_{xx}$, $\phi_{xy}$, $\phi_{yy}$ are typically obtained
from $\phi$ by central differences. Assuming that the level set function $\phi$ is known
with fourth order accuracy on the grid points, second order accurate approximations
to normals and curvature on the grid points are obtained from the
expressions \eqref{eq:surface_normal} and \eqref{eq:curvature}.

Away from grid points (such as it is required by the ghost fluid method), second order
accurate approximations can be obtained by using weighted averages. Several problems
with computing curvature from level set (and volume of fluid) functions, in particular
in connection with reinitialization \eqref{eq:reinitialization}, have been
presented by Raessi et al.~\cite{RaessiMostaghimiBussmann2007}.
In addition, even if a suitably high order advection scheme is used, under various
circumstances the level set function values can cease to be fourth order accurate.
For instance, when a WENO stencil crosses a discontinuity in the gradient
(e.g.~signed distance level set functions \eqref{eq:distance_function} possess
discontinuous derivatives), the accuracy of the approximation can
drop \cite{LiuOsherChan1994}.

Gradient-augmented level set methods offer new possibilities in the approximation of
derivative quantities. For instance, the $p$-cubic Hermite interpolation considered here
gives rise to a particularly simple way to obtain normals and curvatures at arbitrary
points. Inside each cell, all first and second derivatives can be obtained analytically
by differentiating the interpolant \eqref{eq:p-cubic_interpolant}. Normal vectors and
curvature are then simply obtained by \eqref{eq:surface_normal},
respectively \eqref{eq:curvature}.
Due to Thm.~\ref{thm:p-cubic_accuracy}, first derivatives are third order accurate,
and so is the obtained $\vec{n}$. Further, second derivatives are second order
accurate, and so is the obtained $\kappa$. The numerical results shown
in Sect.~\ref{subsubsec:numerics_curvature} verify this.
Hence, using the Hermite $p$-cubic, derivative quantities (of up to second order) can
be obtained everywhere with at least second order accuracy, by a simple recipe with
optimally local stencils, i.e.~they use information only from a single cell.

\subsection{Coherent Advection with Optimally Local Stencils}
\label{subsec:small_stencils}
In order to both accurately advect structures, and calculate derivative quantities,
high order schemes have to be used. In classical approaches, the stencils for high
order schemes reach over multiple cells. This results in a cumbersome
implementation on adaptive grids and near boundaries.

In Sect.~\ref{sec:scheme_CIR}, we present a semi-Lagrangian approach that solves
the advection problem with third order accuracy. The characteristic form of the
equations \eqref{eq:advection_phi} and \eqref{eq:advection_psi} is used to
update the function values and gradients at the grid points.
In each time step, the characteristics are traced backwards from the grid points,
and function values and gradients on the characteristic, at the prior time step,
are extracted from the Hermite interpolant, as defined
in Sect.~\ref{subsec:Hermite_interpolant}.
A key advantage of this approach is that at each grid point, the data is updated using
only information from a single adjacent cell. Hence, the gradient-augmented approach
delivers a high order advection with optimally local stencils.
It seems evident that the optimal locality should allow a simple treatment when combined
with adaptive mesh refinement and near boundaries. The detailed investigation of this
matter is left for future work.

\section{Numerical Methodology of the Generalized CIR Scheme}
\label{sec:scheme_CIR}
We consider the linear advection equation \eqref{eq:advection_phi}
with $\vec{v}=\vec{v}(\vec{x},t)$.
The evolution of the level set function and its gradient, given by
\eqref{eq:advection_phi} and \eqref{eq:advection_psi} can be rewritten as
\begin{equation}
\begin{split}
\phi_t+\vec{v}\cdot\nabla\phi &= 0 \\
\vec{\psi}_t+\vec{v}\cdot\nabla\vec{\psi} &= -\nabla\vec{v}\cdot\vec{\psi}\;.
\end{split}
\label{eq:evolution_phi_psi}
\end{equation}
For functions defined everywhere, the evolution for $\vec{\psi}$ is completely
determined by $\phi$. However, for functions known only on grid points, the
gradients $\vec{\psi}$ carry additional information that is not encoded in $\phi$.
For the moment, we assume the velocity field $\vec{v} = \prn{u,v,w}^T$, and the
velocity deformation matrix
\begin{equation*}
\nabla\vec{v}
= \begin{pmatrix}
\pd{u}{x} & \pd{v}{x} & \pd{w}{x} \\
\pd{u}{y} & \pd{v}{y} & \pd{w}{y} \\
\pd{u}{z} & \pd{v}{z} & \pd{w}{z}
\end{pmatrix}
\end{equation*}
to be exactly accessible everywhere.
The system \eqref{eq:evolution_phi_psi} consists of an advection part (left hand side),
and a source term for $\vec{\psi}$. Its characteristic form is
\begin{equation}
\frac{d\phi}{dt} = 0
\quad\text{and}\quad
\frac{d\vec{\psi}}{dt} = -\nabla\vec{v}\cdot\vec{\psi}
\quad\text{along}\quad
\frac{d\vec{x}}{dt} = \vec{v}(\vec{x},t)\;.
\label{eq:characteristic_equations}
\end{equation}
This system of ODE describes the evolution of $\phi$ and $\vec{\psi}$ along the
characteristic curves defined by $\frac{d\vec{x}}{dt} = \vec{v}(\vec{x},t)$.
We solve these equations by a generalization
of the CIR method by Courant, Isaacson, and Rees \cite{CourantIsaacsonRees1952}:
the characteristics \eqref{eq:characteristic_equations} are traced backwards from the
grid points, and the required data is obtained by the interpolation.
Note that, while the CIR method is, in general, non-conservative for nonlinear
conservation laws, its use is justified here, since equations \eqref{eq:advection_phi}
and \eqref{eq:advection_psi} are linear.

Let $\phi^n$, $\vec{\psi}^n$ denote the data at time $t$, and
$\phi^{n+1}$, $\vec{\psi}^{n+1}$ denote the data at time $t+\Delta t$.
We find the characteristic curves that pass through grid points at time $t+\Delta t$,
and solve \eqref{eq:characteristic_equations} along these characteristic curves.
One time step of the scheme from $t$ to $t+\Delta t$ reads as follows:
\begin{enumerate}
\item
For each grid point $\vec{x}_j$, solve the characteristic equation
\begin{equation}
\pd{}{\tau}\vec{x}(\tau) = \vec{v}(\vec{x}(\tau),\tau)\;,\quad
\vec{x}(t+\Delta t) = \vec{x}_j\;,
\label{eq:cir_characteristics}
\end{equation}
backwards from $\tau = t+\Delta t$ to $\tau = t$,
to find the point $\mathring{\vec{x}}_j = \vec{x}(t)$. The
characteristic curve that passes through $\mathring{\vec{x}}_j$
at time $t$, reaches $\vec{x}_j$ at time $t+\Delta t$.
\item
Assign
$\phi^{n+1}_j(\vec{x}_i) = \phi(\mathring{\vec{x}}_j)$ and
$\mathring{\vec{\psi}}_j = \vec{\psi}(\mathring{\vec{x}}_j)$,
both evaluated analytically from the Hermite interpolation \eqref{eq:p-cubic_interpolant}
in the cell that $\mathring{\vec{x}}_j$ falls into.
\item
Solve \eqref{eq:characteristic_equations} forward along the characteristic curve,
i.e.~solve
\begin{equation}
\pd{}{\tau}\vec{\psi}(\tau) = -\nabla\vec{v}(\vec{x}(\tau),\tau)\cdot\vec{\psi}(\tau)\;,
\quad \vec{\psi}(t) = \mathring{\vec{\psi}}_j\;,
\label{eq:cir_evolution_psi}
\end{equation}
from $\tau = t$ to $\tau = t+\Delta t$.
Assign $\vec{\psi}^{n+1}_j(\vec{x}_i) = \vec{\psi}(t+\Delta t)$.
\end{enumerate}
In principle, there is no requirement for the backward characteristic curves to
remain in a single cell, as long as the integration scheme
for \eqref{eq:cir_characteristics} is accurate enough to prevent the approximate
backwards characteristic curves from intersecting (which may cause oscillations).
However, in practice, it is often convenient, and not too restrictive,
to choose $\Delta t<\tfrac{v_{\text{max}}}{2} h$, where $v_{\text{max}}$ is
the maximum magnitude of $\vec{v}$, and $h=\min\{\Delta x,\Delta y,\Delta z\}$
is the grid size. This guarantees that characteristic curves from different grid points
do not intersect, and that characteristic curves do not cross multiple cells.

As proved in Thm.~\ref{thm:p-cubic_accuracy}, the interpolation approximates $\phi$
with fourth order accuracy and $\vec{\psi}$ with third order accuracy. Therefore,
in order to achieve the maximum possible order of accuracy, the backwards characteristics
equation \eqref{eq:cir_characteristics} needs to be solved with fourth order accuracy
in each time step.\footnote{This then yields a globally third order in time algorithm,
see Sect.~\ref{subsec:convergence}.}
One step of the Shu-Osher RK3 method \cite{ShuOsher1988} does the job.
Since gradients are generally one order less accurate, their evolution
equation \eqref{eq:cir_evolution_psi} needs to be solved with third order accuracy.
One step of Heun's RK2 method (explicit trapezoidal) does the job.
However, in Sect.~\ref{sec:numerical_analysis} we outline another possibility, which is to
systematically inherit the update rule for \eqref{eq:cir_evolution_psi} from the
scheme used for \eqref{eq:cir_characteristics}. We also investigate the accuracy and
stability of the gradient-augmented scheme.

In one space dimension, for constant velocity fields, the presented approach is
equivalent to the CIP method \cite{TakewakiNishiguchiYabe1985,TakewakiYabe1987},
which also uses a $p$-cubic interpolant. Note that the presented gradient-augmented
level set approach is not limited to $p$-cubic interpolants. Within the setting of
superconsistency, introduced in Sect.~\ref{subsec:superconsistency}, the
projection step, defined in Sect.~\ref{subsec:projection}, can be replaced by
other forms of projection (see Rem.~\ref{rem:other_projection}).

The gradient-augmented level set method is not a finite difference method, since
differential operators are not approximated. Instead, fundamental properties of the
exact solution to the underlying equation are used, and derivatives are evaluated
analytically from the interpolation patch. The approach is based on local basis
functions, as a finite element method is. However, the update rule is very different
than for classical finite element approaches.

The generalized CIR approach is cell-based, thus preserving the benefits of the level
set method, such as handling of topology changes, parallelizability, etc.
In addition, it is optimally local: The new data values at a grid point use only
information from corner points of a single adjacent cell. Compared to WENO schemes,
which use information multiple cells away, the CIR method promises an easier treatment of
approximations on locally refined meshes and near domain boundaries. In addition, the
locality allows smaller structures to get close together without spoiling the accuracy
of the approximation.

\subsection{Boundary Conditions}
When solving the linear advection equation \eqref{eq:advection_phi} on a domain $\Omega$
with outward normals $\vec{n}$, boundary conditions have to be prescribed wherever
the flow enters the domain, i.e.~$\vec{v}\cdot\vec{n}<0$.
The accurate treatment of boundary conditions is a common challenge in high order
methods. For level set methods, this is important whenever the interface is close to
the boundary. In WENO methods, ghost points have to be created using appropriate
extrapolations. Since WENO stencils reach over multiple cells, not only do inflow
boundaries pose a challenge, but also outflow boundaries, at which the actual
equation does not require any boundary conditions.

In contrast, the gradient-augmented CIR method presented above, treats outflow
boundaries naturally: no information has to be prescribed, since characteristics come
from inside the domain. At inflow boundaries, information for both $\phi$ and
$\vec{\psi}$ has to be prescribed. Below we give two illustrative examples that show
how to create the required information from the advection
equation \eqref{eq:advection_phi}, for the two most common types of boundary conditions.
Other types of boundary conditions can be treated similarly.

Consider a domain boundary that is aligned with cell edges. WLOG assume that this
boundary is perpendicular to $\vec{n} = (1,0,0)$. Then
\begin{itemize}
\item\textbf{Dirichlet boundary conditions}
prescribe $\phi$ on the boundary, which is perpendicular to $\vec{n} = (1,0,0)$.
Therefore, all partial derivatives perpendicular to $\vec{n}$, i.e.~$\phi_y$
and $\phi_z$, as well as $\phi_t$, are prescribed as well. The missing derivative
$\phi_x$ on the boundary follows from the equation \eqref{eq:advection_phi} as
\begin{equation*}
\phi_x = -\frac{\phi_t+v\phi_y+w\phi_z}{u}\;.
\end{equation*}
Note that $u\neq 0$, since $\vec{v}\cdot\vec{n}<0$.
\item\textbf{Neumann boundary conditions}
prescribe $\phi_x$ on the boundary. Thus equation \eqref{eq:advection_phi} yields a
linear advection equation on the boundary
\begin{equation}
\phi_t+v\phi_y+w\phi_z = -u\phi_x\;,
\label{eq:Neumann_bc_evolution}
\end{equation}
with initial conditions given by the values of $\phi$ on the boundary at $t=0$.
Equation \eqref{eq:Neumann_bc_evolution} can again be solved using the
gradient-augmented CIR scheme. This yields the function values $\phi$ and the
derivatives perpendicular to $\vec{n}$, i.e.~$\phi_y$ and $\phi_z$.
\end{itemize}

\section{Analysis of the Gradient-Augmented CIR Method}
\label{sec:numerical_analysis}
Consider the numerical scheme presented in Sect.~\ref{sec:scheme_CIR}, for
the linear advection equation \eqref{eq:advection_phi}, with $\vec{v}=\vec{v}(\vec{x},t)$
given. Let $\vec{X}(\vec{x},t,\tau)$ be the solution to the characteristic
equations \eqref{eq:characteristic_equations}, defined by
\begin{equation}
\pd{}{\tau}\vec{X}(\vec{x},t,\tau) = \vec{v}(\vec{X}(\vec{x},t,\tau),\tau)\;,
\quad\vec{X}(\vec{x},t,t) = \vec{x}\;.
\label{eq:characteristics_solution}
\end{equation}
The solution operator to the linear advection equation \eqref{eq:advection_phi} is then
\begin{equation}
S_{t,t+\Delta t}\phi(\vec{x},t)
= \phi(\vec{x},t+\Delta t)
= \phi(\vec{X}(\vec{x},t+\Delta t,t),t)\;.
\label{eq:evolution_exact}
\end{equation}
Now consider a numerical approximation $\vec{\mathcal{X}}$ to $\vec{X}$, as arising from
a numerical ODE solver, e.g.~a Runge-Kutta scheme. Let the corresponding approximate
solution operator to \eqref{eq:advection_phi} be denoted by
\begin{equation}
A_{t,t+\Delta t}\phi(\vec{x},t)
= \phi(\vec{\mathcal{X}}(\vec{x},t+\Delta t,t),t)\;.
\label{eq:evolution_approximate}
\end{equation}

\subsection{Superconsistency}
\label{subsec:superconsistency}
Equation \eqref{eq:evolution_exact} provides the exactly evolved solution, while
equation \eqref{eq:evolution_approximate} defines an approximately evolved solution,
both at every point $\vec{x}$ in the computational domain.
In the actual numerical method, we consider only the (approximate) characteristic curves
that go through the grid points at time $t+\Delta t$. However, we can use the solution
operator \eqref{eq:evolution_approximate} to derive a natural update rule for the
gradients. The gradient of the approximately advected function is
\begin{equation}
\nabla\prn{A_{t,t+\Delta t}\phi(\vec{x},t)}
= \nabla\vec{\mathcal{X}}(\vec{x},t+\Delta t,t)
\cdot\nabla\phi(\vec{\mathcal{X}}(\vec{x},t+\Delta t,t),t)\;,
\end{equation}
which yields the update rule for approximate function values $\phi$ and
gradients $\vec{\psi}$
\begin{equation}
\begin{cases}
\phi(\vec{x},t+\Delta t)
= \phi(\vec{\mathcal{X}}(\vec{x},t+\Delta t,t),t) \\
\vec{\psi}(\vec{x},t+\Delta t)
= \nabla\vec{\mathcal{X}}(\vec{x},t+\Delta t,t)
\cdot\vec{\psi}(\vec{\mathcal{X}}(\vec{x},t+\Delta t,t),t)
\end{cases}
\label{eq:superconsistency}
\end{equation}

\begin{defn}
\label{def:superconsistency}
We call a gradient-augmented scheme superconsistent, if it satisfies
\eqref{eq:superconsistency}.
\end{defn}

\begin{exmp}
Using forward Euler to approximate \eqref{eq:characteristics_solution} yields
the superconsistent scheme
\begin{align*}
\mathring{\vec{x}} &= \mathcal{X}(\vec{x},t+\Delta t,t)
= \vec{x}-\Delta t\;\vec{v}(\vec{x},t+\Delta t) \\
\nabla\mathring{\vec{x}} &= I-\Delta t\;\nabla\vec{v}(\vec{x},t+\Delta t) \\
\phi(\vec{x},t+\Delta t) &= \phi(\mathring{\vec{x}},t) \\
\vec{\psi}(\vec{x},t+\Delta t)
&= \nabla\mathring{\vec{x}}\cdot\vec{\psi}(\mathring{\vec{x}},t)
\end{align*}
\end{exmp}
While this scheme is particularly simple, it is only first order accurate.
A globally third order accurate scheme is obtained when using a third order
Runge-Kutta scheme, such as in the following example.
\begin{exmp}
\label{ex:superconsistent_ShuOsher}
Using the Shu-Osher scheme \cite{ShuOsher1988}
for \eqref{eq:characteristics_solution} yields the superconsistent scheme
\begin{align*}
\vec{x}_1 &= \vec{x}-\Delta t\;\vec{v}(\vec{x},t+\Delta t) \\
\nabla\vec{x}_1 &= I-\Delta t\;\nabla\vec{v}(\vec{x},t+\Delta t) \\
\vec{x}_2 &= \vec{x}-\Delta t
\prn{\tfrac{1}{4}\vec{v}(\vec{x},t+\Delta t)+\tfrac{1}{4}\vec{v}(\vec{x}_1,t)} \\
\nabla\vec{x}_2 &= I-\Delta t
\prn{\tfrac{1}{4}\nabla\vec{v}(\vec{x},t+\Delta t)
+\tfrac{1}{4}\nabla\vec{x}_1\cdot\nabla\vec{v}(\vec{x}_1,t)} \\
\mathring{\vec{x}} &= \vec{x}-\Delta t
\prn{\tfrac{1}{6}\vec{v}(\vec{x},t+\Delta t)+\tfrac{1}{6}\vec{v}(\vec{x}_1,t)
+\tfrac{2}{3}\vec{v}(\vec{x}_2,t+\tfrac{1}{2}\Delta t)} \\
\nabla\mathring{\vec{x}} &= I-\Delta t
\prn{\tfrac{1}{6}\nabla\vec{v}(\vec{x},t+\Delta t)
+\tfrac{1}{6}\nabla\vec{x}_1\cdot\nabla\vec{v}(\vec{x}_1,t)
+\tfrac{2}{3}\nabla\vec{x}_2\cdot\nabla\vec{v}(\vec{x}_2,t+\tfrac{1}{2}\Delta t)} \\
\phi(\vec{x},t+\Delta t) &= \phi(\mathring{\vec{x}},t) \\
\vec{\psi}(\vec{x},t+\Delta t)
&= \nabla\mathring{\vec{x}}\cdot\vec{\psi}(\mathring{\vec{x}},t)
\end{align*}
\end{exmp}
Of course, it is not necessary to use a superconsistent scheme to update $\vec{\psi}$.
In fact, other schemes to approximate \eqref{eq:advection_psi} might be more
accurate and/or more efficient. However, superconsistent schemes have the advantage
that the gradients are evolved exactly as if the function values were evolved
everywhere, and then the gradients were obtained by differentiating the function.
This increases the coherence between function values and derivatives. In addition,
superconsistent schemes can be analyzed in function spaces, without actually having
to consider an evolution for the gradient.

\subsection{Projection}
\label{subsec:projection}
The update rule \eqref{eq:superconsistency} uses the function values $\phi$ and the
gradients $\vec{\psi}$ at $\mathring{\vec{x}}_j = \mathcal{X}(\vec{x}_j,t+\Delta t,t)$
for each grid point $\vec{x}_j$. The point $\mathring{\vec{x}}_j$ is in general not a
grid point. Hence, in the gradient-augmented numerical scheme,
$\phi(\mathring{\vec{x}}_j,t)$ and $\vec{\psi}(\mathring{\vec{x}}_j,t)$ are defined
by the Hermite interpolation.
Each time step of a superconsistent gradient-augmented scheme can be interpreted
(in a function space) as an approximate advection step, followed by a projection step
\begin{equation*}
M_{t,t+\Delta t} = PA_{t,t+\Delta t}\;.
\end{equation*}
At time $t$, the level set function $\phi$ is represented by the cell-based $p$-cubic
interpolant approximation. This function is then advanced in time to $t+\Delta t$ by
the approximate advection operator $A_{t,t+\Delta t}$. Then a new representation in
terms of cell-based $p$-cubic interpolants is obtained by the projection operator $P$,
which uses the function and gradient values of the advected function at the grid points.

In this paper, the projection $P$ is given by the $p$-cubic Hermite interpolant defined in
equation \eqref{eq:p-cubic_interpolant}, with the proviso that the required cross
derivatives that appear in \eqref{eq:p-cubic_interpolant} are obtained by one of the
(various possible) finite differentiation schemes described
in Sect.~\ref{subsec:Hermite_interpolant}.

In general, the operator $A_{t,t+\Delta t}$ alone is a much better approximation to
$S_{t,t+\Delta t}$ than the combined $M_{t,t+\Delta t} = PA_{t,t+\Delta t}$. However,
the projection step $P$ is required to be able to represent the numerical approximation
to $\phi$, at a given instance in time, by a finite amount of data. For the
choice of $P$ used in this paper, a function $P\phi(\vec{x})$ is uniquely defined by the
function values $\phi(\vec{x}_i)$ and gradient values $\nabla\phi(\vec{x}_i)$
on the grid points $\vec{x}_i$. Hence, in the numerical scheme, the advection
operator $A_{t,t+\Delta t}$ has to be evaluated only along the characteristic curves
that go through the grid points, using only the function values and derivatives there.

\begin{rem}
\label{rem:other_projection}
In principle, the presented generalized CIR algorithm works for any projection which
depends on a finite amount of information that can be extracted from the function being
projected. A superconsistent scheme is defined by selecting an approximate advection
scheme and an appropriate projection strategy. The $p$-cubic interpolation projection
used here is linear, generalizes nicely to higher space dimensions, and gives rise to
a simple algorithm. However, other projections may exist, providing increased subgrid
resolution, or other advantages. Obvious candidates include the incorporation of higher
order information, such as curvature. We plan to investigate these, and other
alternatives, in future work.
\end{rem}

\subsection{Consistency}
\label{subsec:consistency}
We investigate the consistency of the method, by estimating the truncation error after
applying a single step of the numerical scheme to the exact solution. Let $\Delta t$
denote the time step, and $h = \max\{\Delta x,\Delta y,\Delta z\}$ the grid size.

\begin{thm}[consistency]
Assume that the exact solution $\phi$ is smooth with bounded derivatives up to fourth order,
and that the velocity field $\vec{v}(\vec{x},t)$ is smooth with bounded derivatives up to
third order. If a third order method is used to integrate the characteristic equations,
then the presented gradient-augmented method with $\Delta t \propto h$
is consistent, with errors that are $O(h^4)$ in the level set function $\phi$,
and $O(h^3)$ in the gradient $\vec{\psi}$.
\end{thm}
\begin{proof}
One step of third order scheme is fourth order accurate is fourth order accurate
\begin{equation*}
\mathcal{X}(\vec{x},t+\Delta t,t)
= X(\vec{x},t+\Delta t,t)+O(\Delta t^4)\;.
\end{equation*}
Thus it yields a fourth order accurate approximation to the advected solution
\begin{equation}
\begin{split}
 & \abs{\phi(\mathcal{X}(\vec{x},t+\Delta t,t))-\phi(\vec{x},t+\Delta t)} \\
=& \abs{\phi(\mathcal{X}(\vec{x},t+\Delta t,t))-\phi(X(\vec{x},t+\Delta t,t))} \\
=& \abs{\nabla\phi(X(\vec{x},t+\Delta t,t))}\,O(\Delta t^4)\;.
\end{split}
\label{eq:error_advection_phi}
\end{equation}
The corresponding error in the gradient $\nabla\phi$ is obtained using the triangle
inequality, after adding and subtracting the Jacobian of the approximate advection step
times the gradient of $\phi$ at the exact characteristic at time $t$. For a better
transparency of the formulas, here we omit the arguments for
$\vec{X}$, $\vec{\mathcal{X}}$, $\nabla\vec{X}$, and $\nabla\vec{\mathcal{X}}$,
which are always evaluated at $(\vec{x},t+\Delta t,t)$.
\begin{equation}
\begin{split}
&\abs{\nabla\vec{\mathcal{X}}
\cdot\nabla\phi(\vec{\mathcal{X}},t)
-\nabla\phi(\vec{x},t+\Delta t)} \\
\le &
\abs{\nabla\vec{\mathcal{X}}\cdot\nabla\phi(\vec{\mathcal{X}},t)
-\nabla\vec{\mathcal{X}}\cdot\nabla\phi(\vec{X},t)}
+\abs{\nabla\vec{\mathcal{X}}\cdot\nabla\phi(\vec{X},t)
-\nabla\vec{X}\cdot\nabla\phi(\vec{X},t)} \\
= &
\abs{\nabla\vec{\mathcal{X}}\cdot
\prn{\nabla\phi(\vec{\mathcal{X}},t)-\nabla\phi(\vec{X},t)}}
+\abs{\prn{\nabla\vec{\mathcal{X}}-\nabla\vec{X}}\cdot\nabla\phi(\vec{X},t)} \\
\le &
\abs{\nabla\vec{\mathcal{X}}}\abs{D^2\phi(\vec{X})}
\underbrace{\abs{\vec{\mathcal{X}}-\vec{X}}}_{=O(\Delta t^4)}
+\underbrace{\abs{\nabla\prn{\vec{\mathcal{X}}-\vec{X}}}}_{=O(\Delta t^4)}
\abs{\nabla\phi(\vec{X},t)}
= O(\Delta t^4)\;.
\end{split}
\label{eq:error_advection_psi}
\end{equation}
Equation \eqref{eq:error_advection_psi} yields the error in the gradient, obtained
by a function $\phi$ that is defined everywhere. Due to Def.~\ref{def:superconsistency},
this is exactly the error on $\vec{\psi}$ in any superconsistent scheme.

As proved in Thm.~\ref{thm:p-cubic_accuracy}, the projection of a smooth function $\phi$
incurs a fourth order error in the function values, and a third order error in the
gradients
\begin{align}
\abs{P\phi(\vec{x})-\phi(\vec{x})} &= O(h^4)\;,
\label{eq:error_projection_phi} \\
\abs{\nabla(P\phi)(\vec{x})-\nabla\phi(\vec{x})} &= O(h^3)\;.
\label{eq:error_projection_psi}
\end{align}
Starting with the exact solution $\phi(\vec{x},t)$, we denote the approximately
advected function
\begin{equation*}
\phi^*(\vec{x}) = A_{t,t+\Delta t}\phi(\vec{x},t)
= \phi(\mathcal{X}(\vec{x},t+\Delta t,t),t)\;.
\end{equation*}
Followed by the projection operator, we obtain one step of the numerical scheme.
Adding and subtracting the approximately advected solution, using the
triangle inequality, and using the estimates \eqref{eq:error_advection_phi}
and \eqref{eq:error_projection_phi}, we obtain the error in function value
\begin{equation*}
\begin{split}
& \abs{(P\phi^*)(\vec{x})-\phi(\vec{x},t+\Delta t)} \\
\le &
\underbrace{\abs{(P\phi^*)(\vec{x})-\phi^*(\vec{x})}}_{=O(h^4)}
+\underbrace{\abs{\phi(\mathcal{X}(\vec{x},t+\Delta t,t),t)-\phi(\vec{x},t+\Delta t)}}
_{=O(\Delta t^4)}\;.
\end{split}
\end{equation*}
Hence, with $\Delta t \propto h$, one step is fourth order accurate in $\phi$.
Similarly, the estimates \eqref{eq:error_advection_psi}
and \eqref{eq:error_projection_psi} yield that one step of the numerical scheme is
third order accurate in $\vec{\psi}$.
\end{proof}

When taking multiple steps of the scheme, the approximate advection and projection
operators are applied iteratively. Taking a step from a function that approximates the
true solution with fourth order accuracy in function values and third order accuracy
in gradient values, exactly preserves the above accuracies
because Thm.~\ref{thm:p-cubic_accuracy} applies to the projection.
Thus the gradient-augmented CIR method is consistent, with a local truncation error
that is fourth order in $\phi$ and third order in $\vec{\psi}$.

\begin{rem}
Observe that for superconsistent schemes,
such as the one given in Example~\ref{ex:superconsistent_ShuOsher}, the approximate
advection itself preserves gradients with fourth order accuracy, while the projection
yields (and requires) only a third order accurate approximation of gradients.
Therefore, it can be more efficient to solve the gradient evolution
equation \eqref{eq:cir_evolution_psi} with a Runge-Kutta 2 scheme, while not sacrificing
accuracy.
\end{rem}

\begin{rem}
The orders of accuracy derived above are achieved for smooth functions with derivatives
up to fourth order. Note that level set functions are frequently signed distance
functions \eqref{eq:distance_function} and as such not differentiable along certain sets
of measure zero. In cells that contain such ``ridges'' of the level set function, the
accuracy of the scheme typically drops to first order. For structures that are multiple
grid cells wide, this drop in accuracy does not affect the evolution of the zero contour.
In contrast, structures of subgrid size may be evolved only with first order accuracy,
which is of course still better than losing them completely.
The situation shown in Fig.~\ref{fig:subgrid_1d_cubic} visualizes this effect.
Note further that even near ridges, the presented numerical scheme does not create
spurious oscillations in the first derivative. In 1D, this follows since the Hermite
cubic projection minimizes the $L^2$ norm of the second derivative
(see Rem.~\ref{rem:cubic_1d_minimize_l2}). Although this simple principle does not
transfer directly to 2D or 3D, similar stability properties as in the 1D case are
observed (see Sect.~\ref{sec:numerical_results}).
\end{rem}

\subsection{Stability}
\label{subsec:stability}
Here we investigate the stability of the method, i.e.~integration over a fixed time
interval $0 \leq t \leq T$, with $\Delta t \propto h \to 0$, does not lead
to a blow up of neither $\phi$ nor $\vec{\psi}$.
Proving stability for the gradient-augmented advection scheme is difficult, since the
application of the projection $P$ to a smooth function $\phi$ can both increase the
function values (because of the effect of the gradients at the grid points), and
the gradients (a cubic can have steeper slopes than a given smooth function with the
same values and gradients at the grid points). The following theorem shows the
stability of the method in the constant coefficient case in one space dimension.

\begin{thm}[stability]
The presented gradient-augmented method, considered in one space dimension and
for a constant velocity field, is stable.
\end{thm}
\begin{proof}
Consider the 1D advection equation $\phi_t+v\phi_x = 0$ with $v$ constant.
We choose $|v|\Delta t\le h$, and WLOG assume $v<0$.
Since the scheme is gradient-augmented, the state vector contains both function values
and derivatives. Let $\phi_j^n$ denote the approximate solution, and $\psi_j^n$ denote
the approximate derivative, both at the grid point $x = jh$ and
time $t = n\Delta t$.
One step of the generalized CIR scheme is obtained by applying \eqref{eq:Hermite_cubic_1D}
and \eqref{eq:Hermite_cubic_1D_prime}
\begin{equation}
\begin{split}
\phi_j^{n+1} &=
\phantom{\tfrac{1}{h}}\prn{\phi_j^n\, f(\xi)\, +\phi_{j+1}^n\, f(1-\xi)}\,
+h\prn{\psi_j^n\, g(\xi)\, -\psi_{j+1}^n\, g(1-\xi)}\;, \\
\psi_j^{n+1} &=
\tfrac{1}{h}\prn{\phi_j^n f'(\xi)-\phi_{j+1}^n f'(1-\xi)}
+\phantom{h}\prn{\psi_j^n g'(\xi)+\psi_{j+1}^n g'(1-\xi)}\;,
\end{split}
\label{eq:scheme_1D_constant}
\end{equation}
where $\xi = \tfrac{|v|\Delta t}{h} \in (0,1]$, while $f$ and $g$ are the basis
functions \eqref{eq:1-cubic_basis_functions}. Note that, since the velocity field is
constant, the gradient $\psi$ is also constant along the characteristics.
In the special case $\xi = 1$, the scheme becomes
\begin{equation*}
\phi_j^{n+1} = \phi_{j+1}^n
\quad\text{and}\quad
\psi_j^{n+1} = \psi_{j+1}^n\;,
\end{equation*}
i.e.~it is decoupled and exact, and thus stable.
Hence, in the following, we restrict to the case $0<\xi<1$.

Since we are investigating a linear partial differential equation with constant
coefficients, we can apply von Neumann stability analysis, and analyze the
stability of \eqref{eq:scheme_1D_constant} in Fourier space.
We consider the Fourier transform of the equations, and rewrite them in terms of the
Fourier coefficients $a_k^n$ and $b_k^n$, which are related to $\phi$ and $\psi$ by
\begin{equation}
\phi_j^n = \sum_k a_k^n e^{i k x_j}
\quad\text{and}\quad
\psi_j^n = \sum_k b_k^n e^{i k x_j}\;.
\label{eq:stability_Fourier_series}
\end{equation}
Substituting \eqref{eq:stability_Fourier_series} into \eqref{eq:scheme_1D_constant}
yields an update rule for the Fourier coefficients that is given by
a $2\times 2$ growth factor matrix as
\begin{equation*}
\begin{pmatrix} a_k^{n+1} \\ b_k^{n+1} \end{pmatrix}
= \underbrace{\begin{pmatrix}
f_0+e^{i\theta}f_1 & h\prn{g_0+e^{i\theta}g_1} \\
\tfrac{1}{h}\prn{f_0'+e^{i\theta}f_1'} & g_0'+e^{i\theta}g_1'
\end{pmatrix}}_{=\tilde{G}_{\xi,h}(\theta)}
\cdot \begin{pmatrix} a_k^n \\ b_k^n \end{pmatrix}\;,
\end{equation*}
where $\theta = kh$, we have used the notation
$f_0 = f(\xi)$, $f_1 = f(1-\xi)$, $g_0 = g(\xi)$, $g_1 = -g(1-\xi)$,
and the primes are derivatives with respect to $\xi$.

As we show below, the scheme is stable because:
\begin{enumerate}[(a)]
\item
For $\theta = 0$, one eigenvalue of
$\tilde{G}_{\xi,h}(\theta)$ is $\lambda_1 = 1$, while the other
one satisfies $ -1 < \lambda_2 < 1$.
\item
For $\theta \neq 0$, and $\abs{\theta} \leq \pi$, both
eigenvalues of $\tilde{G}_{\xi,h}(\theta)$ are strictly inside the
complex unit circle.
\end{enumerate}

In the case $\theta = 0$, the matrix $\tilde{G}_{\xi,h}(\theta)$ has the form
\begin{equation*}
\tilde{G}_{\xi,h}(\theta)
= \begin{pmatrix} 1 & g_0+g_1 \\ 0 & 1-6\xi(1-\xi) \end{pmatrix}\;.
\end{equation*}
Hence its eigenvalues are $\lambda_1 = 1$ and $\lambda_2 = 1-6\xi(1-\xi)$.
Clearly, $-1<\lambda_2<1$, provided $0<\xi<1$.

In the case $\theta\neq 0$, we can consider the matrix
\begin{equation*}
G_{\xi}(\theta)
= U^{-1}\cdot\tilde{G}_{\xi,h}(\theta)\cdot U
\quad\text{where}\quad
U = \begin{pmatrix} h & 0 \\ 0 & 1 \end{pmatrix}\;.
\end{equation*}
By construction, this new matrix, given by
\begin{equation*}
G_{\xi}(\theta)
= \begin{pmatrix}
f_0+e^{i\theta}f_1 & g_0+e^{i\theta}g_1 \\
f_0'+e^{i\theta}f_1' & g_0'+e^{i\theta}g_1'
\end{pmatrix}\;,
\end{equation*}
has the same eigenvalues as $\tilde{G}_{\xi,h}(\theta)$,
but it is simpler, since it does not depend on the mesh size $h$.

\begin{figure}
\centering
\begin{minipage}[t]{.24\textwidth}
\centering
\includegraphics[width=0.99\textwidth]{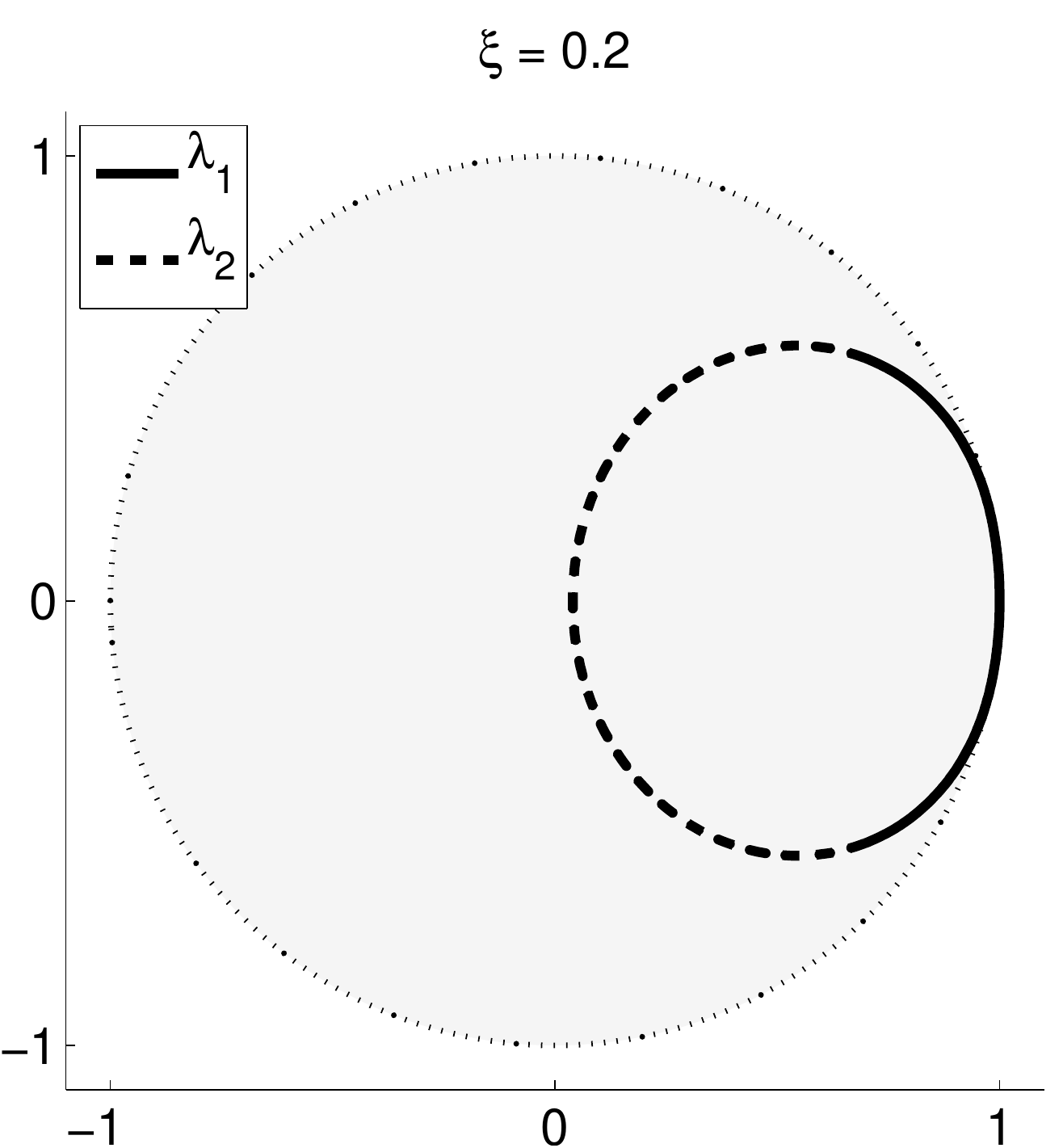}
\end{minipage}
\hfill
\begin{minipage}[t]{.24\textwidth}
\centering
\includegraphics[width=0.99\textwidth]{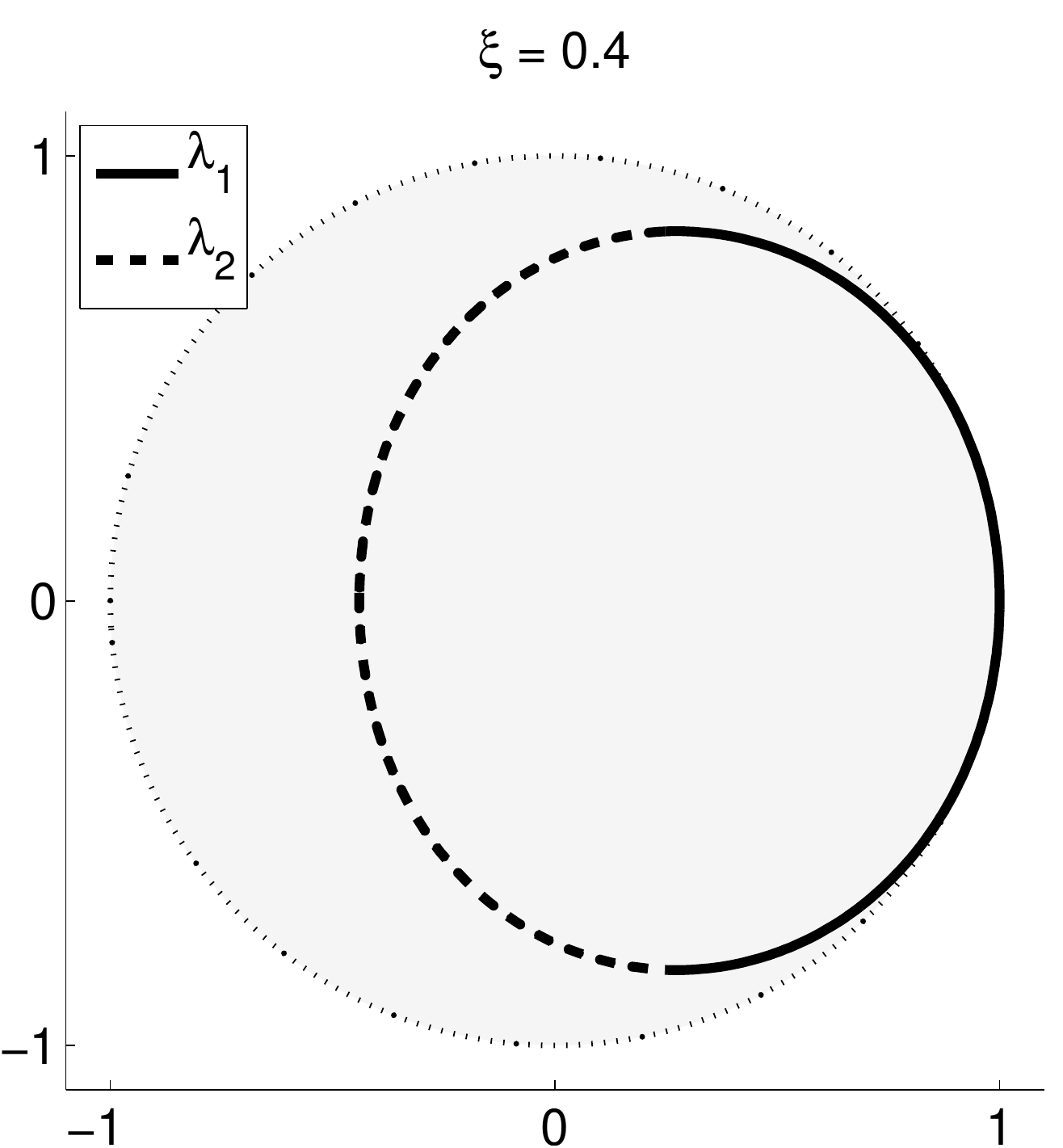}
\end{minipage}
\hfill
\begin{minipage}[t]{.24\textwidth}
\centering
\includegraphics[width=0.99\textwidth]{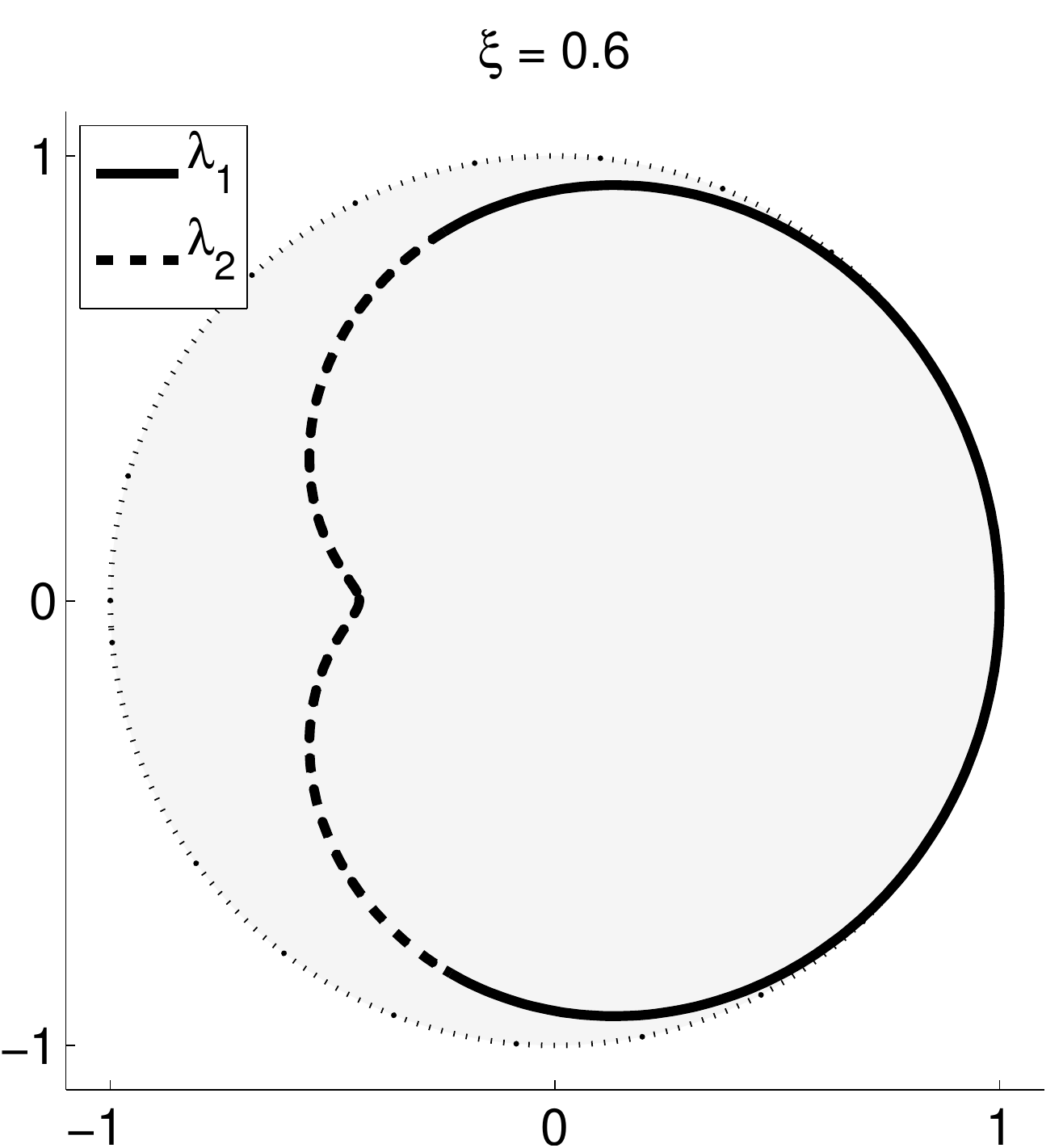}
\end{minipage}
\hfill
\begin{minipage}[t]{.24\textwidth}
\centering
\includegraphics[width=0.99\textwidth]{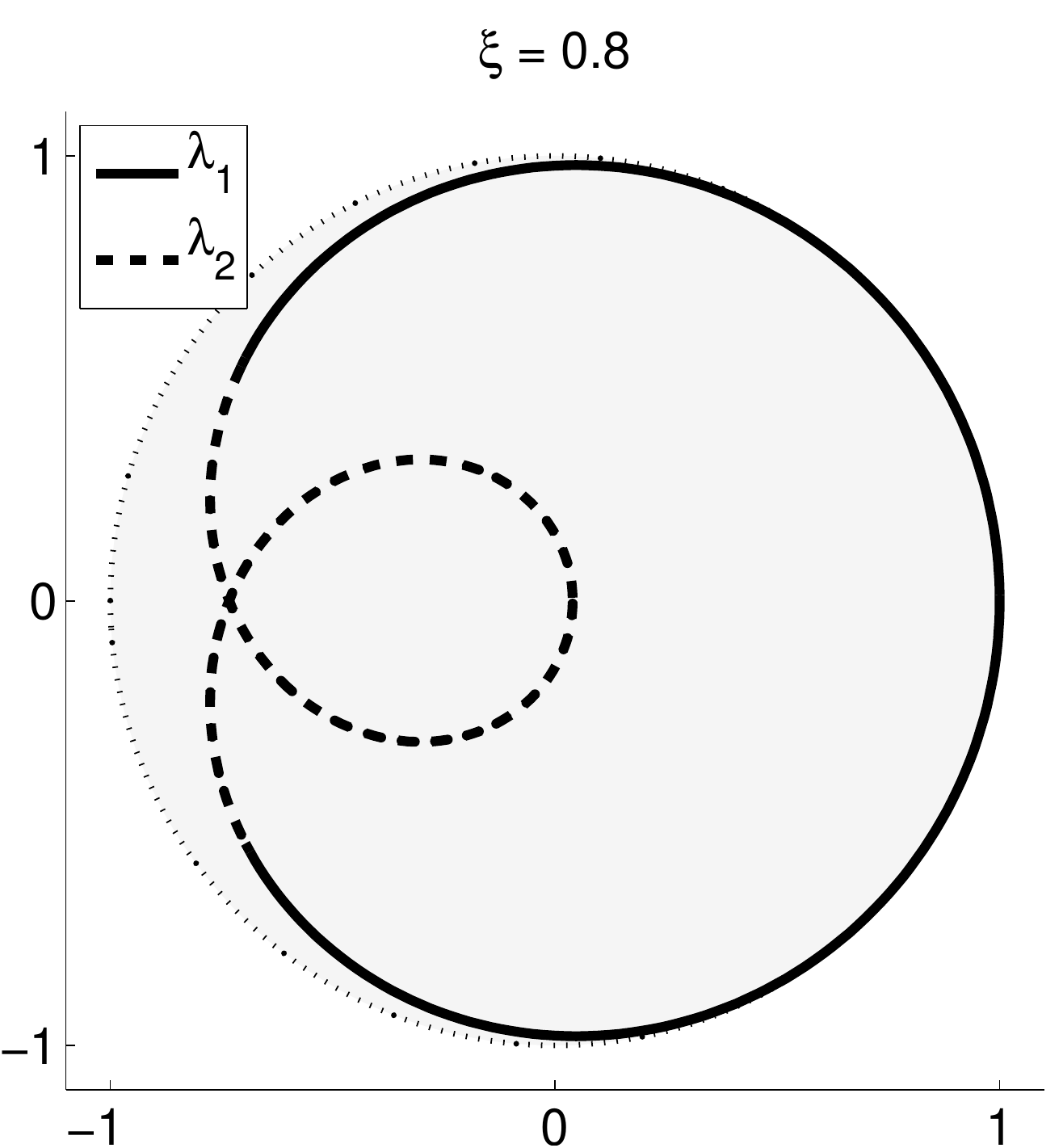}
\end{minipage}
\caption{Eigenvalues of the growth factor matrix $G_{\xi}(\theta)$ of the
cubic CIR scheme, for $\xi\in\{0.2,0.4,0.6,0.8\}$}
\label{fig:stability_growthfactor}
\end{figure}

At the current state, a strict formal estimate on the eigenvalues of $G_{\xi}(\theta)$
remains to be done. However, the eigenvalues can be evaluated numerically and plotted.
Fig.~\ref{fig:stability_growthfactor} shows the eigenvalue graphs
(as functions of $\theta\in [-\pi,\pi]$) of the growth factor matrix for
$\xi\in\{0.2,0.4,0.6,0.8\}$. Note that in each figure the curve consists of both
eigenvalues ($\lambda_1$ solid part, $\lambda_2$ dashed part).
The four figures indicate that for $\theta\neq 0$, both eigenvalues are always inside
the unit circle. This observation hold true for all values of $0<\xi<1$ that we have
tested. Thus the gradient-augmented $p$-cubic CIR method is stable.
\end{proof}

\begin{rem}
It is very plausible that the presented stability result carries over to the case of
variable coefficients. The reason is that the most usual way in which instabilities
arise is in the short wave limit, i.e.~oscillations on the scale of the grid
resolution $h$. Exactly those waves are covered by the above von Neumann analysis,
since any smooth velocity field is locally constant if the chosen resolution $h$ is
sufficiently small.
A general stability proof in higher space dimensions and for variable coefficients
is the subject of current research. In all numerical tests presented in
Sect.~\ref{sec:numerical_results}, and others, the method is observed to be stable.
\end{rem}

\begin{rem}
The prior considerations show that the scheme is stable, so that grid scale oscillations
are not amplified. In addition, it is interesting to see what happens with the
components that are well resolved by the grid, i.e. both $h$ and $\theta$ are small.
In this case, for $0<\xi<1$ fixed, as $h\to 0$ we can write
\begin{equation*}
\tilde{G}_{\xi,h}(\theta)
= \begin{pmatrix}
1 & 0 \\
i k 6\xi(1-\xi) & 1-6\xi(1-\xi)
\end{pmatrix}+O(h)\;.
\end{equation*}
This has the eigenvalue $\lambda_1 = 1+O(h)$, with corresponding eigenvector
$\vec{e}_1 = \prn{1,i k}^T+O(h)$, and the eigenvalue
$\lambda_2 = 1-6\xi(1-\xi)+O(h)$, which is fully inside the unit circle.

The higher order corrections to $\lambda_1$ give the advection at velocity $v$ up to
some error (as we know from consistency, see Sect.~\ref{subsec:consistency}).
Furthermore, since $\lambda_2$ is inside the unit circle, no matter how we start,
the solution will always be driven into a configuration where $\psi = ik\phi$.
This happens for all the ``resolved'' values of $k$, while the others decay,
as shown in the stability considerations.

Hence, the presented method drives the solution \emph{in the spectral sense} into
a configuration in which $\psi$ is the derivative of the $\phi$.
In other words, at least in this case of constant $v$, we have a very strong form
of coherence being enforced by the method. Our conjecture is that this
property carries over to the general case.
\end{rem}

\subsection{Convergence}
\label{subsec:convergence}
With the $p$-cubic projection \eqref{eq:p-cubic_interpolant}, the presented numerical
scheme is linear.
Therefore, due to the Lax equivalence theorem \cite{LaxRichtmyer1956}, consistency
(shown in Sect.~\ref{subsec:consistency}), and stability
(investigated in Sect.~\ref{subsec:stability})) imply convergence, i.e.~the numerical
approximation converges to the true solution as $\Delta t \propto h \to 0$,
provided an appropriate CFL condition is enforced
(i.e.~$C\Delta t < h$, where the constant $C$ depends on the velocity field).

As shown in Sect.~\ref{subsec:consistency}, a fixed number of time steps yield
a fourth order accurate approximation to $\phi$, and the third order accurate
approximation to $\vec{\psi}$.
In order to compute the solution over a fixed time interval $[0,T]$,
a total of $\frac{T}{\Delta t}$ time steps are required. Hence the global error can
only be guaranteed to be third order accurate in $\phi$ and second order accurate
in $\vec{\psi}$.
The numerical results in Sect.~\ref{sec:numerical_results} indicate that this
drop by one order is in fact what happens.

\section{Numerical Results}
\label{sec:numerical_results}
In this section we test the accuracy and performance of the gradient-augmented
level set method, as presented in the prior sections. We use the superconsistent
scheme given in Ex.~\ref{ex:superconsistent_ShuOsher}, which is based on the
Shu-Osher RK3 method. Note that in all presented examples, the use of
Heun's method to update gradients (see Sect.~\ref{sec:scheme_CIR}) leads to
results that differ by less than $0.1\%$.

The local and global accuracy of the generalized CIR method, as theoretically
predicted in Sect.~\ref{sec:numerical_analysis}, as well as the accuracy of the
approximation of the curvature from the $p$-cubic interpolant,
outlined in Sect.~\ref{subsec:derivative_quantities}, are investigated
in Sect.~\ref{subsec:numerics_accuracy}.
In addition, the performance of the gradient-augmented level set method is compared
with a classical high order level set approach. The results of various benchmark
tests are presented in Sect.~\ref{subsec:numerics_performance}.

\subsection{Accuracy of the Advection and Curvature Approximations}
\label{subsec:numerics_accuracy}
We numerically investigate the accuracy of the gradient-augmented CIR method advection
scheme, as presented in Sect.~\ref{sec:scheme_CIR} and analyzed in
Sect.~\ref{sec:numerical_analysis}, as well as the accuracy of curvature when recovered
from the $p$-cubic interpolant, as described in Sect.~\ref{subsec:derivative_quantities}.
The order of convergence of these approximations is based on a Taylor expansion, and
thus requires the considered functions to be smooth. In practice, level set
functions are often chosen as signed distance functions \eqref{eq:distance_function},
which possess discontinuous
gradients. The rationale is that the signed distance property generally yields more
benefits than the jumps in the gradient cause disadvantages. In fact, if the
represented structures are sufficiently large, the drop in accuracy of numerical
approximations is typically located only near the discontinuous gradients,
while near the zero contour, the full accuracy for smooth functions is observed.
In order to investigate the order of accuracy of the presented approaches rigorously,
here in Sect.~\ref{subsec:numerics_accuracy}, we consider the evolution and
differentiation of infinitely often differentiable functions.

In the following, we measure the error of the numerical scheme.
At some time, let the true solution be denoted by $\phi$, and the numerical
approximation at a grid point $\vec{x}_j$ be denoted by $\hat{\phi}_j$ for function
values, and $\hat{\vec{\psi}}_j$ for gradients.
The numerical error of the gradient-augmented scheme is computed in the
$L^{\infty}$-norm, both for function values and for gradients
\begin{equation}
\begin{split}
e_{\text{max}}(\phi) &= \max_j\abs{\hat{\phi}_j-\phi(\vec{x}_j)}\;, \\
e_{\text{max}}(\vec{\psi})
&= \max_j\norm{\hat{\vec{\psi}}_j-\nabla\phi(\vec{x}_j)}_{\text{max}}\;,
\label{eq:maximum_norm}
\end{split}
\end{equation}
where $\norm{\vec{v}}_{\text{max}} = \max_i\abs{v_i}$.

\subsubsection{Local Truncation Error in a Pseudo-1D Advection}
In the 2D domain $[0,1]\times [0,1]$, we consider the advection \eqref{eq:advection_phi}
of the smooth initial conditions
\begin{equation*}
\mathring{\phi}(x,y) = \exp\prn{-\abs{\vec{x}-\vec{x}_0}^2}-\exp\prn{-r_0^2}\;,
\end{equation*}
with $\vec{x} = (x,y)$, $\vec{x}_0 = (0.5,0.5)$, and $r_0 = 0.15$, under the velocity
field
\begin{equation}
\vec{v}(x,y)=\tfrac{1}{\sqrt{2+\pi}}
\exp\prn{\tfrac{\sqrt{2}x+\sqrt{\pi}y}{\sqrt{2+\pi}}}
\begin{pmatrix} \sqrt{2} \\ \sqrt{\pi} \end{pmatrix}\;.
\label{eq:velocity_1}
\end{equation}
This velocity field represents a 1D flow in a constant direction that is not aligned with
the 2D grid. Hence, we call it a ``pseudo-1D flow''. The velocity field possesses a
nonzero divergence, and a non-constant deformation $\nabla\vec{v}$. It is selected
because the advection under \eqref{eq:velocity_1} possesses a simple analytical solution.
At the inflow edges $\vec{v}\cdot\vec{n}<0$, homogeneous Neumann boundary conditions
are prescribed. However, we evaluate the error only on a part of the domain that is
not influenced by the boundary conditions.
We consider the numerical solution with the gradient-augmented CIR method, presented
in Sect.~\ref{sec:scheme_CIR}.

\begin{figure}
\centering
\begin{minipage}[t]{.49\textwidth}
\centering
\includegraphics[width=0.99\textwidth]{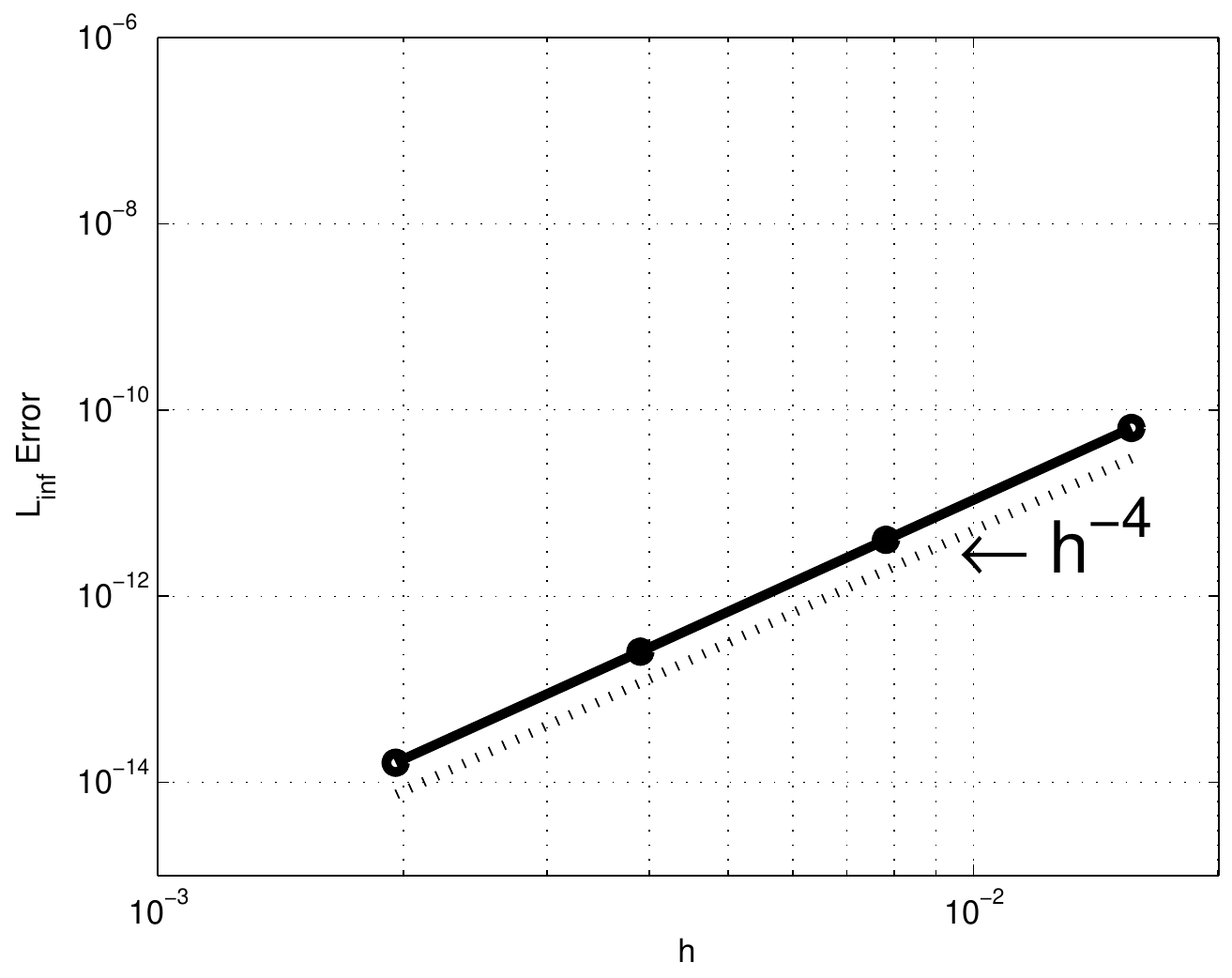}
\caption{Local truncation error for $\phi$ -- Pseudo 1D advection}
\label{fig:LTE_phi}
\end{minipage}
\hfill
\begin{minipage}[t]{.49\textwidth}
\centering
\includegraphics[width=0.99\textwidth]{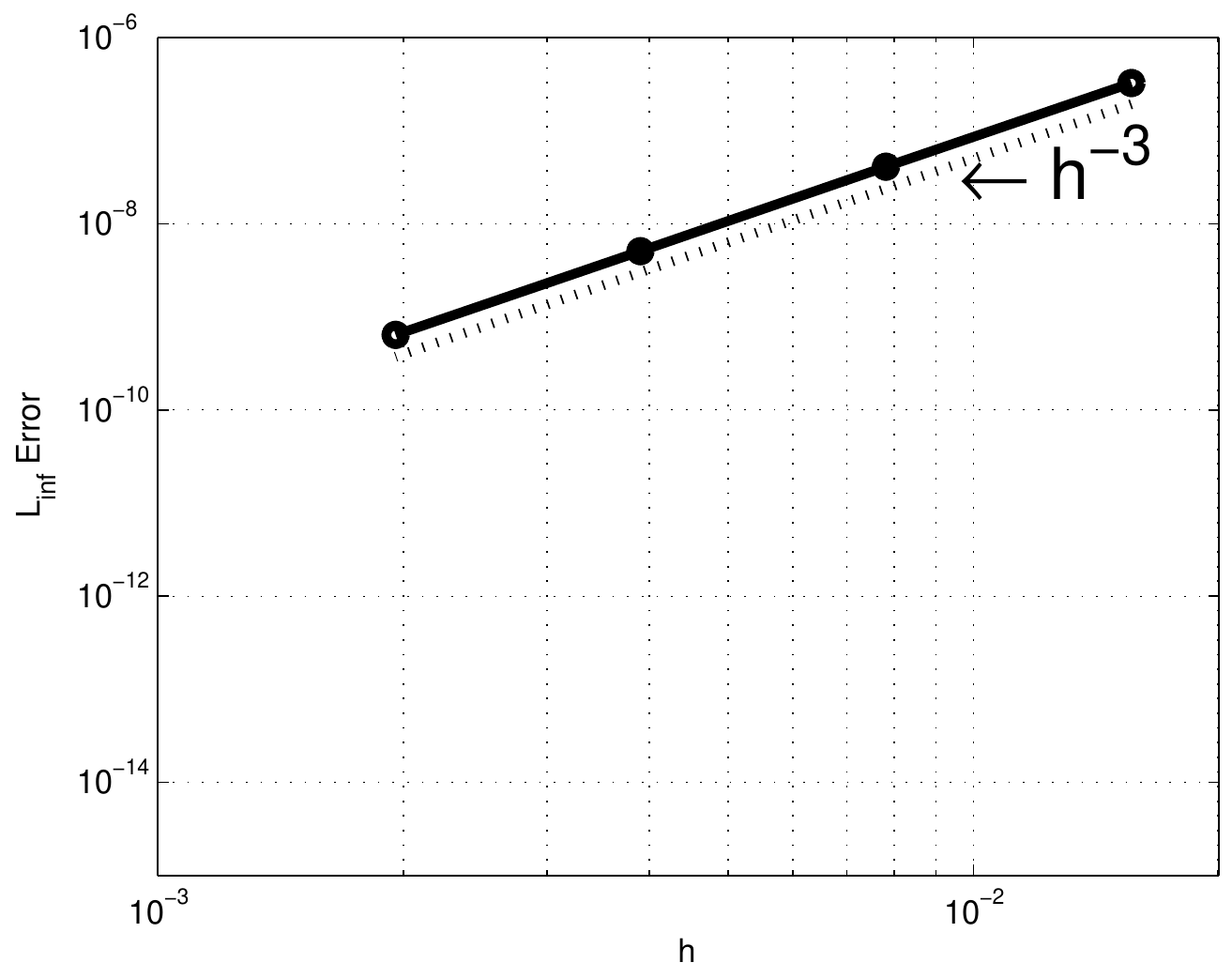}
\caption{Local truncation error for $\nabla\phi$ -- Pseudo 1D advection}
\label{fig:LTE_psi}
\end{minipage}
\end{figure}

The test aims at estimating the local truncation error.
For a sequence of regular grids, with spacing $h = \Delta x = \Delta y$ and time step
$\Delta t = \frac{1}{2}h$, we perform a fixed number ($N=16$) of steps with the
advection scheme.

The convergence of the local truncation error \eqref{eq:maximum_norm} is presented in
Figs.~\ref{fig:LTE_phi} and \ref{fig:LTE_psi}. We observe fourth order accuracy for
the function value $\phi$, and third order accuracy for the gradients $\nabla\phi$.
These results are in agreement with the local truncation error estimates derived
in Sect.~\ref{subsec:consistency}.

\subsubsection{Global Truncation Error in a 2D Deformation Field}
\label{subsubsec:vortex_in_a_box}
In this section we present the results obtained for a time-modulated, two-dimensional
deformation field. This test is often times referred to as ``vortex in a box'' flow,
following LeVeque \cite{LeVeque1996}, and Bell, Colella, and Glaz \cite{BellColellaGlaz1989}.
On the domain $[0,1]\times [0,1]$, we consider the velocity field
\begin{equation}
\vec{v}(x,y,t) = \nabla^{\bot}\varphi(x,y,t) = \prn{-\pd{\varphi}{y},\pd{\varphi}{x}}\;,
\label{eq:swirl_2d_velocity_field}
\end{equation}
given by the stream function
\begin{equation*}
\varphi(x,y,t) = \tfrac{1}{\pi}\cos\prn{\tfrac{\pi t}{T}}
\sin\prn{\pi x}^{2}\sin\prn{\pi y}^{2}\;.
\end{equation*}
This generates a time-dependent incompressible velocity field with a deformation
matrix $\nabla\vec{v}$ that changes both in space and in time. Thus this flow field
provides a realistic test for our method. Since the velocity field is zero at the
domain boundaries, no boundary conditions need to be prescribed.
Notice that the function $\cos\prn{\tfrac{\pi t}{T}}$ is an odd function with respect
to $t = T/2$, thus the advection under \eqref{eq:swirl_2d_velocity_field} is
anti-symmetric around $t = T/2$. In particular, $\phi(\vec{x},T) = \phi(\vec{x},0)$,
i.e.~at time $t = T$, the flow has brought back the level set function to its initial
values. Thus, we can check the error behavior of the numerical method at $t = T$.

We consider the advection \eqref{eq:advection_phi} of the smooth initial conditions
\begin{equation*}
\mathring{\phi} = \exp\prn{-\abs{\vec{x}-\vec{x}_0}^2}-\exp\prn{-r_0^2}\;,
\end{equation*}
with $\vec{x} = (x,y)$, $\vec{x}_0 = (0.5,0.75)$, and $r_0 = 0.15$, under the given
velocity field, and apply the gradient-augmented level set method.

Here, we use $\Delta t = h = \Delta x = \Delta y$, and $T = 2$.
We compute the $L^\infty$ norm of the difference between the computed solution
at $t = T$ and the initial conditions. This provides the global truncation error.
In Figs.~\ref{fig:GTE_swirl_analytical_phi}
and \ref{fig:GTE_swirl_analytical_psi}, the convergence of the error in the
$L^{\infty}$-norm for both $\phi$ and $\nabla\phi$ is shown. We observe third order
accuracy for the function value $\phi$, and second order accuracy for the gradients
$\nabla\phi$, as predicted in Sect.~\ref{sec:numerical_analysis}.

\begin{figure}
\centering
\begin{minipage}[t]{.49\textwidth}
\centering
\includegraphics[width=0.99\textwidth]{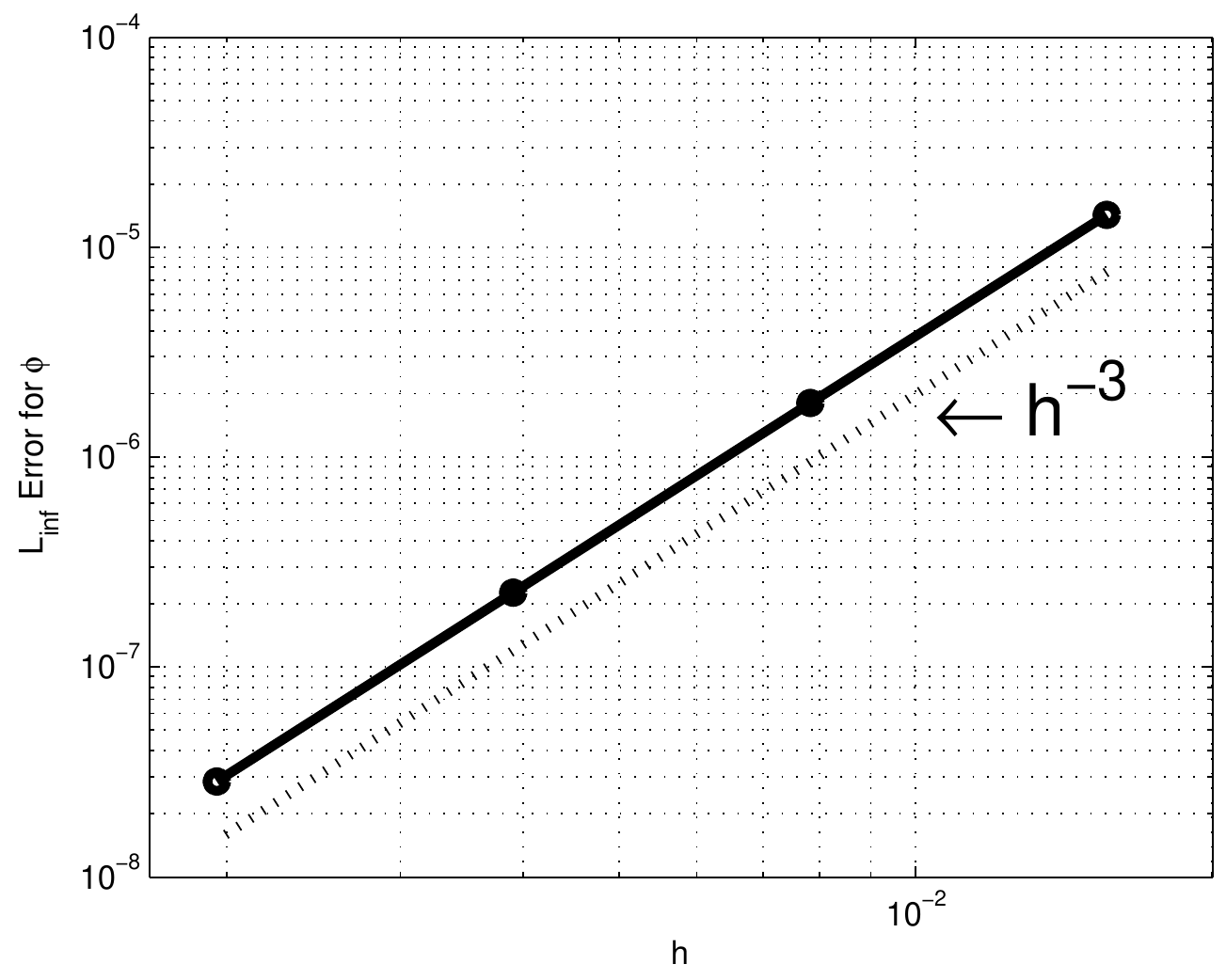}
\caption{Global truncation error for $\phi$ -- 2D deformation field}
\label{fig:GTE_swirl_analytical_phi}
\end{minipage}
\hfill
\begin{minipage}[t]{.49\textwidth}
\centering
\includegraphics[width=0.99\textwidth]{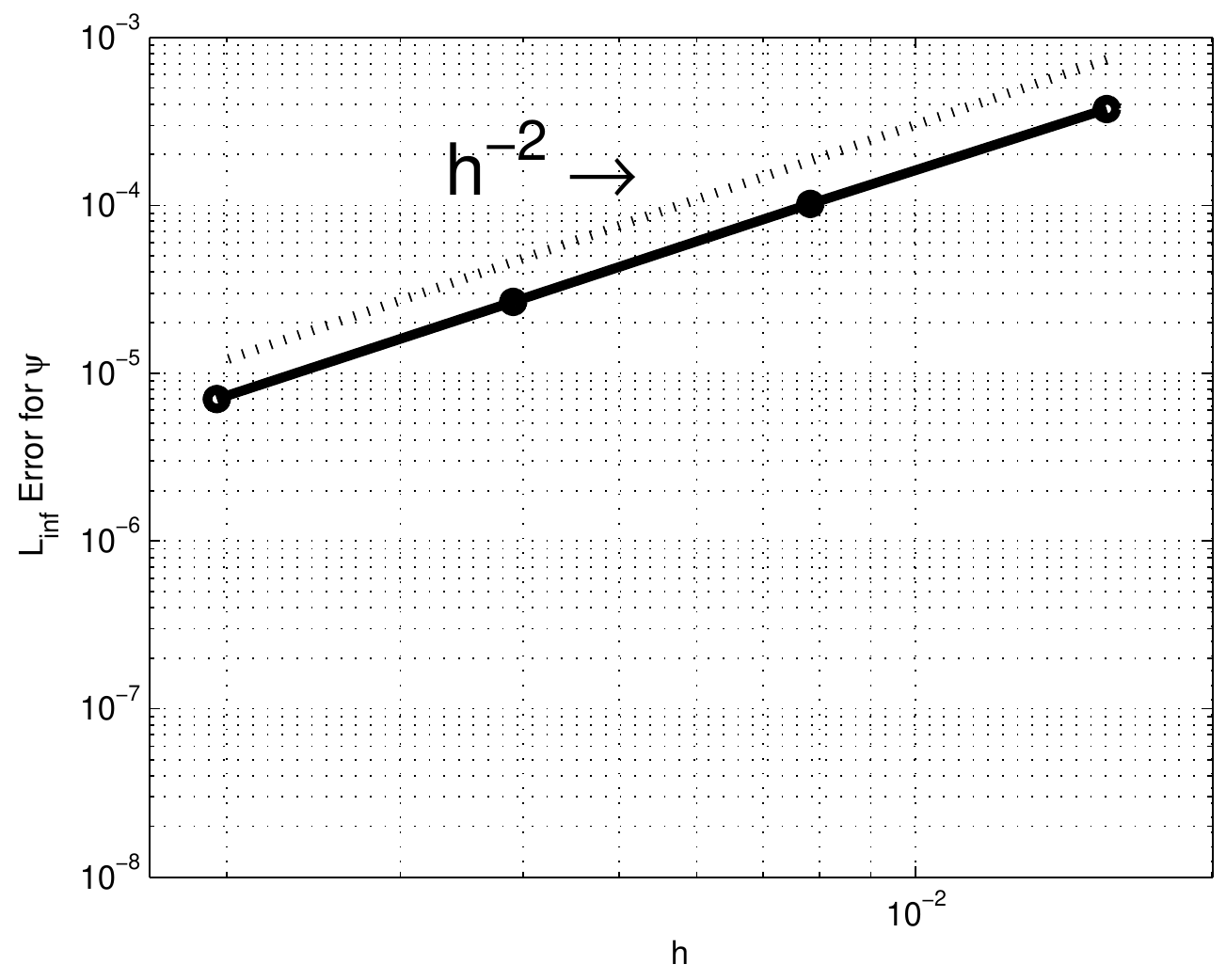}
\caption{Global truncation error for $\nabla\phi$ -- 2D deformation field}
\label{fig:GTE_swirl_analytical_psi}
\end{minipage}
\end{figure}

\subsubsection{Accuracy of Curvature}
\label{subsubsec:numerics_curvature}
We numerically investigate the accuracy of the curvature approximation, when recovering
it by differentiating the $p$-cubic interpolant.
On the domain $[0,1]\times [0,1]$, we consider the function
\begin{equation*}
\phi(x,y) = \prn{(x-2)(y-x)}^3\;.
\end{equation*}
This function possesses enough non-zero derivatives (and cross derivatives) to provide
a non-trivial test. Furthermore, it has no curvature singularities inside the
considered domain.

For a given sequence of grids with various grid spacing $h$, we evaluate
both $\phi$ and $\nabla\phi$ on the grid points. As described in
Sect.~\ref{subsec:derivative_quantities}, an approximation to the curvature is obtained
from this data, using equation \eqref{eq:curvature}, everywhere inside the domain.
Here, the function values and all required derivatives are evaluated analytically from
the $p$-cubic interpolation (which involves the reconstruction of the cross derivatives
from $\nabla\phi$, as presented in Sect.~\ref{subsec:Hermite_interpolant}).
We obtain an estimate of the accuracy of the curvature calculation over the whole domain
(not just at grid points), by computing the $L^{\infty}$-error on a fixed grid with a
much finer grid spacing.

\begin{figure}
\centering
\begin{minipage}[t]{.49\textwidth}
\centering
\includegraphics[width=0.99\textwidth]{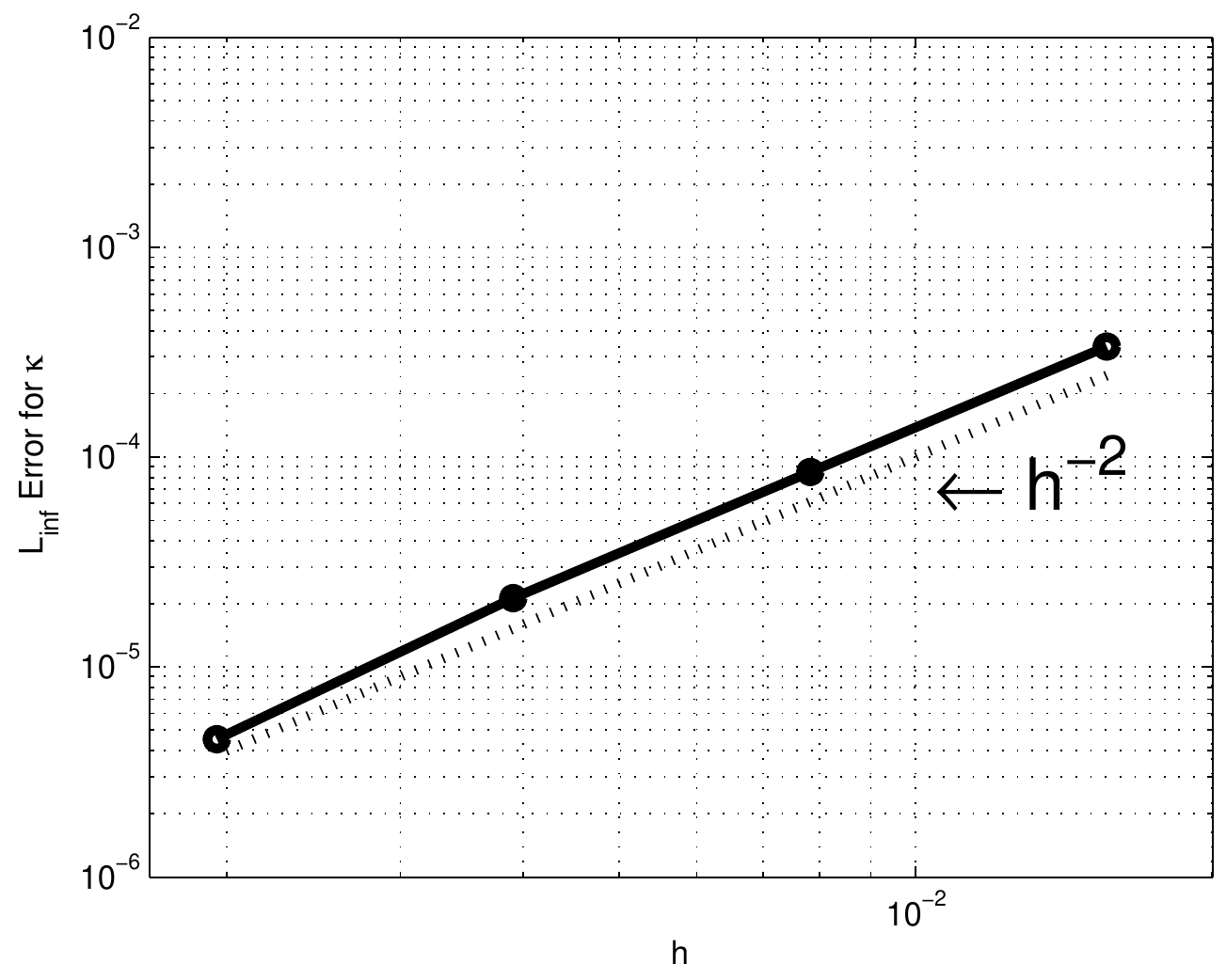}
\caption{Convergence plot for curvature $\kappa(x,y)$,
recovered from the Hermite bi-cubic}
\label{fig:curvature_no_advection}
\end{minipage}
\end{figure}

Convergence results are presented in Fig.~\ref{fig:curvature_no_advection}.
We observe that the curvature is obtained with second order accuracy everywhere in the
domain, which is in agreement with the theoretical considerations in
Sect.~\ref{subsec:derivative_quantities}.

\subsection{Performance of the Gradient-Augmented Level Set Method}
\label{subsec:numerics_performance}
In this section, we compare the performance of the gradient-augmented level set method
with that of a classical level set approach. For all test cases, we represent the initial
surface by a signed distance function \eqref{eq:distance_function}. For the classical
level set method, the advection equation \eqref{eq:advection_phi} is solved using
a fifth order WENO scheme \cite{LiuOsherChan1994} for the spacial approximation,
and the Shu-Osher scheme (a three stage, third order accurate, strongly stability
preserving Runge-Kutta method) \cite{ShuOsher1988} for the time step. In addition, in
each time step, the reinitialization equation \eqref{eq:reinitialization} is solved to
preserve the signed distance property \eqref{eq:distance_function} approximately.
We empirically find that for the presented test cases, the classical level approach
yields best results, when after each advection step of size $\Delta t$, two
reinitialization steps, each of size $0.75\,h$ are performed.
The gradient-augmented level set approach is applied as described in
Sect.~\ref{sec:scheme_CIR}. Here, no reinitialization is applied. Hence, for the
2D and 3D deformation field tests, the velocity field yields significant deformations
of the level set function away from a signed distance function.

\subsubsection{Zalesak's Circle}
We consider the rigid body rotation of Zalesak's circle \cite{Zalesak1979}
in a constant vorticity velocity field.
On the domain $[0,100]\times [0,100]$, let the initial data describe a slotted circle,
centered at $(50,75)$ with a radius of $15$, a slot width of $5$,
and a slot length of $25$. The constant vorticity velocity field is given by
\begin{align*}
u(x,y) &= \tfrac{\pi}{314}(50-y)\;, \\
v(x,y) &= \tfrac{\pi}{314}(x-50)\;.
\end{align*}
The disk completes one revolution in a time interval $0 \leq t \leq 628$.
At the inflow edges $\vec{v}\cdot\vec{n}<0$, homogeneous Neumann boundary conditions
are prescribed.

On a $64\times 64$ grid, we compare the gradient-augmented CIR scheme with the
classical WENO advection scheme with reinitialization.
The time step is $\Delta t = 1$, hence one revolution equals 628 time steps.
Fig.~\ref{fig:zalesak_circle} shows the evolution of the solution in time.
The top figure shows the initial conditions, and the velocity field.
The middle figure shows the obtained surface after one revolution,
and the bottom figure shows the results after four revolutions.
One can observe that the gradient-augmented level set method recovers the shape
significantly more accurately than the classical WENO scheme. In particular, after
four revolutions, with the classical approach the notch in the circle has vanished.
In contrast, with gradient, it has shrunk, yet it is still present.

\begin{figure}
\centering
\begin{minipage}[t]{.48\textwidth}
\centering
\includegraphics[height=0.92\textwidth]{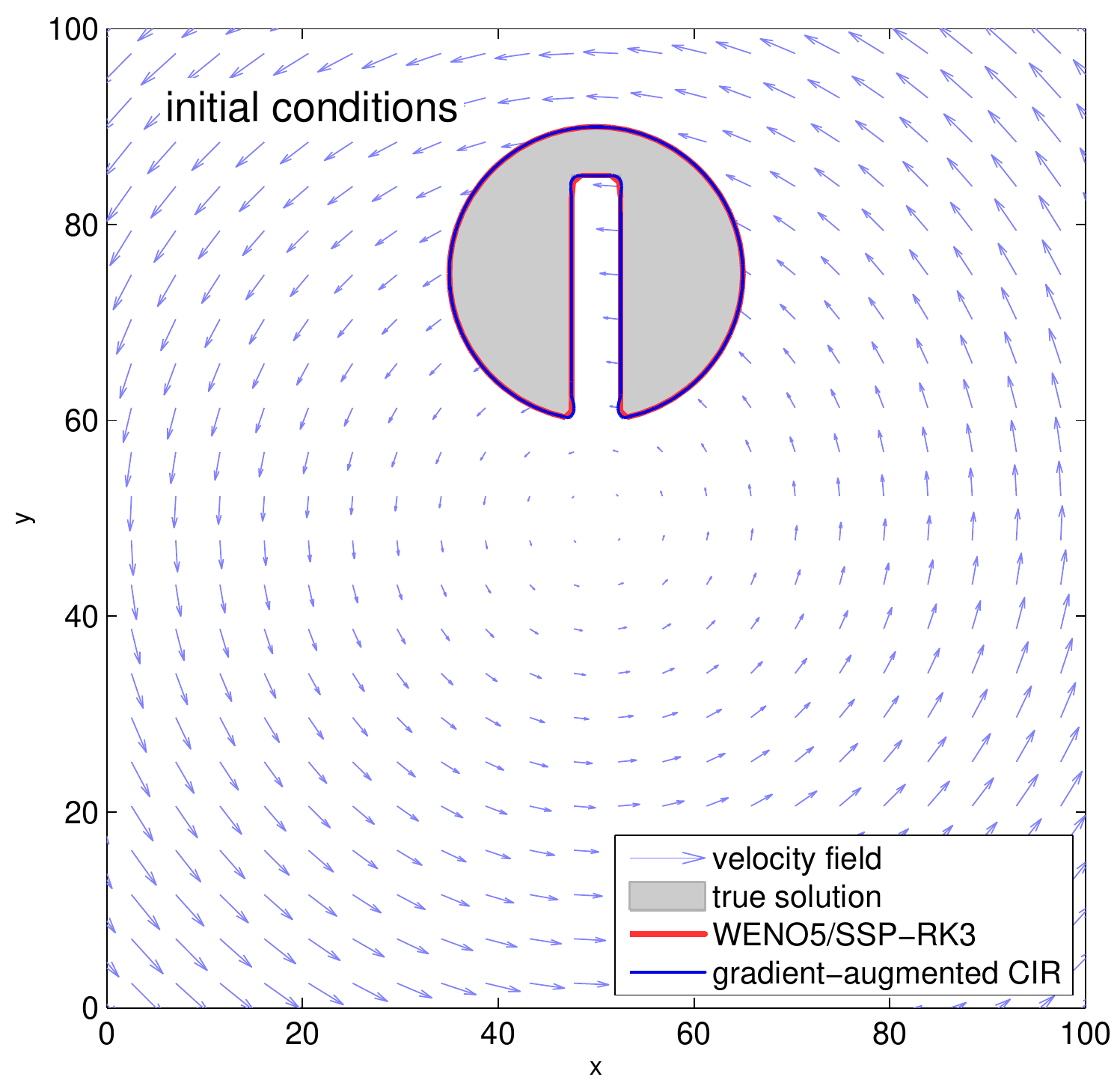} \\[.1em]
\includegraphics[height=0.92\textwidth]{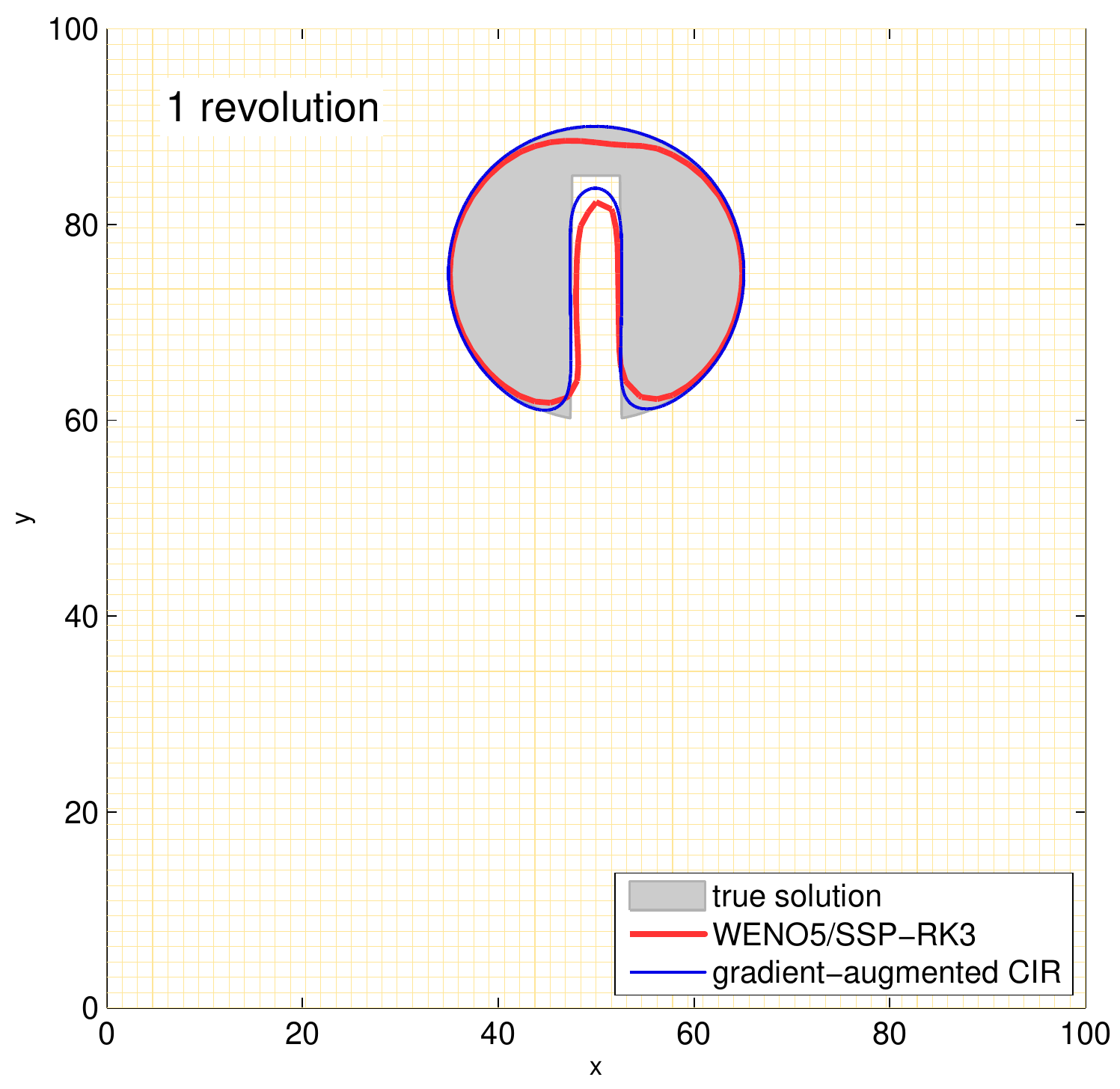} \\[.1em]
\includegraphics[height=0.92\textwidth]{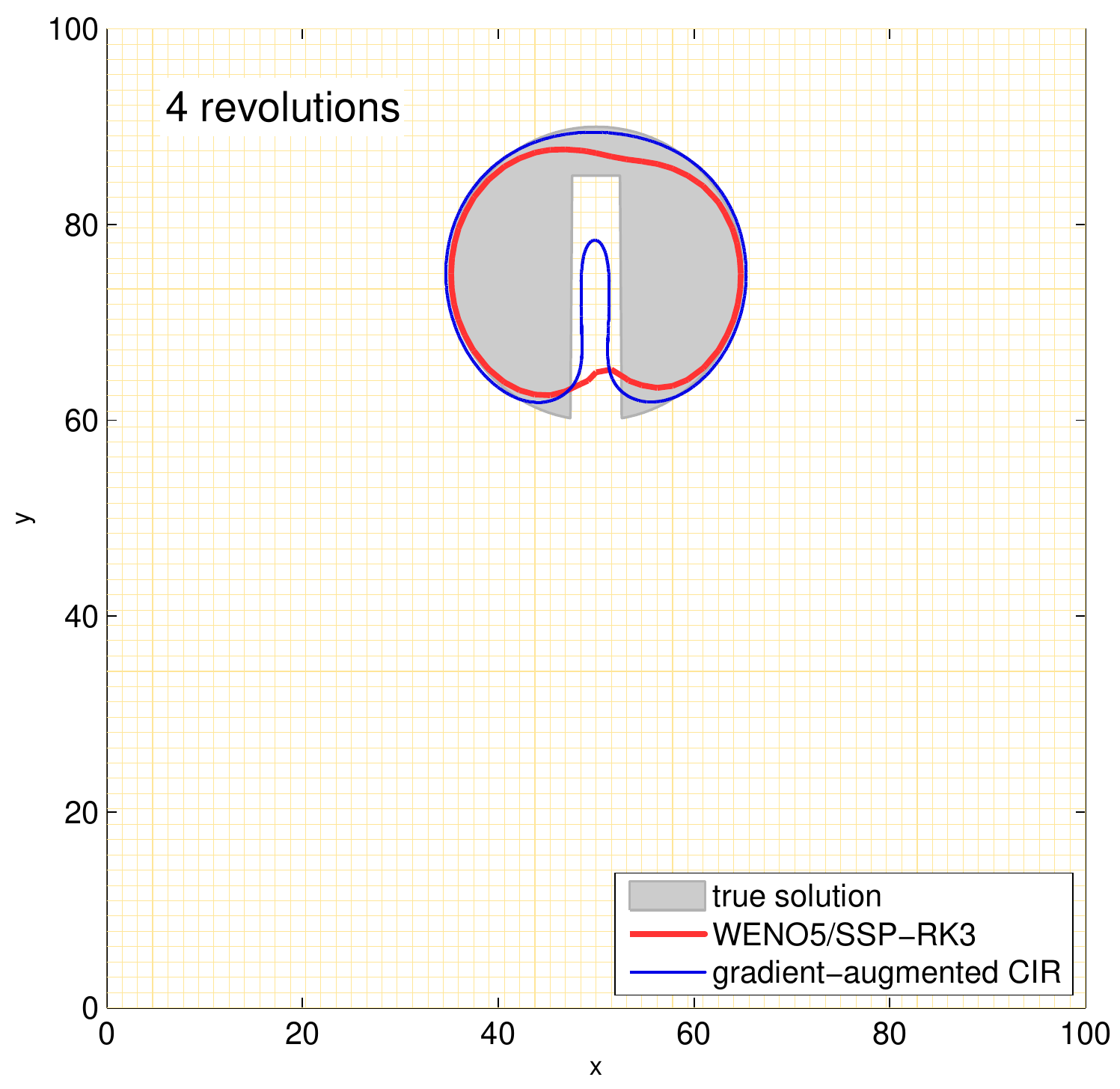}
\caption{Zalesak circle initially, after one, and after four revolutions --
classical approach (blue) vs.~gradient-augmented method (red)}
\label{fig:zalesak_circle}
\end{minipage}
\hfill
\begin{minipage}[t]{.48\textwidth}
\centering
\includegraphics[height=0.92\textwidth]{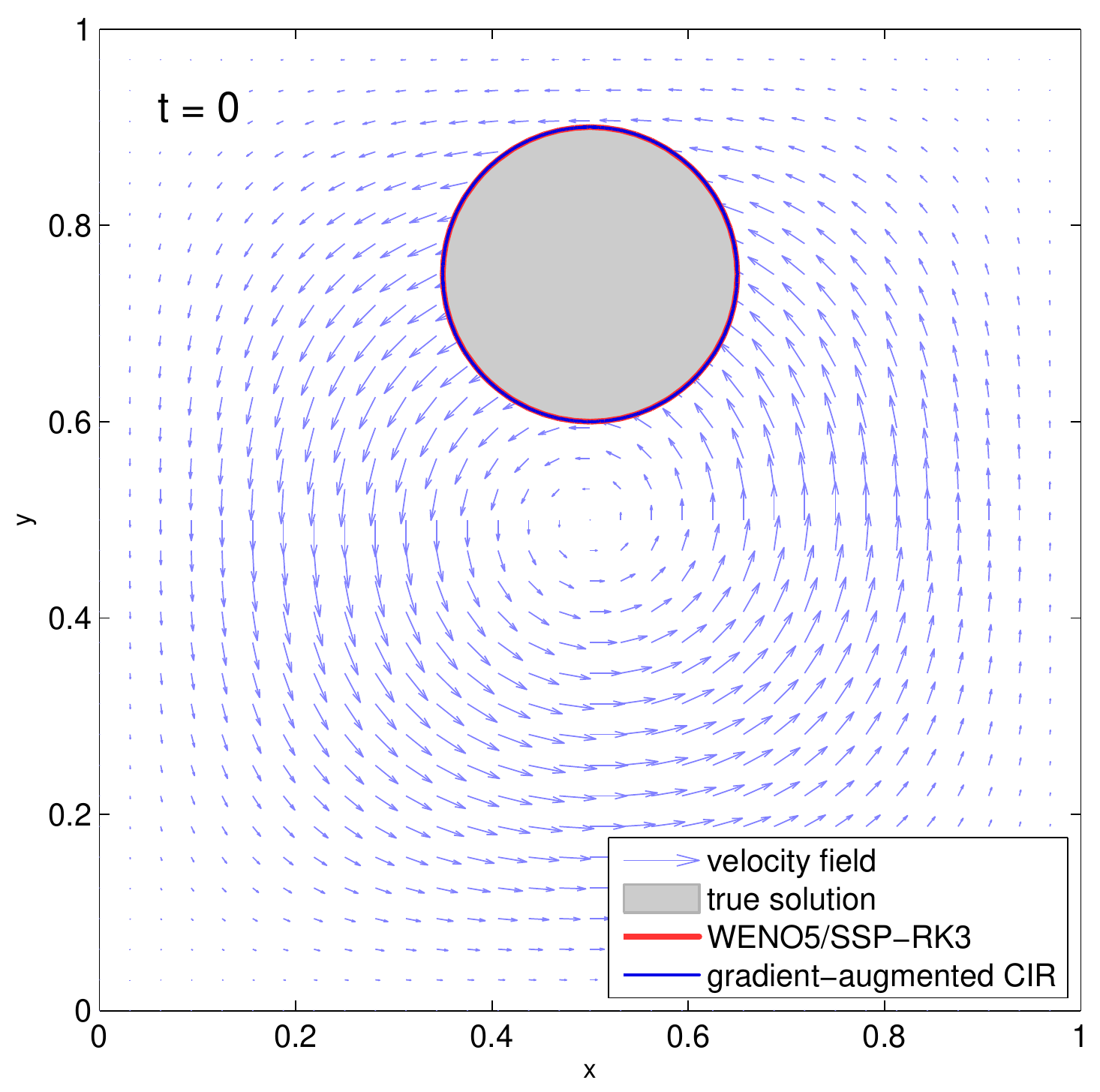} \\[.1em]
\includegraphics[height=0.92\textwidth]{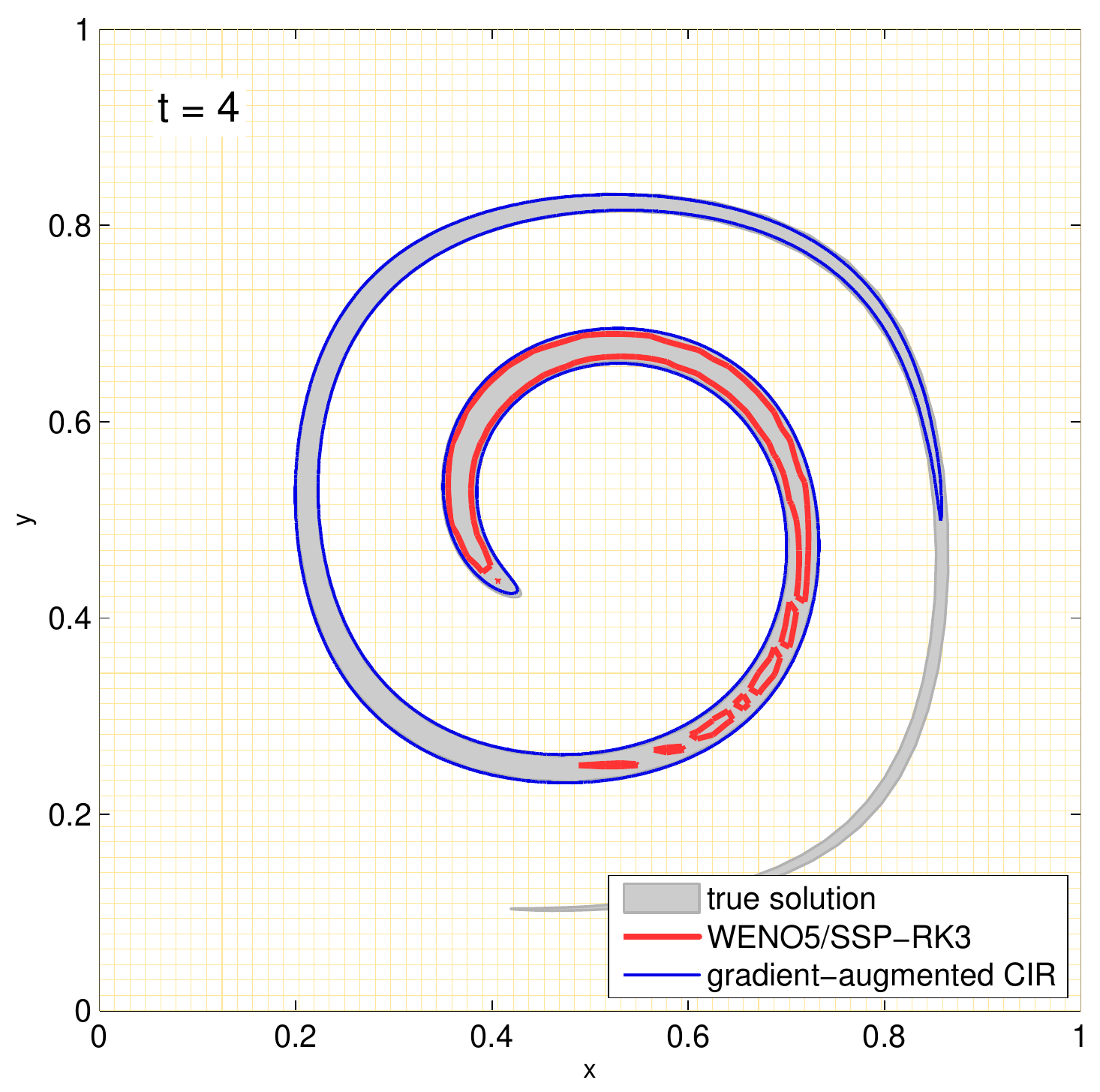} \\[.1em]
\includegraphics[height=0.92\textwidth]{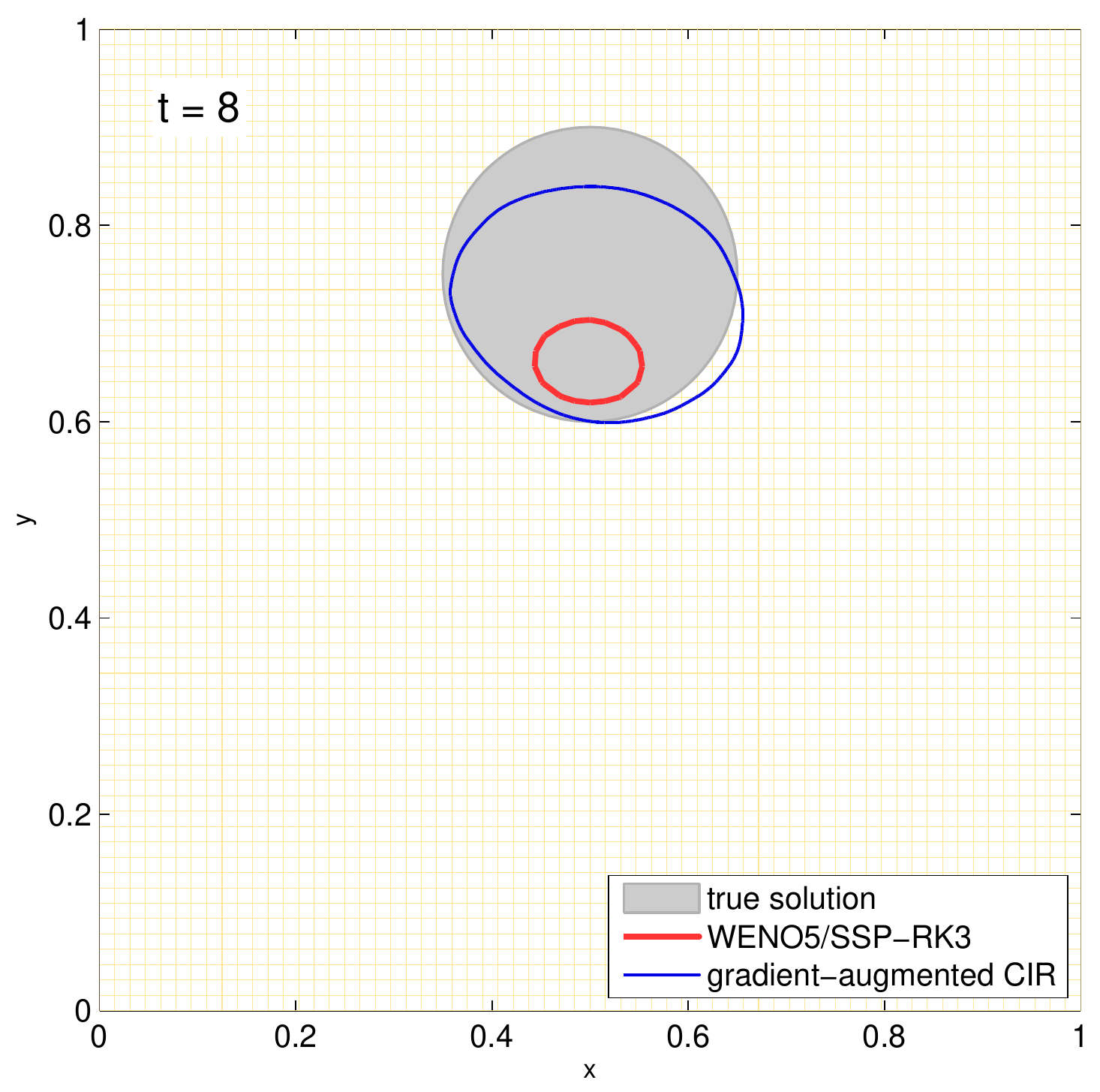}
\caption{2D deformation field test at $t\in\{0,4,8\}$ --
classical approach (blue) vs.~gradient-augmented method (red)}
\label{fig:swirl_2d}
\end{minipage}
\end{figure}

\subsubsection{2D Deformation Field}
\label{subsubsec:numerics_swirl_2d}
We consider the 2D velocity field described in Sect.~\ref{subsubsec:vortex_in_a_box},
with $T = 8$. On a $64\times 64$ grid, we compare the gradient-augmented CIR scheme with
the classical WENO advection scheme with reinitialization.
The time step is $\Delta t = \Delta x$.

Fig.~\ref{fig:swirl_2d} shows the evolution of the solution in time.
The top figure shows the initial conditions, and the velocity field.
The middle figure shows the obtained surface at $t = 4 = T/2$, which is the
time of maximal deformation of the zero contour. Observe that the true solution
has the surface swirled around one-and-a-half times, and the structure's thickness
gets close (or even below) the grid size.
From the analysis in Sect.~\ref{subsec:small_structures} we expect the
gradient-augmented level set method to perform better at representing the thin
structure. The results verify our expectation: the classical WENO scheme loses a
full revolution on the receding tail of the structure. In addition, the surface breaks
up into multiple components. In contrast, the gradient-augmented level set approach
recovers one connected structure, which captures the true shape very well at the
trailing front. Of course, even with gradients, there are limits to the subgrid
resolution, and here, still a quarter revolution of the tail is missing. Nevertheless,
the results are striking, considering that the grid is relatively coarse.
The bottom figure in Fig.~\ref{fig:swirl_2d} shows the results at $t = 8 = T$. At this
time, the flow has brought the structure back to its initial configuration. The
structure evolved with the WENO scheme has lost a remarkable amount of mass.
In contrast, the mass loss obtained with the gradient-augmented scheme is significantly
smaller.

\subsubsection{Zalesak's Sphere}
On the domain $[0,100]\times [0,100]\times [0,100]$,
a three dimensional slotted sphere of radius $15$, initially centered
at $(50,75,50)$, with a slot width of $5$ and slot depth of $25$ is
is rotated under the velocity field
\begin{align*}
u(x,y,z) &= \tfrac{\pi}{314}(50-y)\;, \\
v(x,y,z) &= \tfrac{\pi}{314}(x-50)\;, \\
w(x,y,z) &= 0\;.
\end{align*}
The test is performed on a $50\times 50\times 50$ grid.
The time step is $\Delta t = 1$.
At the inflow edges $\vec{v}\cdot\vec{n}<0$, homogeneous Neumann
boundary conditions are prescribed.

In Fig.~\ref{fig:zalesak_sphere}, the Zalesak's sphere for a sequence of snapshots
during one full revolution is shown, using the gradient-augmented level set method
(on the right), and the classical level set approach (on the left).
We observe that the gradient-augmented approach preserves the correct shape of the
surface better than the classical level set method does. In particular, it is able to
maintain the notch on the Zalesak's sphere, in contrast to the standard approach which
merges both sides while significantly gaining mass.

\begin{figure}
\centering
\begin{minipage}[t]{.48\textwidth}
\centering
\includegraphics[width=0.97\textwidth]{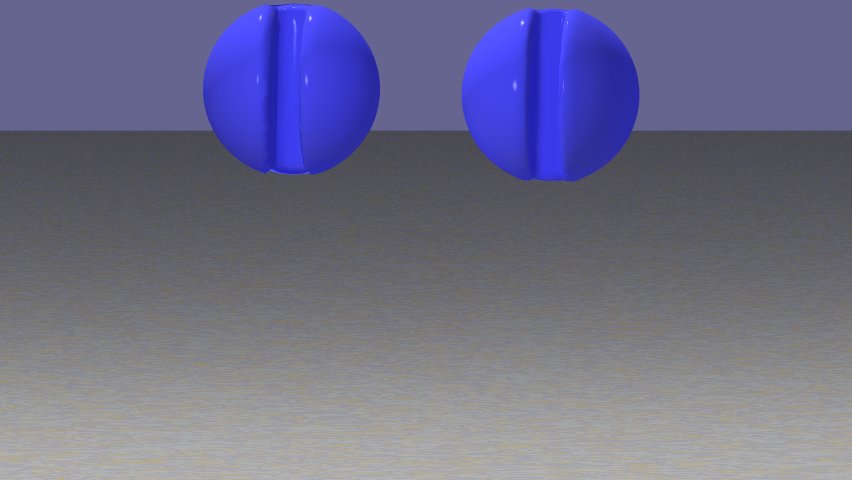} \\[.1em]
\includegraphics[width=0.97\textwidth]{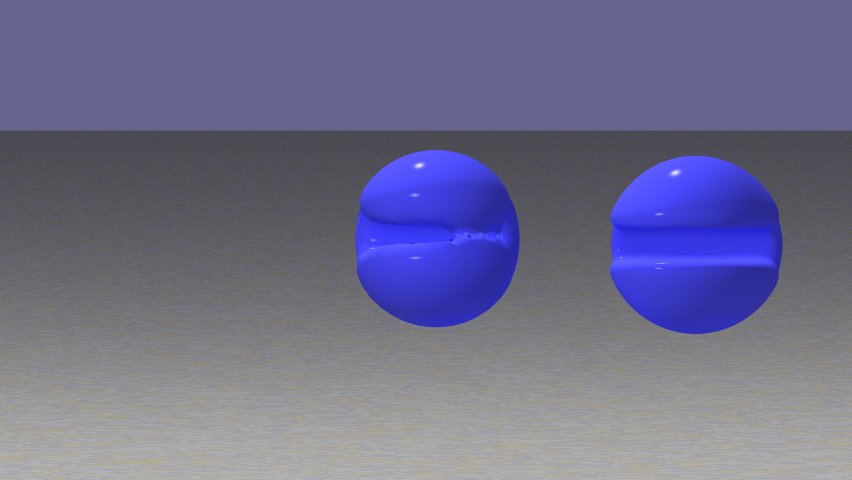} \\[.1em]
\includegraphics[width=0.97\textwidth]{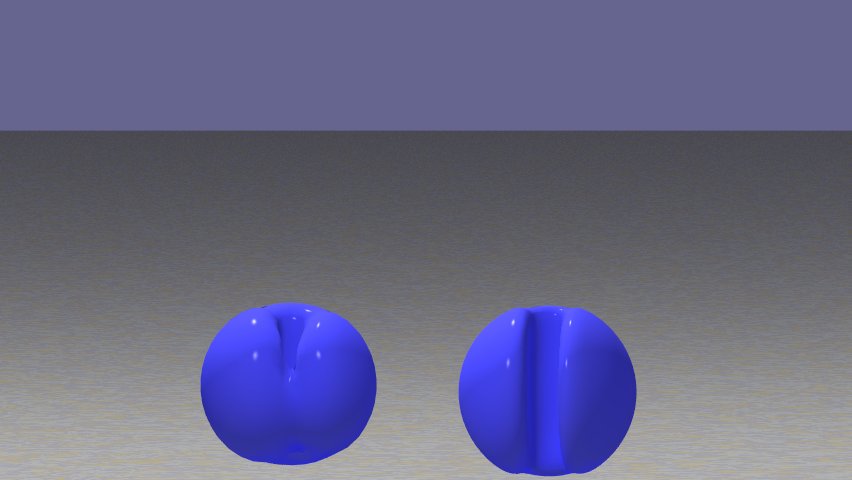} \\[.1em]
\includegraphics[width=0.97\textwidth]{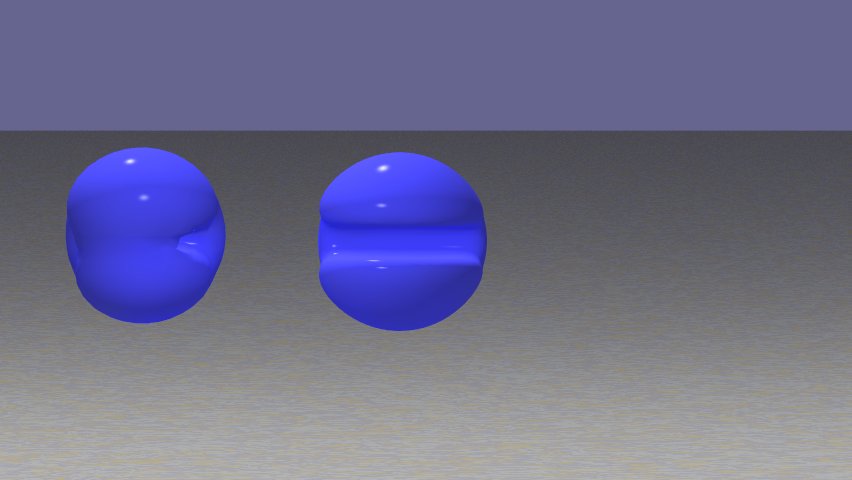} \\[.1em]
\includegraphics[width=0.97\textwidth]{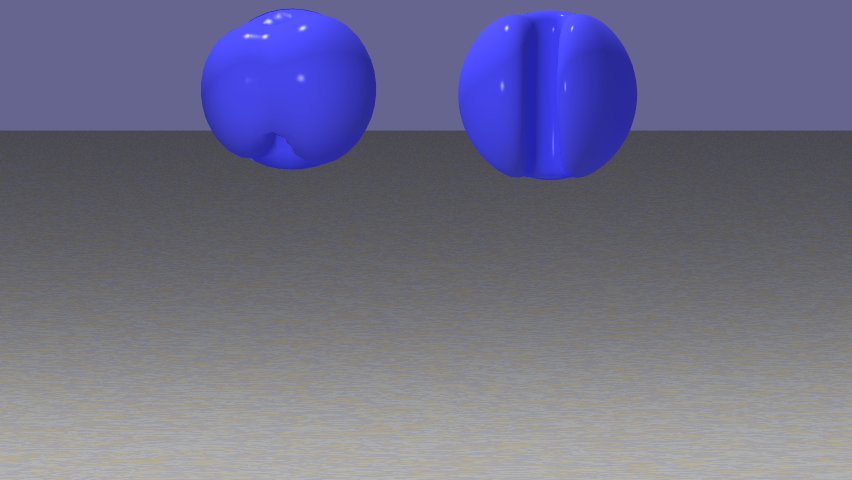}
\caption{Zalesak's sphere at $t\in\{0,157,314,471,628\}$ --
classical approach (left), gradient-augmented method (right)}
\label{fig:zalesak_sphere}
\end{minipage}
\hfill
\begin{minipage}[t]{.48\textwidth}
\centering
\includegraphics[width=0.97\textwidth]{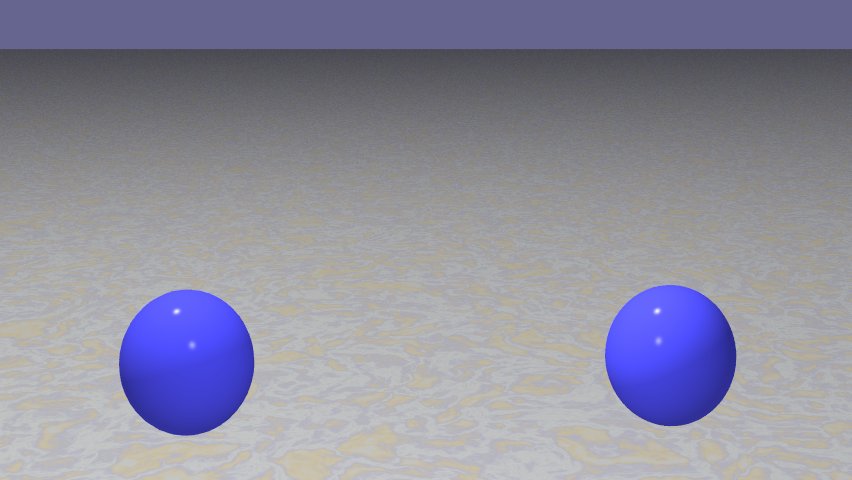} \\[.1em]
\includegraphics[width=0.97\textwidth]{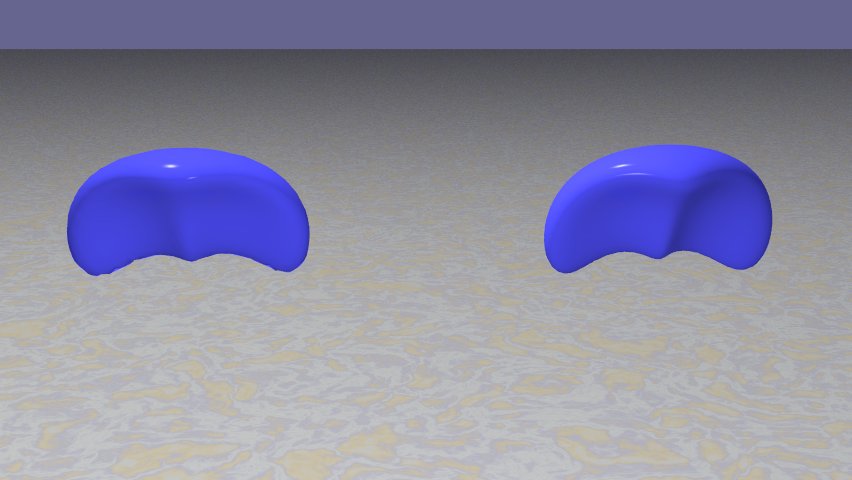} \\[.1em]
\includegraphics[width=0.97\textwidth]{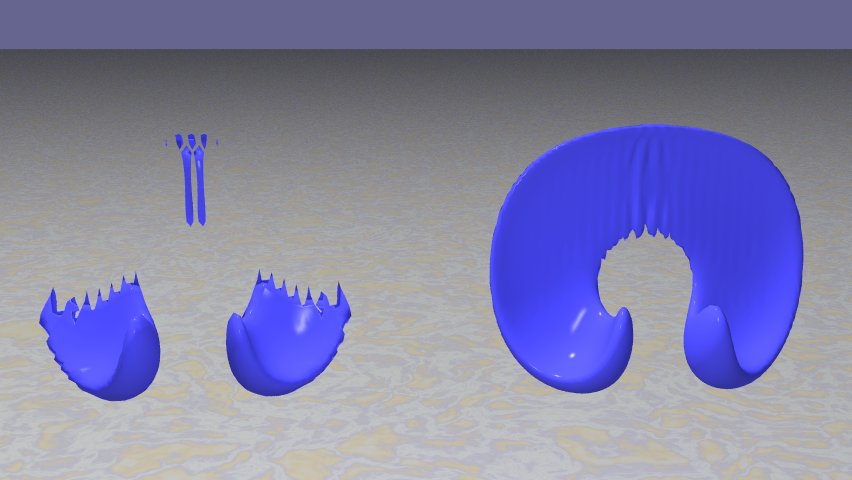} \\[.1em]
\includegraphics[width=0.97\textwidth]{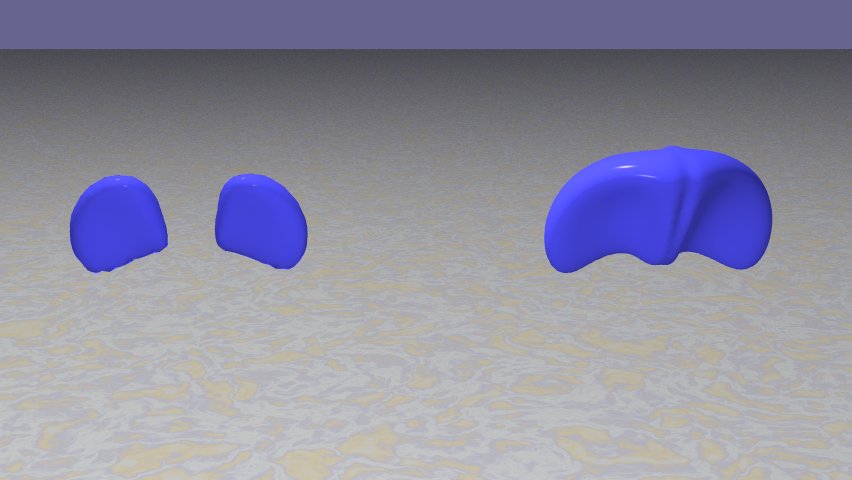} \\[.1em]
\includegraphics[width=0.97\textwidth]{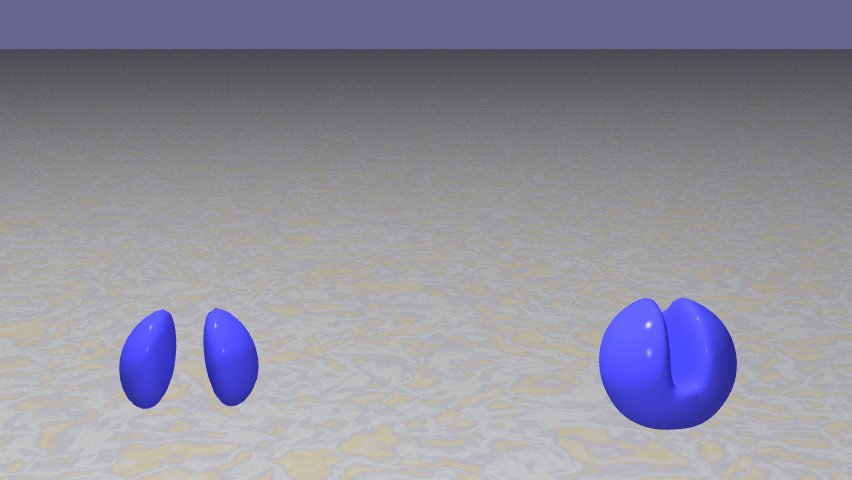}
\caption{3D deformation field test at $t\in\{0,0.625,1.25,1.875,2.5\}$ --
classical approach (left), gradient-augmented method (right)}
\label{fig:3D_deformation_field}
\end{minipage}
\end{figure}

\subsubsection{3D Deformation Field}
\label{subsubsec:3d_deformation}
LeVeque proposed a three dimensional incompressible flow field \cite{LeVeque1996}
which combines a deformation in the $x$-$y$ plane with one in the $x$-$z$ plane.
The velocity is given by
\begin{align*}
u(x,y,z) &= 2\sin(\pi x)^{2}\sin(2\pi xy)\sin(2\pi z)\;, \\
v(x,y,z) &= -\sin(2\pi x)\sin(\pi y)^{2}\sin(2\pi z)\;, \\
w(x,y,z) &= -\sin(2\pi x)\sin(2\pi y)\sin(\pi z)^{2}\;.
\end{align*}
The field is modulated in time using $\cos\prn{\tfrac{\pi t}{T}}$. Here, we
consider $T = 2.5$. The tests is performed on a $50\times 50\times 50$ grid.
A sphere of radius $0.15$ that is initially centered at $(0.35,0.35,0.35)$,
is advected up to $t = 2.5$, i.e.~it is deformed by the flow, and then brought back
to its initial configuration. As before, we compare the gradient-augmented scheme
with WENO, both with $\Delta t = \Delta x$.

In Fig.~\ref{fig:3D_deformation_field}, a sequence of snapshots is shown, using the
gradient-augmented level set method (on the right), and the classical level set
approach (on the left). We observe the CIR method's ability to conserve mass and
maintain the topology of the sphere more accurately than the standard approach.
In the middle frame, one can observe that the classical level set method suffers a
significant loss of mass around the time of maximal deformation $t = 1.25$, at which
the surface is very thin. Evidently, the ability of the gradient-augmented level
set method to represent structures of subgrid size is particularly beneficial here.

Note that the jagged edges seen in the middle frame in
Fig.~\ref{fig:3D_deformation_field} are due to the representation of a surface that
is thinner than the grid resolution. While these grid effects are present in both
numerical schemes, the level set functions themselves do not show any spurious
oscillations.

\subsubsection{Computational Effort}
Both in 2D and 3D, the gradient-augmented level set method carries about the same
computational cost as the classical level set method. More specifically, we compare
the CPU times for the computation of the 2D deformation field, considered
in Sect.~\ref{subsubsec:numerics_swirl_2d}, on a $64\times 64$ grid, using a
single core desktop computer.
The classical level set approach takes 13 seconds with reinitialization,
and 10 seconds without (in which case the quality of the results is significantly
reduced). In comparison, the gradient-augmented approach takes 9 seconds when
the superconsistent Shu-Osher RK3 method is used, and 8 seconds when gradients are
updated using Heun's method.
These ratios carry approximately over to other grid resolutions and other tests.

\begin{figure}[p]
Gradient-augmented CIR method \\
\begin{minipage}[t]{.24\textwidth}
\includegraphics[height=\textwidth]{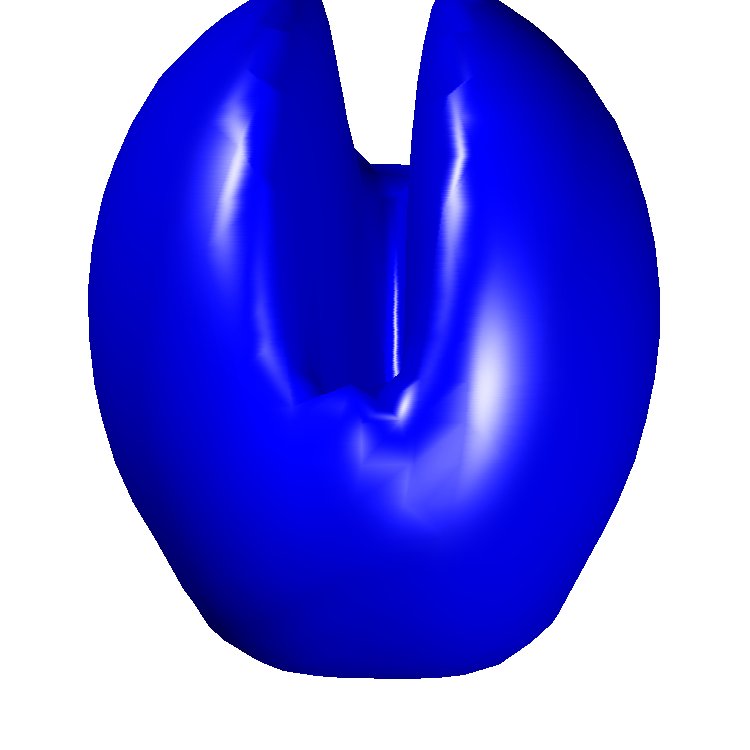} \\[-2em]
\begin{center}\begin{footnotesize} $50\times 50\times 50$ grid \end{footnotesize}\end{center}
\end{minipage}
\hfill
\begin{minipage}[t]{.24\textwidth}
\includegraphics[height=\textwidth]{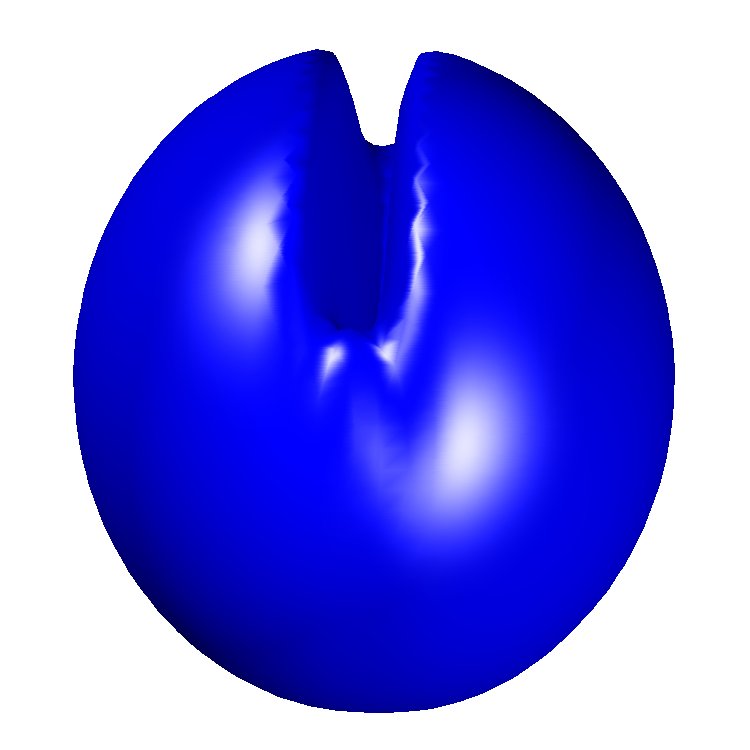} \\[-2em]
\begin{center}\begin{footnotesize} $100\times 100\times 100$ grid \end{footnotesize}\end{center}
\end{minipage}
\hfill
\begin{minipage}[t]{.24\textwidth}
\includegraphics[height=\textwidth]{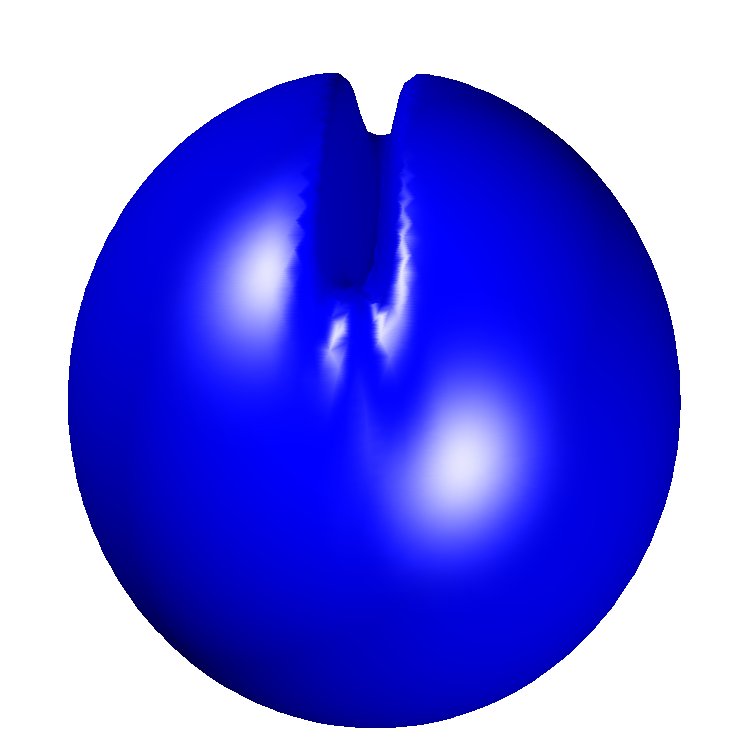} \\[-2em]
\begin{center}\begin{footnotesize} $150\times 150\times 150$ grid \end{footnotesize}\end{center}
\end{minipage}
\hfill
\begin{minipage}[t]{.24\textwidth}
\includegraphics[height=\textwidth]{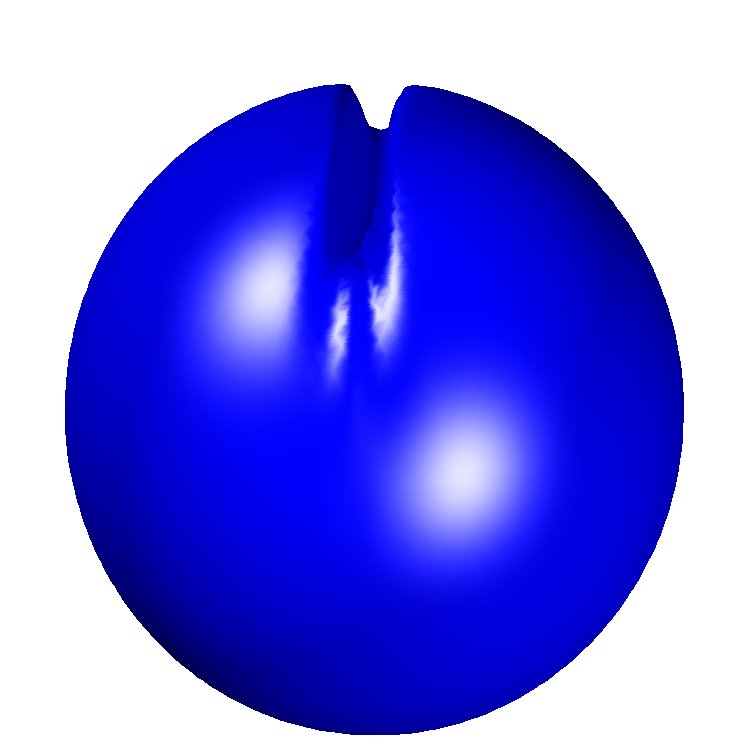} \\[-2em]
\begin{center}\begin{footnotesize} $200\times 200\times 200$ grid \end{footnotesize}\end{center}
\end{minipage}

\vspace{1em}
WENO \\
\begin{minipage}[t]{.24\textwidth}
\includegraphics[height=\textwidth]{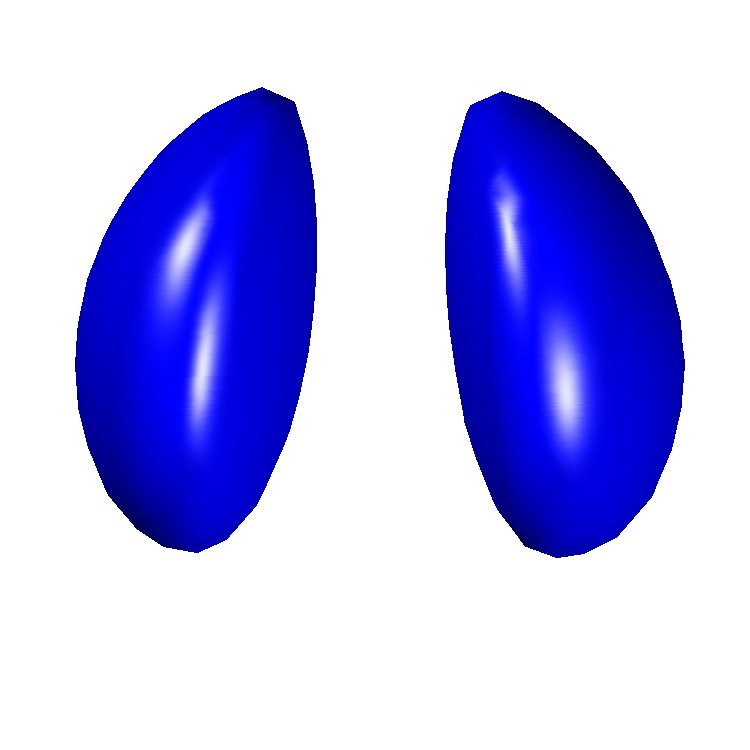} \\[-2em]
\begin{center}\begin{footnotesize} $50\times 50\times 50$ grid \end{footnotesize}\end{center}
\end{minipage}
\hfill
\begin{minipage}[t]{.24\textwidth}
\includegraphics[height=\textwidth]{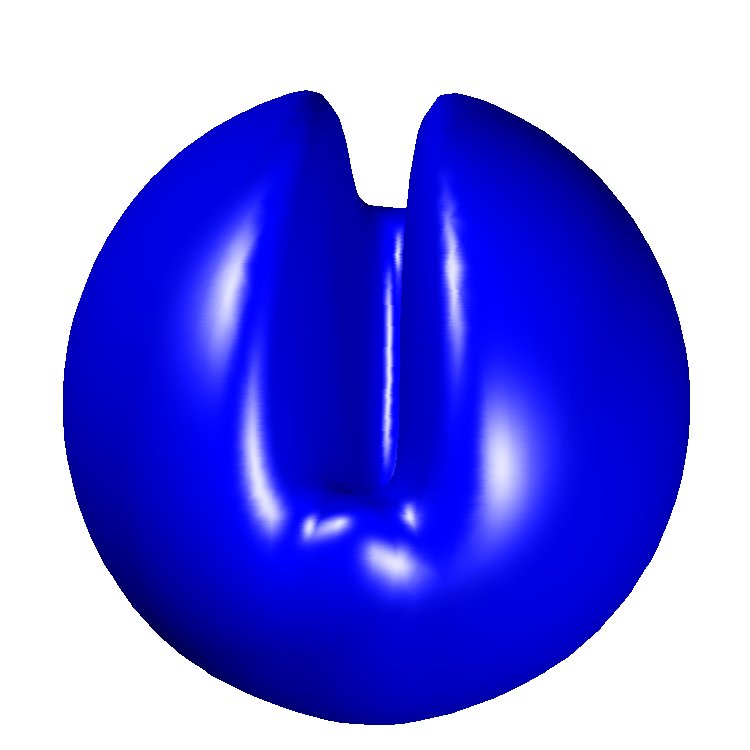} \\[-2em]
\begin{center}\begin{footnotesize} $100\times 100\times 100$ grid \end{footnotesize}\end{center}
\end{minipage}
\hfill
\begin{minipage}[t]{.24\textwidth}
\includegraphics[height=\textwidth]{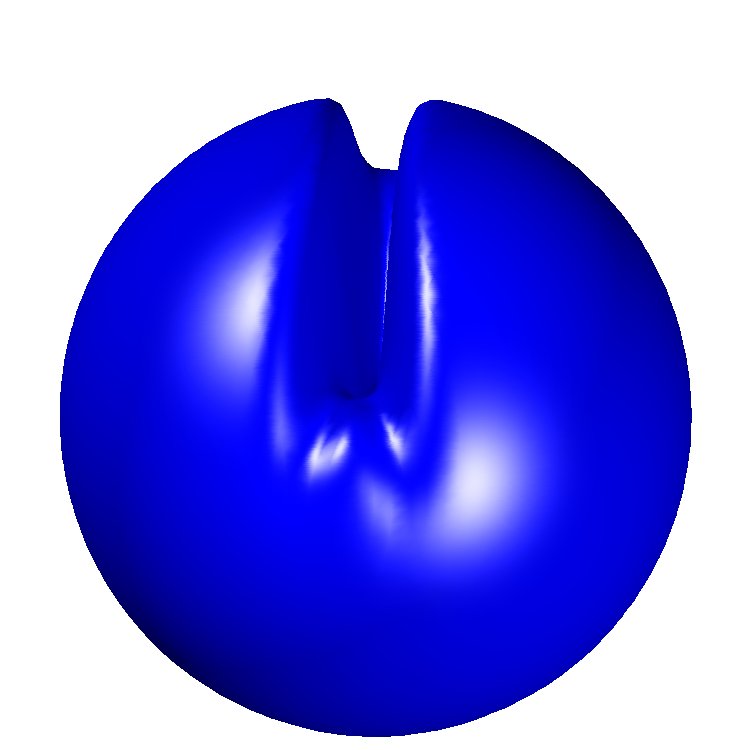} \\[-2em]
\begin{center}\begin{footnotesize} $150\times 150\times 150$ grid \end{footnotesize}\end{center}
\end{minipage}
\hfill
\begin{minipage}[t]{.24\textwidth}
\includegraphics[height=\textwidth]{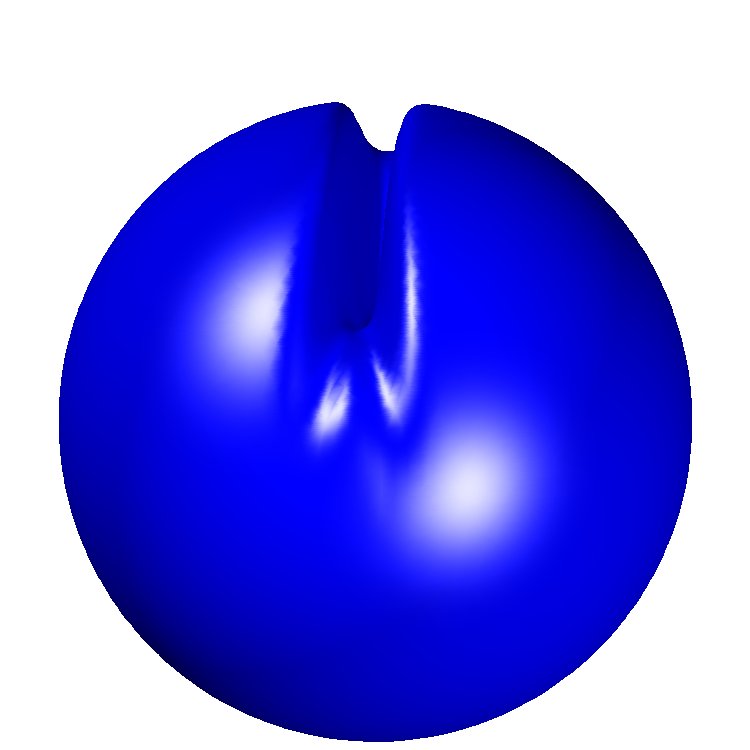} \\[-2em]
\begin{center}\begin{footnotesize} $200\times 200\times 200$ grid \end{footnotesize}\end{center}
\end{minipage}

\vspace{1em}
WENO with reinitialization \\
\begin{minipage}[t]{.24\textwidth}
\includegraphics[height=\textwidth]{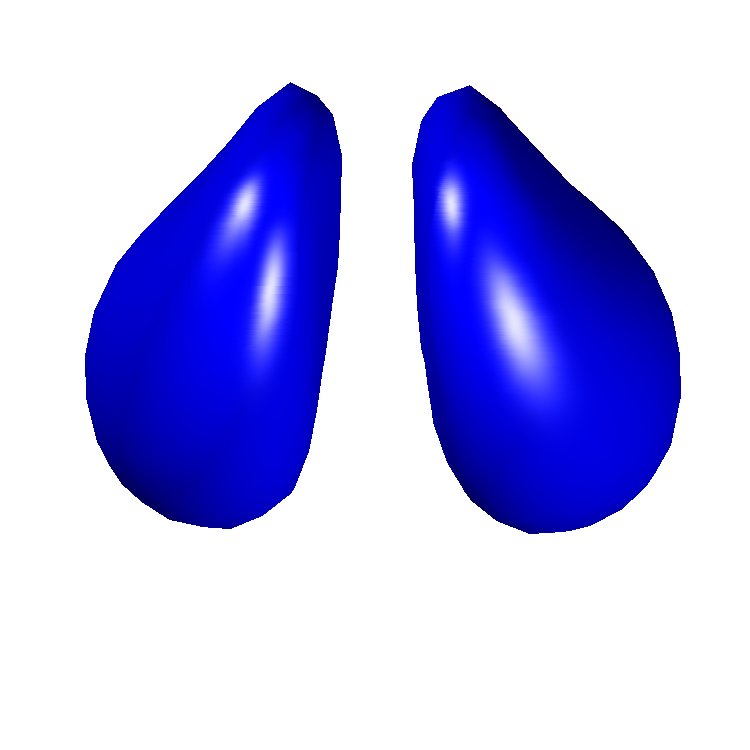} \\[-2em]
\begin{center}\begin{footnotesize} $50\times 50\times 50$ grid \end{footnotesize}\end{center}
\end{minipage}
\hfill
\begin{minipage}[t]{.24\textwidth}
\includegraphics[height=\textwidth]{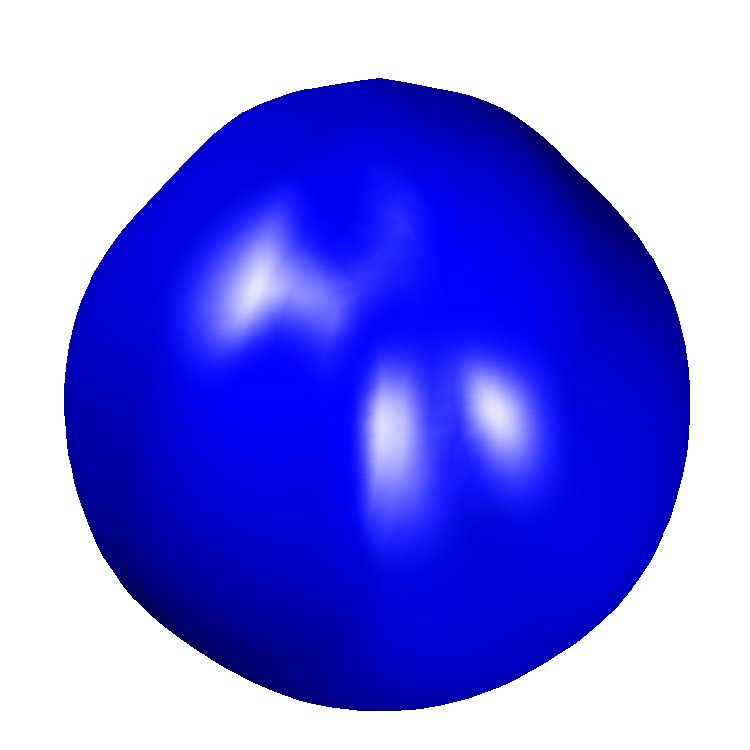} \\[-2em]
\begin{center}\begin{footnotesize} $100\times 100\times 100$ grid \end{footnotesize}\end{center}
\end{minipage}
\hfill
\begin{minipage}[t]{.24\textwidth}
\includegraphics[height=\textwidth]{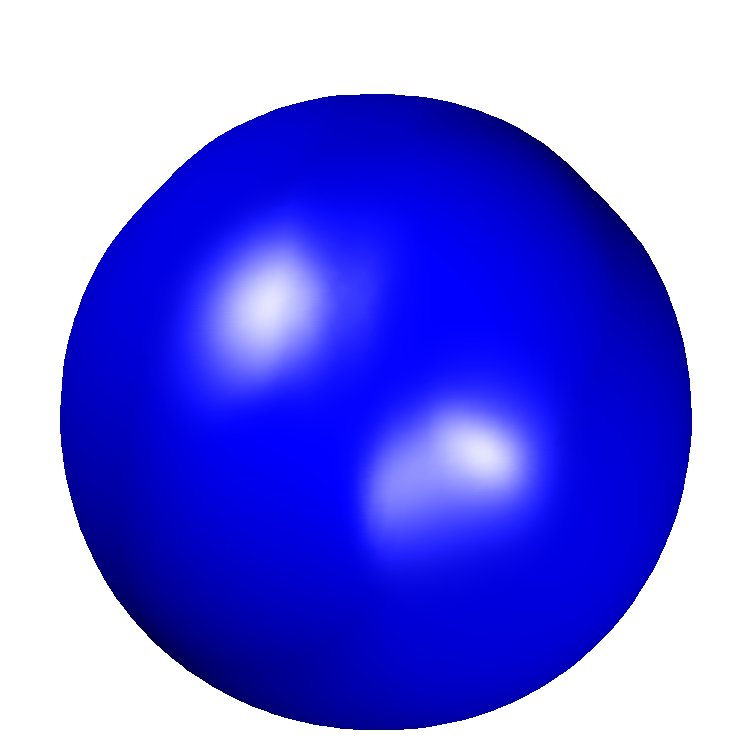} \\[-2em]
\begin{center}\begin{footnotesize} $150\times 150\times 150$ grid \end{footnotesize}\end{center}
\end{minipage}
\hfill
\begin{minipage}[t]{.24\textwidth}
\includegraphics[height=\textwidth]{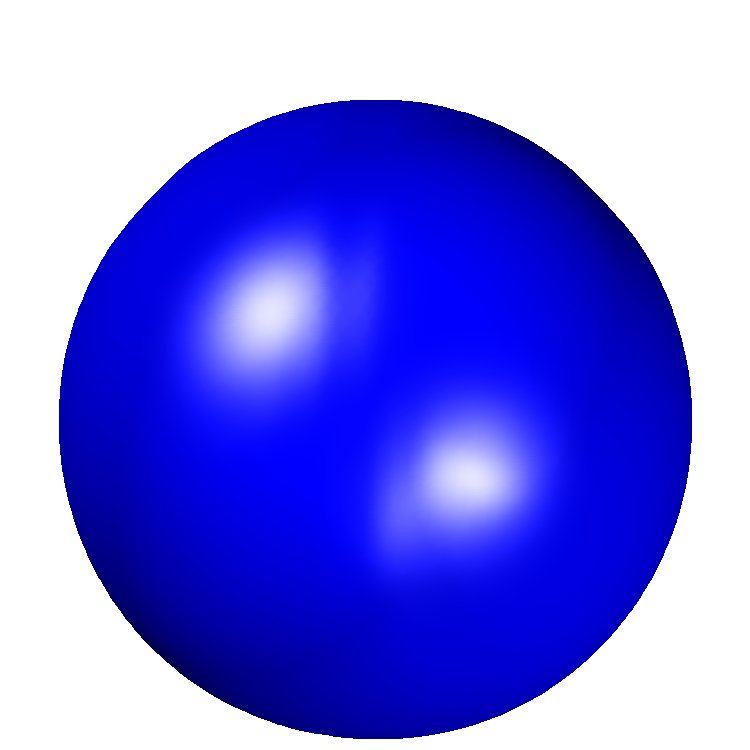} \\[-2em]
\begin{center}\begin{footnotesize} $200\times 200\times 200$ grid \end{footnotesize}\end{center}
\end{minipage}
\caption{Final state of deformation of a sphere for various resolutions}
\label{fig:final_sphere}
\end{figure}

\begin{table}[p]
\centering
\vspace{2em}
\begin{tabular}{|l|rrrr|}
\hline
resolution  & $n=50$  & $n=100$ & $n=150$ & $n=200$ \\
\hline
GA-CIR      &  -6.4\% &  -2.9\% &  -2.0\% &  -1.3\% \\
WENO        & -63.1\% & -14.1\% &  -4.3\% &  -2.0\% \\
WENO~reinit & -66.3\% &  -9.0\% &  -2.9\% &  -1.3\% \\
\hline
\end{tabular}
\caption{Volume loss for the deformation of a sphere}
\label{tab:volume_sphere}
\vspace{3em} ~
\end{table}

\begin{figure}[p]
Gradient-augmented CIR method \\
\begin{minipage}[t]{.24\textwidth}
\includegraphics[height=\textwidth]{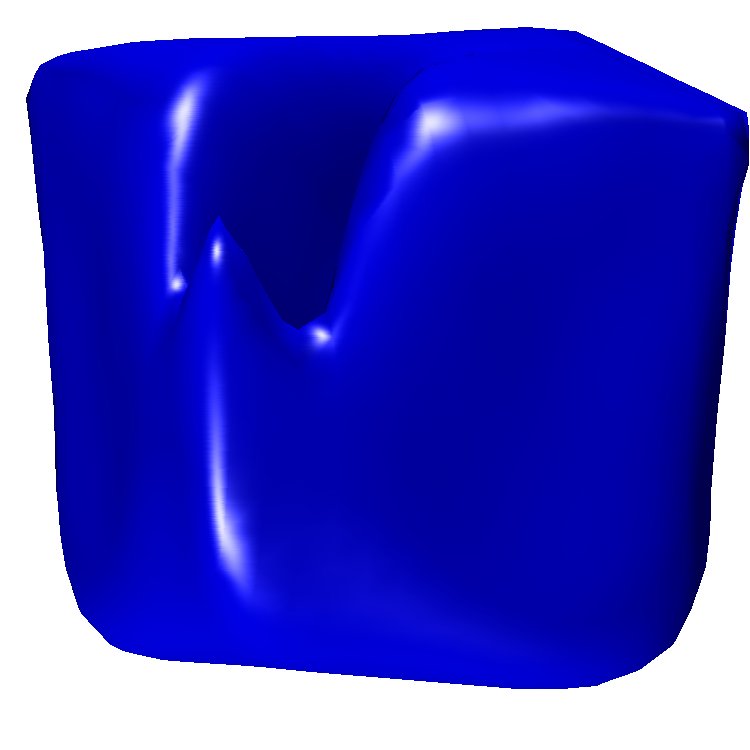} \\[-2em]
\begin{center}\begin{footnotesize} $50\times 50\times 50$ grid \end{footnotesize}\end{center}
\end{minipage}
\hfill
\begin{minipage}[t]{.24\textwidth}
\includegraphics[height=\textwidth]{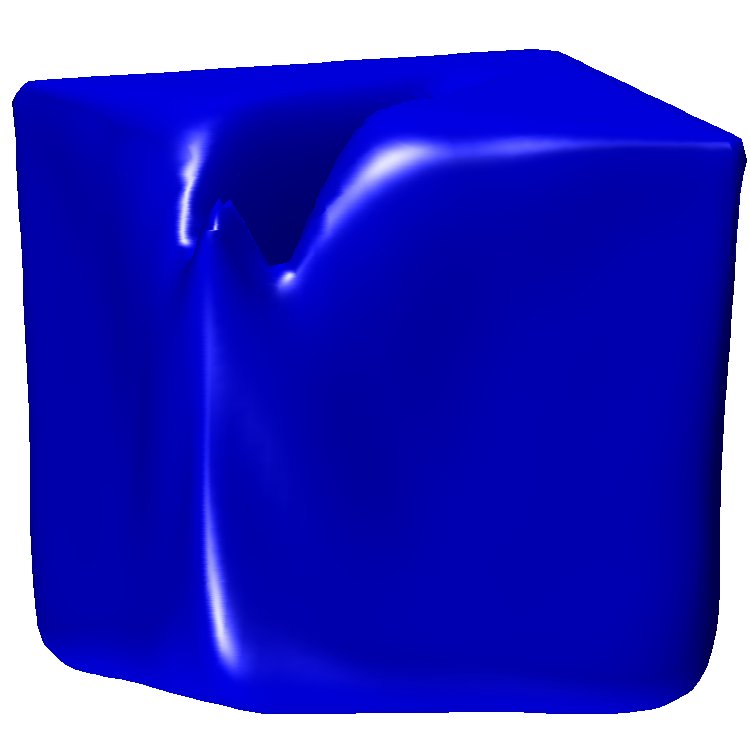} \\[-2em]
\begin{center}\begin{footnotesize} $100\times 100\times 100$ grid \end{footnotesize}\end{center}
\end{minipage}
\hfill
\begin{minipage}[t]{.24\textwidth}
\includegraphics[height=\textwidth]{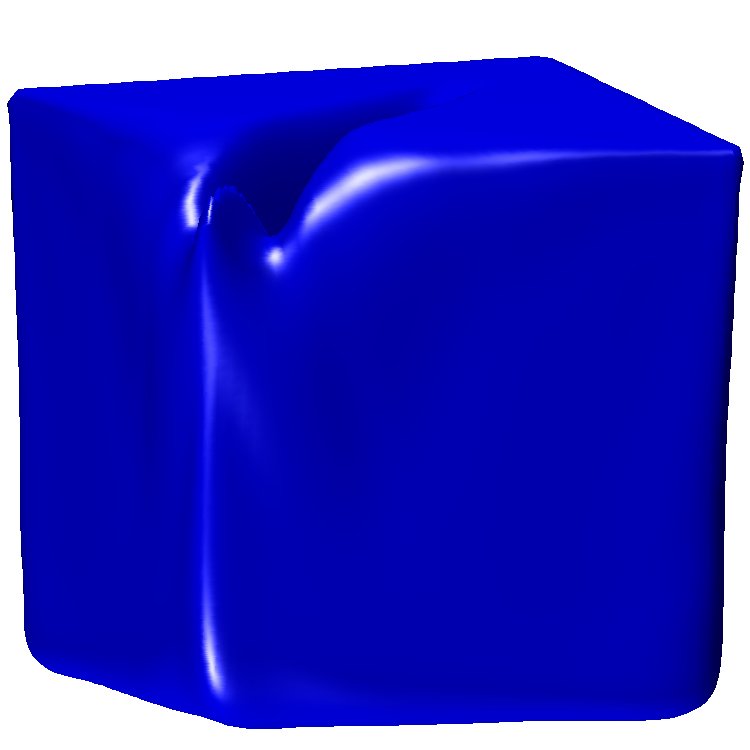} \\[-2em]
\begin{center}\begin{footnotesize} $150\times 150\times 150$ grid \end{footnotesize}\end{center}
\end{minipage}
\hfill
\begin{minipage}[t]{.24\textwidth}
\includegraphics[height=\textwidth]{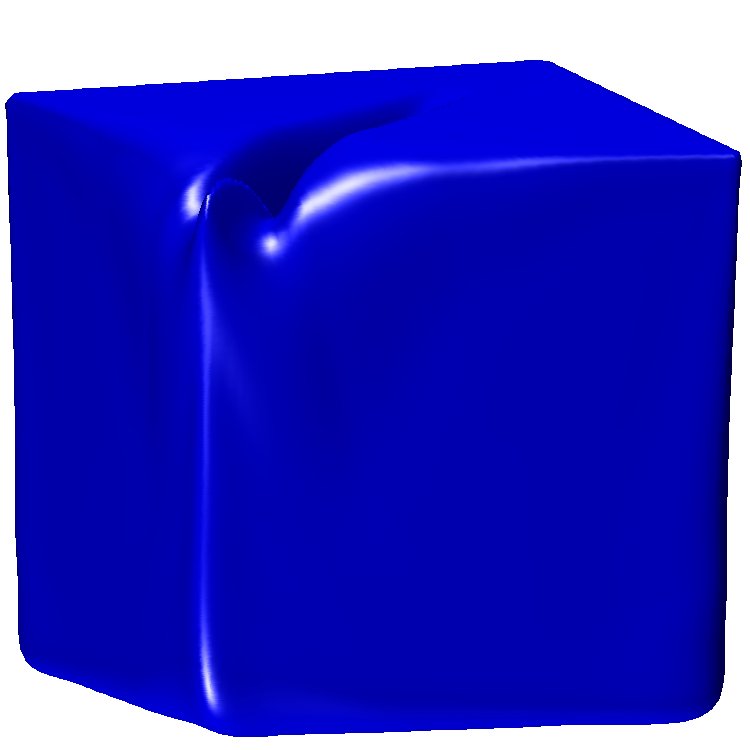} \\[-2em]
\begin{center}\begin{footnotesize} $200\times 200\times 200$ grid \end{footnotesize}\end{center}
\end{minipage}

\vspace{1em}
WENO \\
\begin{minipage}[t]{.24\textwidth}
\includegraphics[height=\textwidth]{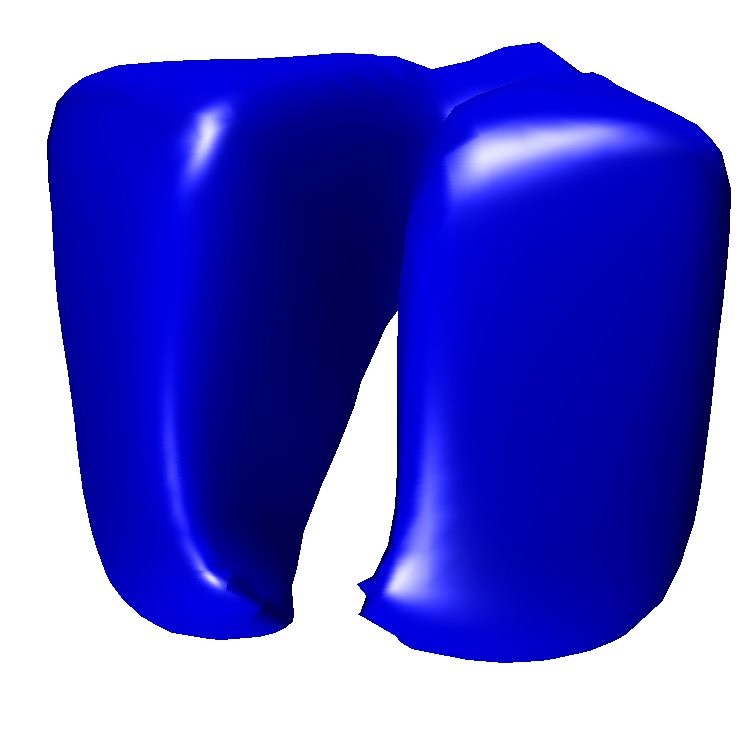} \\[-2em]
\begin{center}\begin{footnotesize} $50\times 50\times 50$ grid \end{footnotesize}\end{center}
\end{minipage}
\hfill
\begin{minipage}[t]{.24\textwidth}
\includegraphics[height=\textwidth]{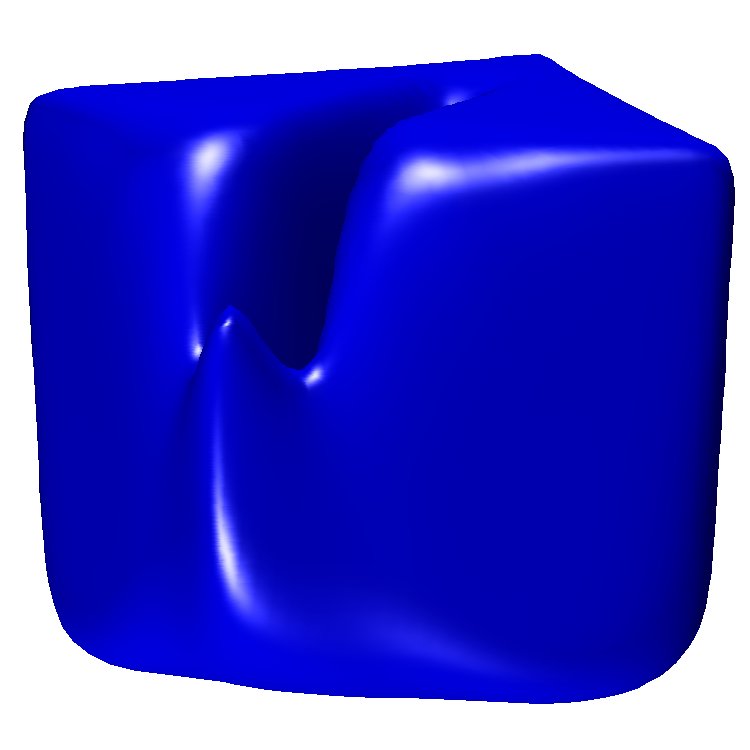} \\[-2em]
\begin{center}\begin{footnotesize} $100\times 100\times 100$ grid \end{footnotesize}\end{center}
\end{minipage}
\hfill
\begin{minipage}[t]{.24\textwidth}
\includegraphics[height=\textwidth]{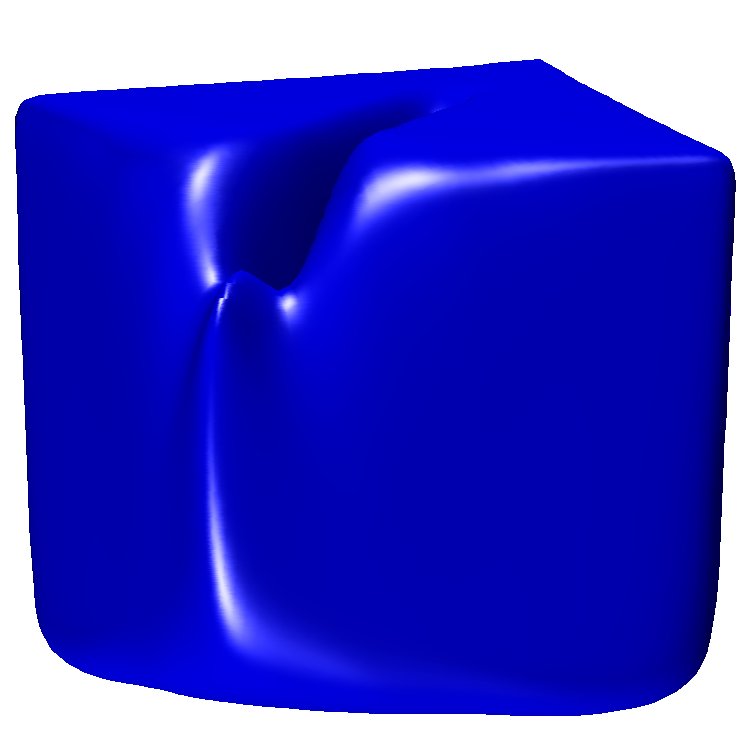} \\[-2em]
\begin{center}\begin{footnotesize} $150\times 150\times 150$ grid \end{footnotesize}\end{center}
\end{minipage}
\hfill
\begin{minipage}[t]{.24\textwidth}
\includegraphics[height=\textwidth]{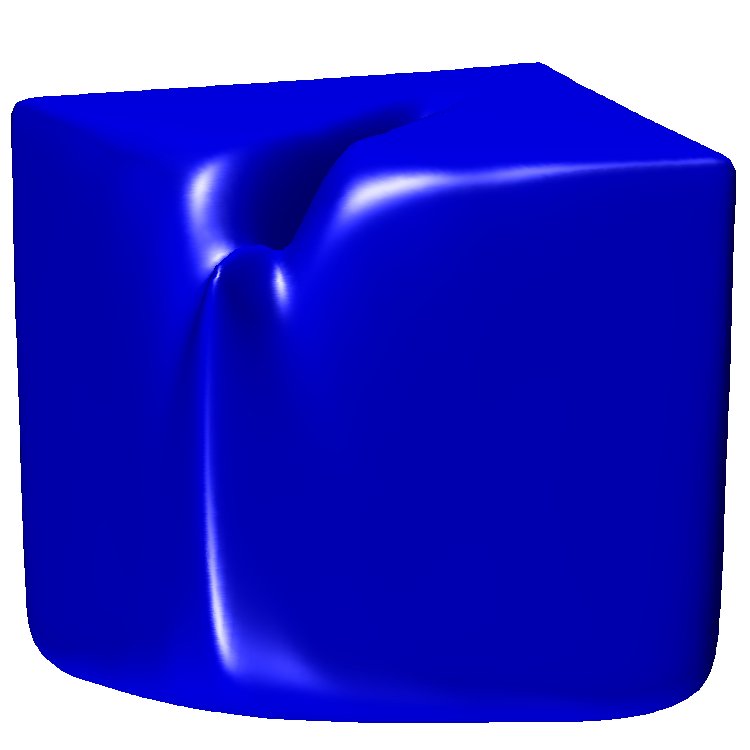} \\[-2em]
\begin{center}\begin{footnotesize} $200\times 200\times 200$ grid \end{footnotesize}\end{center}
\end{minipage}

\vspace{1em}
WENO with reinitialization \\
\begin{minipage}[t]{.24\textwidth}
\includegraphics[height=\textwidth]{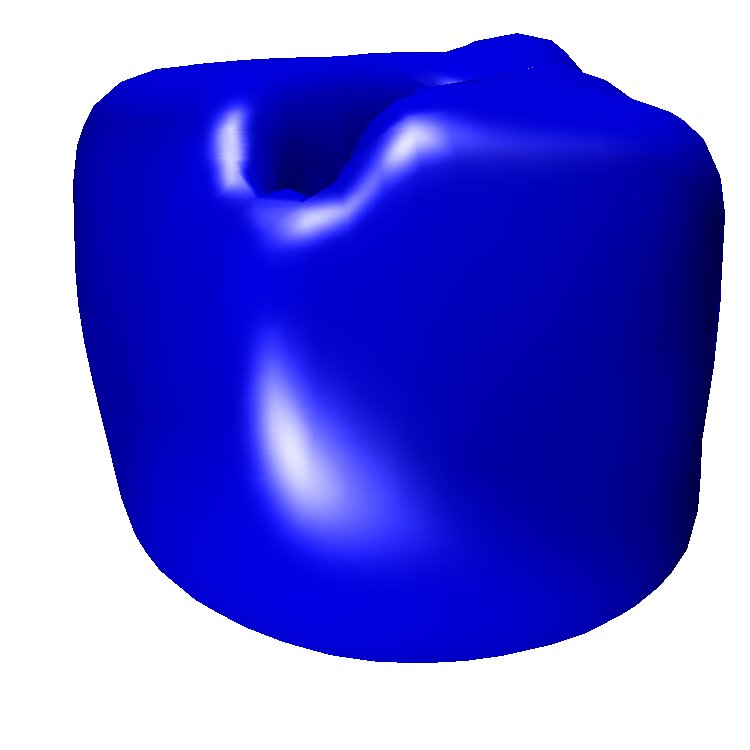} \\[-2em]
\begin{center}\begin{footnotesize} $50\times 50\times 50$ grid \end{footnotesize}\end{center}
\end{minipage}
\hfill
\begin{minipage}[t]{.24\textwidth}
\includegraphics[height=\textwidth]{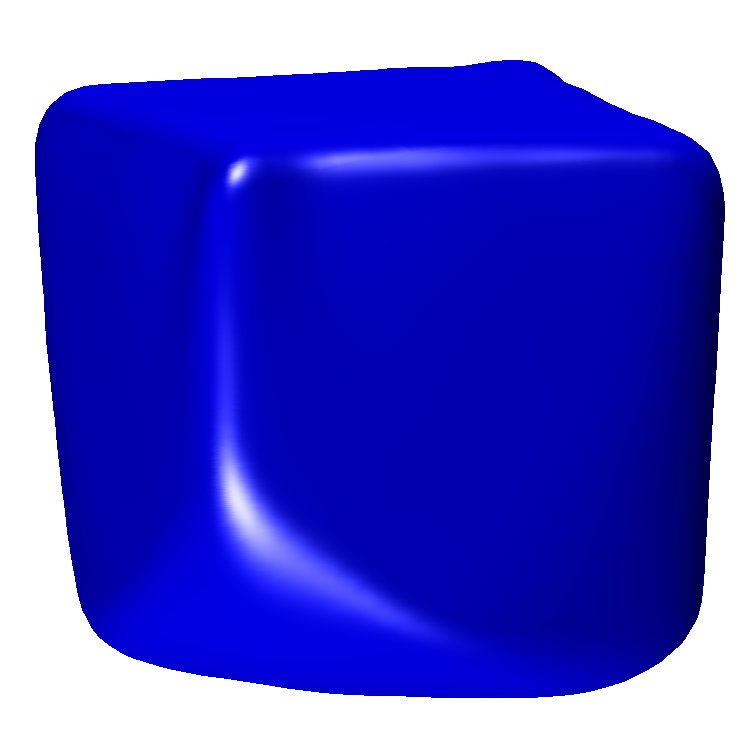} \\[-2em]
\begin{center}\begin{footnotesize} $100\times 100\times 100$ grid \end{footnotesize}\end{center}
\end{minipage}
\hfill
\begin{minipage}[t]{.24\textwidth}
\includegraphics[height=\textwidth]{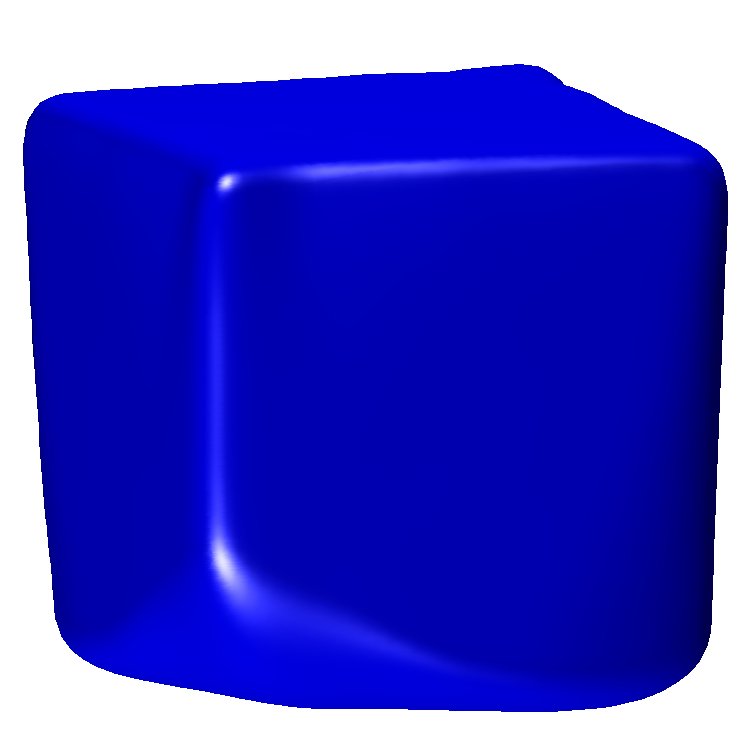} \\[-2em]
\begin{center}\begin{footnotesize} $150\times 150\times 150$ grid \end{footnotesize}\end{center}
\end{minipage}
\hfill
\begin{minipage}[t]{.24\textwidth}
\includegraphics[height=\textwidth]{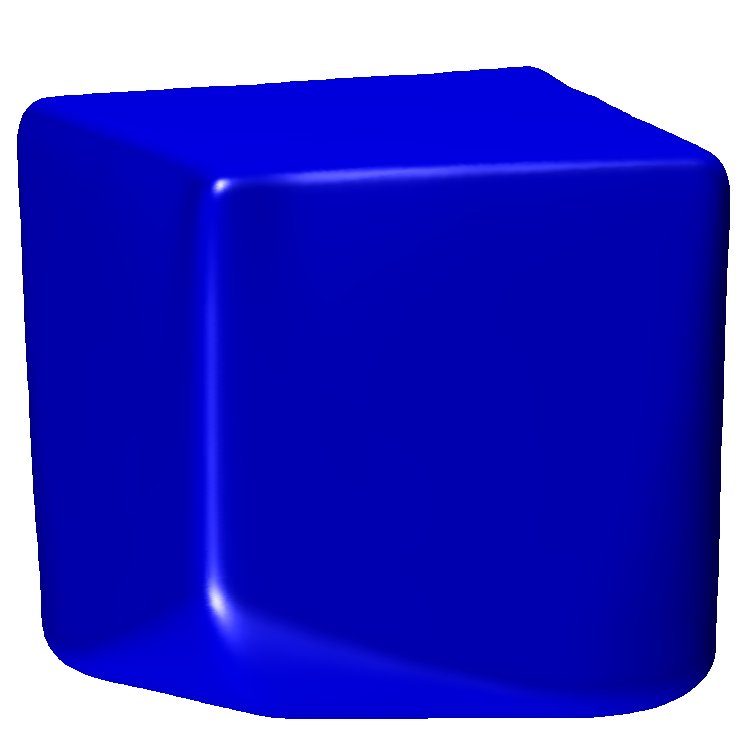} \\[-2em]
\begin{center}\begin{footnotesize} $200\times 200\times 200$ grid \end{footnotesize}\end{center}
\end{minipage}
\caption{Final state of deformation of a cube for various resolutions}
\label{fig:final_cube}
\end{figure}

\begin{table}[p]
\centering
\vspace{2em}
\begin{tabular}{|l|rrrr|}
\hline
resolution  & $n=50$  & $n=100$ & $n=150$ & $n=200$ \\
\hline
GA-CIR      &  -5.4\% &  -3.3\% &  -1.6\% &  -0.9\% \\
WENO        & -54.5\% & -15.8\% &  -9.6\% &  -6.5\% \\
WENO~reinit & -35.4\% & -11.2\% &  -7.1\% &  -5.0\% \\
\hline
\end{tabular}
\caption{Volume loss for the deformation of a cube}
\label{tab:volume_cube}
\vspace{3em} ~
\end{table}

\subsubsection{Volume Loss}
For the 3D deformation field defined in Sect.~\ref{subsubsec:3d_deformation}, we investigate
the loss of volume of two closed surfaces (a sphere and a cube), and its dependence on the
grid resolution.
We compare the gradient-augmented CIR method with the classical WENO scheme, once without,
and once with reinitialization. The four considered grid resolutions are
$50\times 50\times 50$, $100\times 100\times 100$, $150\times 150\times 150$,
and $200\times 200\times 200$. As the middle picture in Fig.~\ref{fig:3D_deformation_field}
shows, the surface becomes thinner than the grid resolution on a $50\times 50\times 50$
grid.

Fig.~\ref{fig:final_sphere} shows the final state ($t = T = 2.5$) of a sphere of radius
$0.15$, centered at $(0.35,0.35,0.35)$, under the evolution by the velocity field given
in Sect.~\ref{subsubsec:3d_deformation}. In all cases, the final shape has lost some
volume. The specific relative volume loss is shown in Table~\ref{tab:volume_sphere}.
One can observe that the gradient-augmented scheme performs significantly better than the
classical schemes if the grid resolution is on the order of the size of the smallest
structures. This indicates that gradient-augmented schemes can be expected to particularly
improve the resolution of small structures in realistic flow simulations, for which
features on various length scales have to be resolved. The results shown in
Fig.~\ref{fig:final_sphere} also clearly show the convergence of all three schemes to
a perfect sphere, as the grid resolution is increased.

Another interesting observation can be made about reinitialization. The comparison of
WENO without reinitialization and WENO with reinitialization, shown
in Fig.~\ref{fig:final_sphere}, indicates that reinitialization generally improves the
quality of the obtained results. Small structures are represented in a more robust fashion,
and consequently the loss of volume is reduced. In addition, the shape of the final surface
appears closer to a true sphere. In particular, reinitialization flattens out the numerical
notch that is present in the top two rows in Fig.~\ref{fig:final_sphere}.

However, reinitialization has a tendency to make surfaces round, and thus it performs
particularly well for surfaces that are similar to spheres. This can be clearly seen in
the results shown in Fig.~\ref{fig:final_cube}. For the same 3D deformation field, now
a cube of size $0.3\times 0.3\times 0.3$, centered at $(0.35,0.35,0.35)$, is evolved.
Observe that reinitialization smears out the numerical notch (which is good), but also
rounds the bottom face of the cube (which is bad). This effect is also visible in the
loss of volume, given in Table~\ref{tab:volume_cube}. Observe that for the considered
cube, a significant volume loss occurs for WENO, both with and without reinitialization.
In contrast, the volume loss with the gradient-augmented scheme is much smaller.
In addition, the results in Fig.~\ref{fig:final_cube} again show that
in relation to classical schemes, the gradient-augmented scheme performs particularly
well for low grid resolutions.

\section{Conclusions and Outlook}
\label{sec:conclusions_outlook}
The results presented in this paper show that common problems with level set methods
can be ameliorated when incorporating gradients as an independent quantity into the
computation. The presented gradient-augmented level set method is based on a Hermite
interpolation, which is used as a fundamental ingredient for the reconstruction of
the surface, the approximation of surface normals and curvature, and the advection
under a given velocity field. In this paper, we consider a $p$-cubic Hermite
interpolant, that in each grid cell is defined solely by data on the cell vertices.
We have shown that this $p$-cubic interpolant gives rise to a certain level of subgrid
resolution, to a second order accurate approximation of curvature, and to a globally
third order accurate advection scheme. Curvature is obtained at arbitrary positions
by analytically differentiating the $p$-cubic interpolant. The advection scheme is
based on a generalization of the CIR method. Characteristic curves are traced
backwards from grid points, and function values and derivatives of the level set
function are obtained from the $p$-cubic interpolant, respectively its derivatives.
Therefore, the resulting approximations and schemes are optimally local,
i.e.~information at a given (grid) point is updated by using data from only a single
cell. This promises a simpler treatment of adaptive mesh refinement and boundary
conditions.

In the theoretical investigation of gradient-augmented schemes, the concept of
superconsistency has been introduced, which admits the interpretation of a
gradient-augmented method as an evolution rule in a function space, in which
analytical coherence between function values and gradients is preserved.
In each time step a projection rule is applied, that is based on an interpolation.
In this paper, the specific case of a $p$-cubic Hermite interpolation is studied.
Employing the introduced concepts, the accuracy and stability of the generalized CIR
method have been investigated theoretically and verified in numerical experiments.
In addition, the performance of the gradient-augmented level set approach has been
compared to a classical high-order level set method, both in 2D and 3D benchmark
tests. While of similar computational cost, the gradient-augmented approach generally
yields more accurate results. In particular, small structures are preserved much
better than with the classical level set method. The ability to represent structures
of subgrid size turns out to be of great benefit.

While the presented approach looks promising and performs well in numerical tests,
various aspects remain to be investigated. The proof of stability presented in this
paper covers only a special case, and a general proof of stability is one key objective
of our current research. Another question of theoretical importance is the issue of
reinitialization. All numerical tests with the generalized level set method considered
in this paper have been done without reinitialization. While the knowledge of gradients
itself generally gives rise to a more accurate recovery of the surface, an additional
reinitialization may improve the quality of the method even further.

A fundamental question of interest in the technical realm is the combination of
the gradient-augmented level set approach with adaptive mesh refinement.
Here, the optimal locality of the advection scheme may prove advantageous, and
preserve automatically the high-order accuracy through various levels of grid
refinement. Also, the combination with Lagrangian particles, such as done
in \cite{EnrightFedkiwFerzigerMitchell2002} for the classical level set method, is of
interest. Another important aspect to be investigated is the case of the velocity
field, and its gradient, not being accessible everywhere. Most prominently, this
is the case in multi-phase fluid flow simulations, in which the evolution of the
level set function is coupled to an evolution for the fluid velocities.
In future research, we plan to investigate the incorporation of the presented
gradient-augmented level set method into a ghost fluid method \cite{FedkiwLiu2002}.
One of various challenges with this approach is the new possibility of up to three
intersections of the reconstructed surface with each cell edge.

The $p$-cubic interpolant considered in this paper yields an accurate method, but as
discussed in Sect.~\ref{subsec:small_structures}, it may be too smooth to capture
small structures when these are represented by a signed distance function. Hence, it is
an important question to investigate whether other types of interpolation can improve
the ability to capture subgrid structures. A related question is whether it is
beneficial to additionally augment the method by higher derivatives. For instance,
having direct access to the Hessian of the level set function may give rise to an
even better representation of small structures, and an even more accurate
approximation of curvature.
Another question is whether the gradient-augmented approach can be generalized to,
and yields good methods for, other types of evolution equations. Simple
generalization of the linear advection equation are the level set reinitialization
equation, or the G-equation in combustion modeling \cite{OberlackWenzelPeters2001}.
More complex examples are problems involving diffusion, up to the actual equations
of fluid flow. A consistently coupled gradient-augmented scheme for both the two-phase
Navier-Stokes equations and the level set equation for the interface could
not only allow the representation of subgrid structures, but also yield a
certain level of subgrid resolution in the actual fluid flow simulation.

\section*{Acknowledgments}
The authors would like to acknowledge the support by the
National Science Foundation. This research was supported by grant DMS--0813648.
Rodolfo Ruben Rosales would like to acknowledge partial support by the
Universidad Carlos III de Madrid, Spain.

\bibliographystyle{plain}
\bibliography{references_complete}

\end{document}